\documentclass[11pt, twoside, leqno]{article}

\usepackage{amssymb}
\usepackage{amsfonts}
\usepackage{amsmath}
\usepackage{amsthm}
\usepackage{color}
\usepackage{mathrsfs}
\usepackage{txfonts}
\usepackage{bbm}

\usepackage{indentfirst}

\allowdisplaybreaks

\def\rr{{\mathbb{R}}}
\def\cc{{\mathbb{C}}}
\def\rn{{\mathbb{R}^n}}
\def\zz{{\mathbb Z}}
\def\nn{{\mathbb N}}
\def\cx{{\mathbb{X}}}
\def\cm{{\mathcal M}}
\def\cg{{\mathcal G}}
\def\ca{{\mathcal A}}

\DeclareMathOperator{\diam}{diam}

\def\rmf{ M^+(f)}
\def\ntm{M_{\theta}(f)}

\def\fz{\infty }
\def\az{\alpha}
\def\bz{\beta}
\def\dz{\delta}
\def\lz{\lambda}
\def\oz{\omega}
\def\ez{\epsilon}
\def\gz{\gamma}
\def\tz{\theta}
\def\Oz{\Omega}

\def\lf{\left}
\def\r{\right}
\def\ls{\lesssim}

\def\supp{\mathop\mathrm{\,supp\,}}
\def\gs{\mathcal{G}^{\eta}_0(\bz,\gz)}

\def\fs{f^{\star}}

\newtheorem{theorem}{Theorem}[section]
\newtheorem{lemma}[theorem]{Lemma}

\newtheorem{proposition}[theorem]{Proposition}

\theoremstyle{definition}
\newtheorem{remark}[theorem]{Remark}
\newtheorem{definition}[theorem]{Definition}
\newtheorem{assumption}[theorem]{Assumption}
\renewcommand{\appendix}{\par
   \setcounter{section}{0}%
   \setcounter{subsection}{0}%
   \setcounter{subsubsection}{0}%
   \gdef\thesection{\@Alph\c@section}%
   \gdef\thesubsection{\@Alph\c@section.\@arabic\c@subsection}%
   \gdef\theHsection{\@Alph\c@section.}%
   \gdef\theHsubsection{\@Alph\c@section.\@arabic\c@subsection}%
   \csname appendixmore\endcsname
}

\numberwithin{equation}{section}

\begin{document}
	
\arraycolsep=1pt

\title{\bf\Large Weak Hardy Spaces Associated
with Ball Quasi-Banach Function Spaces on Spaces
of Homogeneous Type:
Decompositions, Real Interpolation,
and Calder\'{o}n--Zygmund Operators
\footnotetext{\hspace{-0.35cm} 2020
{\it Mathematics Subject Classification}. Primary 42B30;
Secondary 42B25, 42B20, 42B35, 46E36, 30L99.\endgraf
{\it Key words and phrases.} space of homogeneous type,
ball quasi-Banach function space,
weak Hardy space, maximal function, atom, real interpolation,
Calder\'{o}n--Zygmund operator. \endgraf
This project is partially supported by the National Key Research
and Development Program of China (Grant No.
2020YFA0712900) and the National
Natural Science Foundation of China (Grant Nos. 11971058,
12071197, 12122102 and 11871100).}}
\author{Jingsong Sun, Dachun Yang\footnote{Corresponding author,
E-mail: \texttt{dcyang@bnu.edu.cn}/{\color{red}January 23, 2022}/Final version.}
\ and Wen Yuan}
\date{}
\maketitle
\vspace{-0.8cm}

\begin{center}
\begin{minipage}{13cm}
{\small {\bf Abstract}\quad
Let $(\mathbb{X},d,\mu)$ be a space of homogeneous type
in the sense of R. R. Coifman and
G. Weiss, and $X(\mathbb{X})$ a ball quasi-Banach
function space on $\mathbb{X}$. In this article,
the authors introduce the weak Hardy space
$WH_X(\mathbb{X})$ associated with $X(\cx)$ via the
grand maximal function, and characterize $WH_X(\mathbb{X})$
by other maximal functions and atoms.
The authors then apply these characterizations to obtain
the real interpolation and the boundedness of
Calder\'{o}n--Zygmund operators in the critical case.
The main novelties of this article exist in that
the authors use the Aoki--Rolewicz theorem and both the dyadic system
and the exponential decay of approximations of the
identity on $\cx$, which closely connect with the geometrical
properties of $\cx$, to overcome the difficulties caused by
the absence of both the triangle inequality of $\|\cdot\|_{X(\mathbb{X})}$ and the
reverse doubling assumption of the measure $\mu$ under consideration,
and also use the relation
between the convexification of $X(\mathbb{X})$ and the
weak space $WX(\mathbb{X})$ associated with
$X(\mathbb{X})$ to prove that the infinite summation
of atoms converges in the space of distributions on $\mathbb{X}$.
Moreover, all these results have a wide range of generality and, particularly,
even when they are applied to the weighted Lebesgue space, the Orlicz space,
and the variable Lebesgue space, the obtained results are also new
and, actually, some of them are new even on RD-spaces
(namely, spaces of homogeneous type satisfying
the additional reverse doubling condition).
}
\end{minipage}
\end{center}

\vspace{0.2cm}

\tableofcontents

\vspace{0.2cm}

\section{Introduction}
Hardy spaces
play a vital role in harmonic
analysis and partial differential equations.
The classical Hardy space was originally introduced by
Stein and Weiss \cite{sw60} and developed by
Fefferman and Stein \cite{fs72}. Later,
Fefferman and Soria \cite{fs86} introduced
the weak Hardy space $WH^1(\rn)$
which proves the largest space $Z$ such that
the Hilbert transform is bounded from $Z$
to the weak Lebesgue space $WL^1(\rn)$
and, in the same article, they
also obtained the boundedness of some
Calder\'{o}n--Zygmund operators from
$WH^1(\rn)$ to $WL^1(\rn)$.
Furthermore, Fefferman et al.
\cite{frs74} found that the weak Hardy
space $WH^p(\rn)$ is the intermediate space of
the real interpolation between
the Hardy space $H^p(\rn)$ and the Lebesgue space $L^\fz(\rn)$;
Liu \cite{l91} obtained the boundedness of any $\dz$-type
Calder\'{o}n--Zygmund operator $T$ from
the Hardy space $H^{p_\dz}(\rn)$ to
the weak Hardy space $WH^{p_\dz}(\rn)$,
where $\dz\in(0,1]$ and
$p_{\dz}:=n/(n+\dz)$. Recall that $p_\dz$ is called the
\emph{critical case} or the \emph{endpoint case}
because the boundedness of $T$ succeeds on
$H^p(\rn)$ with $p\in(p_\dz,1]$ but fails on
$H^{p_\dz}(\rn)$; see, for instance, \cite{am86}.
From then on, (weak) Hardy spaces
have attracted a lot of attention; see, for instance,
\cite{st89,qy00,h17b} for the weighted (weak) Hardy space,
\cite{j80,jy10,ns14} for the (weak) Orlicz--Hardy space,
\cite{ns12,s13,yyyz16} for the variable (weak)
Hardy space, \cite{s06,jw09,h13,h15,h17a} for
the (weak) Hardy--Morrey space, and \cite{k14,lhy12,lyj16,ylk17}
for the (weak) Musielak--Orlicz
Hardy space.

Recently, Sawano et al. \cite{shyy17} introduced
the ball quasi-Banach function space $X$ and investigated
the Hardy space $H_X(\rn)$ associated with $X$.
In fact, the Hardy space $H_X(\rn)$ contains
many well-known Hardy-type spaces on $\rn$ such as the classical
Hardy space, the weighted Hardy space, the Orlicz--Hardy
space, the variable Hardy space, and the
Musielak--Orlicz Hardy space, which makes the
results obtained in \cite{shyy17} have
wide applications. Later, Zhang et al. \cite{zyyw20}
and Wang et al. \cite{wyyz21} introduced
the weak Hardy-type space associated with the ball
quasi-Banach function space on $\rn$ and obtained various maximal
function, atomic, molecular, and various Littlewood--Paley
function characterizations
and their applications to both the real interpolation and the
boundedness of Calder\'{o}n--Zygmund operators.
Actually, more and more articles related to the ball quasi-Banach
function space $X$ on $\rn$ spring up;
see, for instance, \cite{cwyz20,wyy20,yyy20}.

On the other hand, recall that Coifman and Weiss \cite{cw71,cw77}
introduced the concept of the space of homogeneous type, $(\cx,d,\mu)$,
which is a general underlying space
covering the Euclidean space $\rn$ and many other known
underlying spaces and, under this general framework,
they studied the boundedness of Calder\'{o}n--Zygmund
operators and introduced atomic
Hardy spaces. Later, there was a
surge of interest in the analysis on $\cx$;
see, for instance,
\cite{ms79a,ms79b,dy03,hyz09,l10,l13}.
However, when generalizing some results
from $\rn$ to $\cx$, the predecessors added
some essential additional assumptions on $\cx$
and hence only investigated
some special cases of $\cx$. Among them,
an important case is the RD-space, which is a space of
homogeneous type, $(\cx,d,\mu)$,
with $\mu$ satisfying the additional reverse
doubling condition; we refer the reader
to \cite{hmy06,gly08,hmy08,my09,dw10,
kyz10,kyz11,yz11,ww12,zsy16} for the real-variable
theory of various function spaces on RD-spaces,
and its applications.

In what follows, we always use $(\cx,d,\mu)$ or $\cx$
to denote a given space
of homogeneous type in the sense of Coifman and Weiss.
A recent inspiring breakthrough about
the analysis on $\cx$ comes from Auscher and
Hyt\"{o}nen \cite{ah13}. To be precise, using a
randomized dyadic structure, Auscher and
Hyt\"{o}nen \cite{ah13} constructed
an orthonormal wavelet basis, with the H\"{o}lder regularity,
of $L^2(\cx)$. This outcome boosts
the rapid development of the real-variable
theory of function spaces on $\cx$.
For instance, Han et al. \cite{hlw18}
introduced Hardy spaces over $\cx$ via the
wavelet bases constructed by Auscher and Hyt\"{o}nen;
Bui et al. \cite{bdl18} obtained
various maximal function characterizations
of local Hardy spaces associated
with operators over $\cx$;
He et al. \cite{hlyy19} introduced
approximations of the identity with
exponential decay and constructed
Calder\'{o}n reproducing formulae on $\cx$;
He et al. \cite{hhllyy,hyy19} established a complete
real-variable theory of (local) Hardy spaces on $\cx$;
Zhou et al. \cite{zhy20} studied Hardy--Lorentz spaces on $\cx$,
which contain weak Hardy spaces on $\cx$.
Furthermore, we refer the reader to
\cite{bdl18,bdl20} for various maximal function
characterizations of Hardy spaces on $\cx$,
to \cite{llw16,dhl19,bdk20} for other characterizations
of Hardy-type spaces on $\cx$, to \cite{n17,bd21,dgklwy}
for the boundedness of operators on function spaces over $\cx$,
and to \cite{whhy21,hwyy21,whyy21} for Besov and Triebel--Lizorkin
spaces on $\cx$.

Very recently, Yan et al. \cite{yhyy21a,yhyy21b}
introduced the Hardy space associated with the
ball quasi-Banach function space on $\cx$,
and established its various maximal function, atomic, molecular,
and various Littlewood--Paley function
characterizations to obtain both the boundedness of
Calder\'{o}n--Zygmund operators on it, and its dual space.
In order to have a deeper understanding of
Hardy spaces and put the weak Hardy spaces
into a general framework, in this article, we introduce
weak Hardy spaces associated with ball
quasi-Banach function spaces on $\cx$.
Moreover, we use various maximal function
and atomic characterizations to obtain both the real
interpolation and the boundedness
of Calder\'{o}n--Zygmund operators
in the critical case.
The Littlewood--Paley function characterizations
of these weak Hardy spaces on $\cx$ will be investigated
in a forthcoming article \cite{syy}.

To be precise, let $X(\cx)$ be a ball quasi-Banach
function space on $\cx$. We first introduce the weak Hardy space
$WH_X(\cx)$ associated with $X(\cx)$ via the
grand maximal function, and characterize $WH_X(\cx)$
by the radial maximal function, the non-tangential
maximal function, and the atom.
We then apply these characterizations to obtain both the
real interpolation and the boundedness of
Calder\'{o}n--Zygmund operators in the critical case.
Furthermore, the atomic decomposition of $WH_X(\cx)$
is also used to establish various Littlewood--Paley function
characterizations of $WH_X(\cx)$ in a forthcoming article \cite{syy}.
The main novelties of this article exist in that
we use the Aoki--Rolewicz theorem and both the dyadic system
and the exponential decay of approximations of the
identity on $\cx$, which closely connect with the geometrical
properties of $\cx$, to overcome the difficulties caused by
the absence of both the triangle inequality of $\|\cdot\|_{X(\mathbb{X})}$ and the
reverse doubling assumption of the measure $\mu$ under consideration,
and also use the relation
between the convexification of $X(\mathbb{X})$ and the
weak space $WX(\mathbb{X})$ associated with
$X(\mathbb{X})$ to prove that the infinite summation
of atoms converges in the space of distributions on $\mathbb{X}$.
Moreover, all these results have a wide range of generality and, particularly,
even when they are applied to the weighted Lebesgue space, the Orlicz space,
and the variable Lebesgue space, the obtained results are also new
and, actually, some of them are new even on RD-spaces.

The remainder of this article is organized as follows.

For the simplicity of the representation,
throughout this article,
we always use the symbol $X(\cx)$ to denote
a ball quasi-Banach function space on $\cx$,
and the symbol $WX(\cx)$ to denote the weak ball
quasi-Banach function space
associated with $X(\cx)$.

Section \ref{sp} is divided into two subsections
to give some  basic concepts and properties.
In Subsection \ref{ssht}, we recall the concept
of spaces of homogeneous type and
mainly give a basic property about the representation
of the action of an approximation of the identity
with exponential decay and integration 1 on a
distribution (see Proposition \ref{prdc} below).
Proposition \ref{prdc} is important and was
implicitly used in the literature,
but its proof is nontrivial and was always neglected.
In Subsection \ref{ssbb}, we introduce the weak
ball quasi-Banach function space. Besides, we establish
a modified version of the Aoki--Rolewicz theorem
for ball quasi-Banach function spaces
in Proposition \ref{prtz} below,
and some upper bound estimates about
the truncation of functions in $WX^p(\cx)$
[the $p$-convexification of $WX(\cx)$]
in Proposition \ref{prin} below.

In Section \ref{sm}, we introduce the weak Hardy space
$WH_X(\cx)$ associated with $X(\cx)$ via the grand
maximal function; see Definition \ref{de} below.
Assuming that the Hardy--Littlewood
maximal operator is bounded on the convexification of
$WX(\cx)$, we find that various
maximal functions are equivalent in terms of the
quasi-norm of $WX(\cx)$, and that
$WH_X(\cx)$ is independent of the choice of the
spaces of distributions under consideration
(see Theorem \ref{thm} below);
we show that $WH_X(\cx)=WX(\cx)$ with equivalent quasi-norms
(see Theorem \ref{thmw} below).

In Section \ref{sa}, we  first propose
two crucial assumptions (Assumptions
\ref{asfs} and \ref{asas} below).
These assumptions relate to the
boundedness of the Hardy--Littlewood maximal operator
on the convexification of $X(\cx)$, and its associate space.
If Assumption \ref{asfs} holds true,
by an interpolation result (see Lemma \ref{lewi}), we prove that
the Fefferman--Stein vector-valued maximal
inequality also holds true on $WX(\cx)$; see Theorem \ref{thwfs} below.
Then we deduce a key Proposition \ref{prfs} (resp.,
Proposition \ref{pras}) from Assumption \ref{asfs}
(resp., Assumption \ref{asas}). These assumptions and the
results induced by them are repeatedly used in
later Sections \ref{sat}, \ref{sr}, and \ref{sc}.

Section \ref{sat} is divided into the reconstruction
and the decomposition subsections
to establish the atomic characterization.
If $X(\cx)$ satisfies Assumptions \ref{asfs} and \ref{asas},
via borrowing some ideas from the
proof of \cite[Theorem 4.7]{zyyw20},
we can reconstruct elements in $WH_X(\cx)$
by infinite sums of multiples of atoms;
see Theorem \ref{tham} below.
The novelty of the proof of Theorem \ref{tham}
is that we  use the Aoki--Rolewicz theorem
to replace the assumption that the convexification
of $X(\cx)$ is  a Banach space and we
consider the index $q\in (p_0,\fz]\cap[1,\fz]$
rather than $q\in (p_0,\fz]\cap(1,\fz]$.
By means of the Calder\'{o}n--Zygmund decomposition,
we show that, if $X(\cx)$ satisfies Assumption
\ref{asfs}, we can decompose any
element in $WH_X(\cx)$ into an infinite sum of
multiples of atoms; see Theorem \ref{thma} below.
The novelty appearing in the proof of Theorem \ref{thma}
is that we use the relation between $WX(\cx)$ and
the convexification of $X(\cx)$ to prove the
convergence of the infinite summation of atoms.

We obtain some real interpolation results in Section \ref{sr}.
Borrowing some ideas from the proof of \cite[Theorem 4.1]{kv14},
we establish a real interpolation theorem between $X(\cx)$
and the Lebesgue space $L^\fz(\cx)$; see Theorem
\ref{thin} below. By Theorem \ref{thin} and
the atomic decomposition in Theorem \ref{thma},
we conclude that the real
interpolation intermediate space between the
Hardy space, associated with $X(\cx)$,
and the Lebesgue space $L^\fz(\cx)$
is the weak Hardy space associated with
$X(\cx)$; see Theorem \ref{thhin} below.

The main purpose of Section \ref{sc} is
to establish the boundedness of Calder\'{o}n--Zygmund
operators in the critical case. If $X(\cx)$ has an
absolutely continuous quasi-norm and satisfies
Assumptions \ref{asfs}, \ref{asas}, and \ref{asw}, using some
estimates from the proof of \cite[Theorem 8.13]{zhy20},
we show that a given Calder\'{o}n--Zygmund
operator has the unique bounded extension
in the critical case; see Theorem \ref{thcz} below.
If $X(\cx)$ does not have an absolutely continuous quasi-norm
but embeds into a weighted Lebesgue space,
we use this embedding to extend a given Calder\'{o}n--Zygmund operator to
a bounded linear operator in the critical case;
see Theorem \ref{thc} below.

We apply the main results obtained in the
above sections, respectively, to some specific ball
quasi-Banach function spaces in Section \ref{sap}.
We discuss Lebesgue spaces,
 weighted Lebesgue spaces, Orlicz spaces,
and variable Lebesgue spaces in turn.

At the end of this section, we make some conventions
on notation. Throughout this article,
the symbol $(\cx,d,\mu)$ or $\cx$ always denotes a
\emph{space of homogeneous type}.
Let $A_0$ be as in \eqref{eqA0},
$\oz$ the \emph{upper dimension} as in \eqref{eqoz},
$\dz_0$ as in Lemma \ref{ledy}, and $\eta$ the
\emph{H\"{o}lder regularity exponent}
as in \eqref{eqwa}. For any $r\in (0,\fz)$,
let $\dz_0^r:=(\dz_0)^r$
and $A^r_0:=(A_0)^r$.
For any $r\in (0,\fz)$ and $x,\,y\in\cx$,
let
$$
V_r(x):=\mu(B(x,r))
\quad \mathrm{and}\quad
V(x,y):=\mu(B(x,d(x,y))).
$$
For any given $r$, $\gz\in (0,\fz)$ and any
$x$, $y\in\cx$, let
\begin{equation}\label{eqg}
G_{r,\gz}(x,y):=
\frac{1}{V_r(x)+V(x,y)}
\lf[\frac{r}{r+d(x,y)}\r]^{\gz}.
\end{equation}
For any subset $E\subset\cx$, the symbol
$\mathbf{1}_E$ denotes its \emph{characteristic function},
and the symbol $E^\complement$ the complementary
set of $E$ in $\cx$.
For any $x\in\cx$ and $\Oz\subset\cx$, let
$d(x,\Oz):=\inf_{z\in\Oz}\{d(x,z)\}$.
The symbol $C$ always denotes a
positive constant independent
of the main parameters involved,
but may vary from line to line.
If $f\le Cg$,
we then write $f\ls g$ or $g\gtrsim f$;
if $f\ls g\ls f$, we then write $f\sim g$.
Moreover, if $f\ls g$ and $g=h$, or
$f\ls g$ and $g\le h$,
we then write $f\ls g\sim h$ or $f\ls g\ls h$,
\emph{rather than} $f\ls g=h$ or $f\ls g\le h$.
Let $\nn:=\{1,2,3,\ldots\}$, $\zz_+:=\nn\cup\{0\}$,
and $\zz$ be the set of all integers.
For any $p\in[1,\fz]$, the symbol $p'$ denotes
the \emph{conjugate index} of $p$, namely, $1/p'+1/p=1$.
For any set $E$, the symbol $\# E$ denotes its
\emph{cardinality}.

\section{Preliminaries\label{sp}}
This section is devoted to
some common symbols and basic concepts.
For this purpose, we divide this section into two subsections:
in Subsection \ref{ssht}, we recall the definition
of spaces of homogeneous type; in Subsection \ref{ssbb},
we introduce the weak ball quasi-Banach function space.

\subsection{Spaces of Homogeneous Type\label{ssht}}
For a non-empty set $\cx$,
$d$ is called a \emph{quasi-metric} on $\cx$
if $d$ is a non-negative function
on $\cx\times\cx$ such that,
for any $x$, $y$, $z\in\cx$,
\begin{enumerate}
\item[(i)] $d(x,y)=0$ if and only if $x=y$;
\item[(ii)] $d(x,y)=d(y,x)$;
\item[(iii)] there exists a constant $A_0\in [1,\fz)$,
independent of $x$, $y$, and $z$,
such that
\begin{equation}\label{eqA0}
d(x,z)\le A_0\lf[d(x,y)+d(y,z)\r].
\end{equation}
\end{enumerate}
Moreover, if $A_0=1$ in \eqref{eqA0},
then $d$ is called a \emph{metric} on $\cx$.

If $d$ is a quasi-metric on  a non-empty set $\cx$,
then $(\cx,d)$ is called a \emph{quasi-metric space}.
Let
$$
\diam\cx:=\sup_{x,\, y\in\cx}d(x,y).
$$
For any $x\in\cx$ and $\lz$, $r\in (0,\fz)$,
the set $\lz B:=\lz B(x,r):=\{z\in\cx:\ d(x,z)<\lz r\}$
is called a \emph{ball} of $\cx$
with the \emph{center} $x$ and the \emph{radius} $\lz r$.
A set $\Oz\subset\cx$ is said to be
\emph{open} if, for any $x\in\Oz$,
there exists a ball $B$ such that
$x\in B\subset \Oz$.
The \emph{Borel $\sigma$-algebra} of $\cx$
is the smallest $\sigma$-algebra that contains
all the open sets of $\cx$, and any set in
the Borel $\sigma$-algebra is called a
\emph{Borel set} of $\cx$.

Let $(\cx,d)$ be a quasi-metric space and
$\mu$ a non-negative measure on $\cx$.
A triple $(\cx,d,\mu)$ is called a
\emph{space of homogeneous type}
if $\mu$ satisfies the \emph{doubling condition}:
there exists a constant $C\in [1,\fz)$ such that,
for any ball $B\subset\cx$,
\begin{equation}\label{eqcu}
\mu(2B)\leq C\mu(B).
\end{equation}
Let $C_{(\mu)}$ denote the infimum of all the
positive constants $C$ satisfying \eqref{eqcu}.
This implies that,
for any ball $B\subset\cx$ and $\lz\in [1,\fz)$,
\begin{equation}\label{eqoz}
\mu(\lz B)\leq C_{(\mu)}\lz^{\oz}\mu(B),
\end{equation}
where $\oz:=\log_2C_{(\mu)}$ is called the
\emph{upper dimension} of $ (\cx,d,\mu)$.

Now, we recall the definition of the RD-space
which is a special case of the space of homogeneous type.
A triple $(\cx,d,\mu)$ is called an RD-\emph{space}
if $d$ is a metric on the non-empty set $\cx$,
and $\mu$ is a measure satisfying both
the doubling condition \eqref{eqcu}
and the \emph{reverse doubling condition}
that there exist constants $C\in(0,1]$
and $\kappa\in(0,\oz]$ such that,
for any $x\in\cx$, $r\in(0,\frac{\diam\cx}{2})$,
and $\lz\in[1,\frac{\diam\cx}{2r})$,
$$
C\lz^\kappa\mu(B(x,r))\le \mu(B(x,\lz r)).
$$

Throughout this article, let $(\cx,d,\mu)$
be a space of homogeneous type satisfying
the following mild assumptions:
\begin{enumerate}
\item[(i)]
for any $x\in\cx$, $\{B(x,r)\}_{r\in (0,\fz)}$
form a basis of open neighborhoods of $x$;
\item[(ii)]
$\mu$ is \emph{Borel regular}
which means that all open sets are $\mu$-measurable
and every set $A\subset\cx$ is
contained in a Borel set $E$ such that $\mu(A)=\mu(E)$;
\item[(iii)]
for any $x\in\cx$ and $r\in (0,\fz)$,
the ball $B(x,r)$ is $\mu$-measurable,
$\mu(B(x,r))\in (0,\fz)$, and $\mu(\{x\})=0$;
\item[(iv)]
$\diam\cx=\fz$.
\end{enumerate}
From \cite[Lemma 8.1]{ah13} (or \cite[Lemma 5.1]{ny97})
and (iv), it follows that
$\mu(\cx)=\fz$.
Here and thereafter, unless necessary,
we always use $\cx$ to denote $(\cx,d,\mu)$.

Now, we recall some symbols on $\cx$ needed in this article.
The symbol $\mathscr{M}(\cx)$ denotes
the set of all $\mu$-measurable functions on $\cx$.
For any $\mu$-measurable set $E\subset\cx$,
the \emph{Lebesgue space} $L^p(E)$ with
$p\in(0,\fz)$ is defined by setting
$$
L^p(E):=\lf\{f\in\mathscr{M}(\cx):\,\|f\|_{L^p(E)}:
=\lf[\int_E|f(z)|^p\,d\mu(z)\r]^{1/p}<\fz\r\}
$$
and the \emph{Lebesgue space}
$L^{\fz}(E)$ is defined by setting
$L^{\fz}(E):=\{f\in\mathscr{M}(\cx):
\,\|f\|_{L^\fz(E)}<\fz\}$,
where $\|f\|_{L^\fz(E)}$ denotes
the essential supremum of $f$ on $E$.
For any $p\in(0,\fz)$ and any non-negative $\mu$-measurable function
$w$ on $\cx$, the \emph{weighted Lebesgue space}
$L^p_w(\cx)$ is defined by setting
\begin{equation}\label{eqweight}
L^p_w(\cx):=\lf\{f\in\mathscr{M}(\cx):\,\|f\|_{L^p_w(\cx)}:
=\lf[\int_\cx|f(z)|^pw(z)\,d\mu(z)\r]^{1/p}<\fz\r\}.
\end{equation}
The \emph{Hardy--Littlewood maximal operator} $\cm$
is defined by setting, for any $f\in \mathscr{M}(\cx)$
and $x\in\cx$,
\begin{equation*}
\cm(f)(x):=\sup_{B\ni x}\lf\{\frac{1}{\mu(B)}
\int_B|f(z)|\,d\mu(z)\r\},
\end{equation*}
where the supremum is taken over all the balls
$B$ containing $x$.

The following lemma shows some estimates
widely used in this article; we omit its proof
here and refer the reader to \cite[Lemma 2.1]{hmy06}
for the details.
\begin{lemma}\label{lees}
\begin{enumerate}
\item[\textup{(i)}]
Let $x$, $y\in\cx$ and $r \in (0,\fz)$.
Then $V(x, y)\sim V (y, x)$ and
$$
V_r(x)+V_r(y)+V(x,y)\sim V_r(x)+V (x,y)
\sim V_r(y)+V(x,y)\sim \mu(B(x,r+d(x,y))).
$$
If $d(x,y)\leq r$, then $V_r(x)\sim V_r(y)$.
Here, all the positive
equivalence constants are
independent of $x$, $y$, and $r$.
\item[\textup{(ii)}]
Let $\gz\in(0,\fz)$. Then
there exists a positive constant $C$ such that,
for any $r\in(0,\fz)$ and $x\in\cx$,
$$
\int_\cx G_{r,\gz}(x,y)\,d\mu(y)\le C
$$
with $G_{r,\gz}$ as in \eqref{eqg}.
\item[\textup{(iii)}]
Let $\oz$ be as in \eqref{eqoz}.
Then there exists a positive constant $C$ such that,
for any ball $B:=B(x,r)$, with $x\in\cx$ and
$r\in(0,\fz)$, and
any $y\in B^{\complement}$,
$$
\lf[\frac{r}{d(x,y)}\r]^\oz\le C\cm(\mathbf{1}_B)(y)
\quad\mathrm{and}\quad
\frac{\mu(B)}{V(x,y)}\le C\cm(\mathbf{1}_B)(y).
$$
\item[\textup{(iv)}]
Let $s\in (0,1]$. Then, for any sequence
$\{\lz_j\}^\fz_{j=1}\subset [0,\fz)$,
$$
\lf(\sum^\fz_{j=1}\lz_j\r)^s
\le \sum^\fz_{j=1}\lf(\lz_j\r)^s.
$$
\end{enumerate}
\end{lemma}

Now, we present a dyadic structure on $\cx$
in the following lemma
which is a part of \cite[Theorem 2.2]{hk12}.

\begin{lemma}\label{ledy}
Fix constants $0<c_0\le C_0<\fz$ and
$\delta_0\in (0,1)$ such that
$12A_0^3C_0\dz_0\leq c_0$. Suppose that,
for any $k\in\zz$,
$\ca_k$ is a countable set of indices, and
$\mathcal{Z}_k:=\{z^k_\az\}_{\az\in\ca_k}\subset\cx$
has the following properties:
\begin{enumerate}
\item[\textup{(i)}]
for any $\az,\ \bz\in\ca_k$ with $\az\neq \bz$,
$d(z^k_{\az},z^k_{\bz})\geq c_0\dz_0^k$;
\item[\textup{(ii)}]
for any $x\in\cx$,
$\min_{\az\in\ca_k}d(x,z^k_\az)\le C_0\dz_0^k$.
\end{enumerate}
Then there exists a system of dyadic cubes,
$\{Q^k_{\az}\}_{k\in\zz,\ \az\in\ca_k}$,
such that
\begin{enumerate}
\item[\textup{(iii)}]
for any $k\in\zz$, $\bigcup_{\az\in\ca_k}Q^k_{\az}=\cx$
and $\{Q^k_{\az}\}_{\az\in\ca_k}$
is a set of pairwise disjoint cubes;
\item[\textup{(iv)}]
if $k,\ l\in\zz$ with $k\le l$, $\az\in\ca_k$,
and $\bz\in\ca_l$, then either
$Q^l_{\bz}\subset Q^k_{\az}$ or
$Q^l_{\bz}\cap Q^k_{\az}=\emptyset$;
\item[\textup{(v)}]
for any $k\in\zz$ and $\az\in\ca_k$,
$$
B\lf(z^k_{\az},c_s\dz_0^k\r)\subset
Q^k_{\az}\subset B\lf(z^k_{\az},C_L\dz_0^k\r),
$$
where $c_s:=(3A^2_0)^{-1}c_0$ and $C_L:=2A_0C_0$.
The point $z^k_{\az}$ is called
the ``center" of $Q^k_{\az}$.
\end{enumerate}
\end{lemma}

In what follows,
we keep using the same symbols of Lemma \ref{ledy}
and, for any $k\in\zz$,
we assume that $\ca_k\subset \ca_{k+1}$,
$\mathcal{Z}_k\subset \mathcal{Z}_{k+1}$,
and
$\mathcal{Y}_k:=\mathcal{Z}_{k+1}
\setminus \mathcal{Z}_k
=:\{y^k_{\az}\}_{\az\in\ca_{k+1}\setminus \ca_k}$.
Via an adapted randomized dyadic structure,
Auscher and Hyt{\"o}nen \cite[Theorem 7.1]{ah13}
constructed an orthonormal wavelet basis
$\{\psi^k_\az\}_{k\in \zz,\ \az\in\ca_{k+1}\setminus
\ca_k}$, with the H\"{o}lder regularity,
of $L^2(\cx)$; more precisely, there exist
an $\eta\in (0,1]$, called the \emph{
H\"{o}lder regularity exponent}, and
positive constants $C$ and $v$ such that, for
any $k\in\zz$, $\az\in \ca_{k+1}\setminus\ca_k$,
and $x$, $y\in\cx$ with $d(x,y)\leq \dz^k_0$,
\begin{equation}\label{eqwa}
\lf|\psi^k_{\az}(x)-\psi^k_{\az}(y)\r|
\le\frac{C}{\sqrt{\mu(B(y^k_{\az},\dz^k_0))}}
\lf[\frac{d(x,y)}{\dz^k_0}\r]^{\eta}
\exp \lf\{-v\left[
\frac{d(y^k_{\az},x)}{\dz^k_0}\right]^a\r\},
\end{equation}
where $a:=(1+2\log_2A_0)^{-1}$ with
$A_0$ as in \eqref{eqA0}.
Motivated by this and other main results
in \cite{ah13}, He et al. \cite[Definition 2.7]{hlyy19}
and \cite[Definition 2.8]{hhllyy} introduced
two approximations of the identity with
exponential decay, which are restated
in the following definitions.

\begin{definition}\label{dee0}
Let $\dz_0$ be as in Lemma \ref{ledy}, and
$\eta$ as in \eqref{eqwa}.
A sequence $\{E_k\}_{k\in\zz}$
 of bounded linear integral operators on
$L^2(\cx)$ is called an \emph{approximation of
the identity with exponential decay}
(for short, exp-ATI) if there exist
constants $C$, $v\in (0,\fz)$ and
$a \in (0,1]$ such that,
for any $k\in\zz$, the kernel of the operator $E_k$,
which is still denoted by $E_k$, satisfying
\begin{enumerate}
\item[(i)] (the \emph{identity condition})
$\sum_{k=-\fz}^{\fz}E_k=I$ in $L^2(\cx)$,
where $I$ is the identity operator on $L^2(\cx)$;
\item[(ii)] (the \emph{size condition})
for any $x$, $y\in\cx$,
\begin{align*}
\lf|E_k(x,y)\r|
&\le C\frac{1}{\sqrt{V_{\dz_0^k}(x)V_{\dz_0^k}(y)}}
\exp\lf\{-v\lf[\frac{d(x,y)}{\dz_0^k}\r]^a\r\}\\
&\qquad\times
\exp\lf\{-v\lf[\frac{\max\{d(x,\mathcal{Y}_k),
d(y,\mathcal{Y}_k)\}}{\dz_0^k}\r]^a\r\}\\
&=:C\mathcal{E}_k(x,y);
\end{align*}
\item[(iii)] (the \emph{regularity condition})
for any $x$, $x'$, $y\in\cx$
with $d(x,x')\leq \dz_0^k$,
\begin{align*}
\lf|E_k(x,y)-E_k(x',y)\r|+\lf|E_k(y,x)-E_k(y,x')\r|
\le C\lf[\frac{d(x,x')}{\dz_0^k}\r]^{\eta}
\mathcal{E}_k(x,y);
\end{align*}
\item[(iv)]
(the \emph{second difference regularity condition})
for any $x$, $x'$, $y$, $y'\in\cx$
with $d(x,x')\le \dz_0^k$ and $d(y,y')\leq \dz_0^k$,
\begin{align*}
&\lf|\left[E_k(x,y)-E_k(x',y)\r]-\lf[E_k(x,y')
-E_k(x',y')\r]\r|\\
&\quad
\le C\lf[\frac{d(x,x')}{\dz_0^k}\r]
^{\eta}\lf[\frac{d(y,y')}{\dz_0^k}\r]^{\eta}
\mathcal{E}_k(x,y);
\end{align*}
\item[(v)] (the \emph{cancellation condition})
for any $x\in\cx$,
$\int_{\cx}E_k(x,z)\,d\mu(z)=0
=\int_{\cx}E_k(z,x)\,d\mu(z).$
\end{enumerate}
\end{definition}
As is mentioned above, the existence of such
an exp-ATI on $\cx$ is guaranteed by \cite[Theorem 7.1]{ah13}
with $\eta$ as in \cite[Theorem 3.1]{ah13} which
might be very small (see also \cite[Remark 2.8(i)]{hlyy19}).
However, if $d$ is a metric, then $\eta$ can
be taken arbitrarily close to 1
(see \cite[Corollary 6.13]{ht14}).

\begin{definition}\label{dee1}
A sequence $\{P_k\}_{k\in\zz}$ of bounded
linear integral operators on $L^2(\cx)$ is called an
\emph{approximation of the identity with exponential
decay and integration 1} (for short, 1-exp-ATI)
if $\{P_k\}_{k\in\zz}$ has the following properties:
\begin{enumerate}
\item[(i)] for any $k\in\zz$,
$P_k$ satisfies (ii) through (iv)
of Definition \ref{dee0} but without the factor
\begin{equation*}
\exp\lf\{-v\lf[\frac{
\max\{d(x,\mathcal{Y}_k),d(y,\mathcal{Y}_k)\}}
{\dz_0^k}\r]^a\r\};
\end{equation*}
\item[(ii)] for any $k\in\zz$ and $x\in\cx$,
$\int_{\cx}P_k(x,z)\,d\mu(z)=1
=\int_{\cx}P_k(z,x)\,d\mu(z);$
\item[(iii)]
$\{P_k-P_{k-1}\}_{k\in\zz}$ is an exp-ATI
as in Definition \ref{dee0}.
\end{enumerate}
\end{definition}

\begin{remark}\label{rede1}
Let $\gz\in(0,\fz)$, $\dz_0$ be as in Lemma \ref{ledy},
$\eta$ as in \eqref{eqwa}, and
$\{P_k\}_{k\in\zz}$ a 1-exp-ATI as
in Definition \ref{dee1}.
By some arguments similar to those used
in the proof of \cite[Proposition 2.10]{hlyy19},
we find that
there exists a positive constant $C$ such that
\begin{itemize}
\item[(i)]
for any $k\in\zz$ and $x$, $y\in\cx$,
$|P_k(x,y)|\le CG_{\dz_0^k,\gz}(x,y)$,
here and thereafter,
$G_{\dz_0^k,\gz}$ is as in \eqref{eqg}
with $r$ replaced by $\dz_0^k$;
\item[(ii)]
for any $k\in\zz$ and $x$, $x'$, $y\in\cx$
with $d(x,x')\le (2A_0)^{-1}[\dz_0^k+d(x,y)]$,
\begin{align*}
\lf|P_k(x,y)-P_k(x',y)\r|+\lf|P_k(y,x)-P_k(y,x')\r|
\le C\lf[\frac{d(x,x')}{\dz_0^k+d(x,y)}\r]^{\eta}
G_{\dz_0^k,\gz}(x,y);
\end{align*}
\item[(iii)]
for any $k\in\zz$ and $x$, $x'$, $y$, $y'\in\cx$
with $d(x,x')\le (2A_0)^{-2}[\dz_0^k+d(x,y)] $
and $d(y,y')\le (2A_0)^{-2}[\dz_0^k+d(x,y)]$,
\begin{align*}
&\lf|\lf[P_k(x,y)-P_k(x',y)\r]-\lf[P_k(x,y')
-P_k(x',y')\r]\r|\\
&\quad\le C\lf[\frac{d(x,x')}{\dz_0^k+d(x,y)}\r]
^{\eta}\lf[\frac{d(y,y')}{\dz_0^k+d(x,y)}\r]^{\eta}
G_{\dz_0^k,\gz}(x,y).
\end{align*}
\end{itemize}
\end{remark}

Now, we recall the concepts of both
test functions and distributions;
see, for instance, \cite[Definition 2.2]{hmy06}
and \cite[Definition 2.8]{hmy08}.

\begin{definition}\label{deg}
Let $x\in\cx$, $r\in(0,\fz)$, $\bz\in(0,\eta]$
with $\eta$ as in \eqref{eqwa},
and $\gamma\in(0,\fz)$. The \emph{space $\cg(x,r,\bz,\gz)$}
is defined to be the set of all the $\mu$-measurable functions
$f$ on $\cx$ such that there exists
a positive constant $C$ satisfying the
following conditions:
\begin{enumerate}
\item [(i)] (the \emph{size condition})
for any $y\in\cx$,
$|f(y)|\le C G_{r,\gz}(x,y)$,
here and thereafter, $G_{r,\gz}$ is
as in \eqref{eqg};
\item[(ii)] (the \emph{regularity condition})
for any $y$, $y'\in\cx$
satisfying $d(y,y')\le (2A_0)^{-1}[r+d(x,y)]$,
$$
\lf|f(y)-f(y')\r|\le C\lf[\frac{d(y,y')}
{r+d(x,y)}\r]^\bz G_{r,\gz}(x,y).
$$
\end{enumerate}
Furthermore, for any $f\in \cg(x,r,\bz,\gz)$, the norm
$\|f\|_{\cg(x,r,\bz,\gz)}$ is defined by setting
$$
\|f\|_{\cg(x,r,\bz,\gz)}:=\inf\lf\{C\in(0,\fz):
\ C\ \mathrm{satisfies}\ \mathrm{(i)}\
\mathrm{and}\ \mathrm{(ii)}\r\}.
$$
\end{definition}

Let $\bz\in(0,\eta]$ with
$\eta$ as in \eqref{eqwa},
and $\gz\in(0,\fz)$.
Next, we choose a \emph{basepoint}
$x_0\in\cx$. Observe that, for any
$x\in\cx$ and $r\in(0,\fz)$,
$\cg(x,r,\bz,\gz)=\cg(x_0,1,\bz,\gz)$
with the positive equivalence constants depending
on $x$ and $r$. In what follows, we fix an $x_0\in\cx$
and write $\cg(\bz,\gz):=\cg(x_0,1,\bz,\gz)$.
For any $\bz,\ \gz\in (0,\eta)$, the space $\gs$
is defined to be the closure of the set
$\cg(\eta,\eta)$ in the space
$\cg(\bz,\gz)$ and the norm
$\|\cdot\|_{\gs}:=\|\cdot\|_{\cg(\bz,\gz)}$.
The dual of $\gs$ is denoted by $(\gs)'$
which is equipped with the weak-$*$ topology.
In this article, we always choose $\gs$
as the \emph{space of test functions},
and $(\gs)'$ as
the \emph{space of distributions}.

Observe that $\cg(\eta,\eta)$ is a dense subset of $\gs$.
Let $\cg_{\mathrm{b}}(\eta,\eta)$ denote
the set of all the functions in $\cg(\eta,\eta)$
with bounded support.
We then show that $\cg_{\mathrm{b}}(\eta,\eta)$
is dense in $\gs$.
\begin{proposition}\label{prd}
Let $\bz$, $\gz\in(0,\eta)$ with
$\eta$ as in \eqref{eqwa}. Then the set
$\cg_{\mathrm{b}}(\eta,\eta)$ is dense in $\gs$.
\end{proposition}
\begin{proof}
Let $\eta$, $\bz$, and $\gz$ be as in the present proposition.
Since, by the definition of $\gs$,
$\cg(\eta,\eta)$ is dense in $\gs$,
to prove the present proposition,
it suffices to show that, for any $f\in \cg(\eta,\eta)$,
there exists a sequence
$\{f_j\}_{j\in\nn}\subset \cg_{\mathrm{b}}(\eta,\eta)$
such that $\lim_{j\to\fz}\|f-f_j\|_{\cg(\bz,\gz)}=0.$

To this end, fix $f\in\cg(\eta,\eta)$.
Recall that $x_0\in\cx$ is the fixed basepoint.
By \cite[Corollary 4.2]{ah13}, we find that
there exist a positive constant $C$ and
a sequence $\{\varphi_j\}_{j\in\nn}$
of functions such that, for any $j\in\nn$,
\begin{equation}\label{phij}
\mathbf{1}_{B(x_0,j)}\le \varphi_j
\le \mathbf{1}_{B(x_0,2A_0j)}
\end{equation}
and,
for any $j\in\nn$ and $y$, $y'\in\cx$,
\begin{equation}\label{eqce}
\lf|\varphi_j(y)-\varphi_j(y')\r|
\le C\lf[\frac{d(y,y')}{j}\r]^\eta.
\end{equation}
Now, for any $j\in\nn$, let $f_j:=f\varphi_j$.

Next, we show that
\begin{equation}\label{eqds1}
\lim_{j\to\fz}\lf\|f-f_j\r\|_{\cg(\bz,\gz)}=0.
\end{equation}
Indeed, to prove the size condition of
$f-f_j$, from \eqref{phij},
the size condition of $f$,
and $\gz<\eta$,
it follows that, for any $j\in\nn$ and $y\in\cx$,
\begin{align}\label{alds1}
\lf|f(y)-f_j(y)\r|
&\le |f(y)|
\mathbf{1}_{(B(x_0,j))^\complement}(y)
\le \|f\|_{\cg(\eta,\eta)}
G_{1,\eta}(x_0,y)
\mathbf{1}_{(B(x_0,j))^\complement}(y)\\
&\le j^{\gz-\eta}\|f\|_{\cg(\eta,\eta)}
G_{1,\gz}(x_0,y),\notag
\end{align}
here and in the remainder of the present proof,
 $G_{1,\eta}$ is as in
\eqref{eqg} with $r=1$ and $\gz=\eta$, and
$G_{1,\gz}$ as in \eqref{eqg}
with $r=1$.
Then we consider the regularity of $f-f_j$.
Indeed, we have, for any $j\in\nn$ and $y$,
$y'\in\cx$ with $d(y,y')\le (2A_0)^{-1}
[1+d(x_0,y)]$,
\begin{align}\label{alds2}
&\lf|\lf[f(y)-f_j(y)\r]-
\lf[f(y')-f_j(y')\r]\r|\\
&\quad=\lf|f(y)\lf[1-\varphi_j(y)\r]
-f(y)\lf[1-\varphi_j(y')\r]
+f(y)\lf[1-\varphi_j(y')\r]
-f(y')\lf[1-\varphi_j(y')\r]\r|\notag\\
&\quad\le |f(y)|\lf|\varphi_j(y)-\varphi_j(y')\r|
+\lf|f(y)-f(y')\r|\lf|1-\varphi_j(y')\r|
=:I_1+I_2.\notag
\end{align}

Observe that, for any $y$, $y'\in\cx$ with
$d(y,y')\le(2A_0)^{-1}[1+d(x_0,y)]$,
\begin{equation}\label{eqds4}
(2A_0)^{-1}[d(x_0,y)-1]
\le d(x_0,y')
\le(A_0+1/2)d(x_0,y)+1/2.
\end{equation}
Now, we estimate $I_1$ by considering three cases
on $d(x_0,y)$.

Case 1)  $d(x_0,y)\ge 4A_0^2j+1$.
In this case, by both the first
inequality of \eqref{eqds4},
and \eqref{phij},
we find that, if $j\in\nn$ and
$y'\in\cx$ with
$d(y,y')\le(2A_0)^{-1}[1+d(x_0,y)]$, then
$d(x_0,y')\ge 2A_0j$ and hence
$I_1=0.$

Case 2) $d(x_0,y)< (A_0+1/2)^{-1}(j-1/2)$.
In this case, from the second inequality
of \eqref{eqds4}, and \eqref{phij},
we deduce that, if $j\in\nn$ and
$y'\in\cx$ with
$d(y,y')\le(2A_0)^{-1}[1+d(x_0,y)]$, then
$d(x_0,y')<j$ and hence also $I_1=0.$

Case 3) $(A_0+1/2)^{-1}(j-1/2)\le d(x_0,y)< 4A_0^2j+1$.
In this case,
by the size condition of $f$, $\eqref{eqce}$,
and $\max\{\bz,\gz\}<\eta$,
we find that, if $j\in\nn$ and
$y'\in\cx$
with $d(y,y')\le(2A_0)^{-1}[1+d(x_0,y)]$, then
\begin{align*}
I_1
&\ls \|f\|_{\cg(\eta,\eta)}
G_{1,\eta}(x_0,y)
\lf[\frac{d(y,y')}{j}\r]^\eta\\
&\sim j^{\gz-\eta}\|f\|_{\cg(\eta,\eta)}
\lf[\frac{d(y,y')}{1+d(x_0,y)}\r]^\bz
G_{1,\gz}(x_0,y),
\end{align*}
where all the implicit
positive constants are independent of
$j$, $y$, and $y'$.

Combining the above three cases,
we conclude that, for any $j\in\nn$ and $y$,
$y'\in\cx$ with $d(y,y')\le (2A_0)^{-1}[1+d(x_0,y)]$,
\begin{align}\label{alds3}
I_1
\ls j^{\gz-\eta}\|f\|_{\cg(\eta,\eta)}
\lf[\frac{d(y,y')}{1+d(x_0,y)}\r]^\bz
G_{1,\gz}(x_0,y).
\end{align}

Next, we estimate $I_2$.
From the regularity condition of $f$, \eqref{phij},
the second inequality of \eqref{eqds4},
and $\max\{\bz,\gz\}<\eta$,
it follows that, for any $j\in\nn$ and $y$,
$y'\in\cx$ with
$d(y,y')\le (2A_0)^{-1}[1+d(x_0,y)]$,
\begin{align}\label{alds4}
I_2
&\le \|f\|_{\cg(\eta,\eta)}
\lf[\frac{d(y,y')}{1+d(x_0,y)}\r]^\eta
G_{1,\eta}(x_0,y)
\mathbf{1}_{(B(x_0,j))^\complement}(y')\\
&\le \|f\|_{\cg(\eta,\eta)}
\lf[\frac{d(y,y')}{1+d(x_0,y)}\r]^\eta
G_{1,\eta}(x_0,y)
\mathbf{1}_{(B(x_0,(j-1/2)/(A_0+1/2)))^\complement}(y)\notag\\
&\ls j^{\gz-\eta}\|f\|_{\cg(\eta,\eta)}
\lf[\frac{d(y,y')}{1+d(x_0,y)}\r]^\bz
G_{1,\gz}(x_0,y).\notag
\end{align}

By \eqref{alds2}, \eqref{alds3},
and \eqref{alds4}, we conclude that,
for any $j\in\nn$
and $y$, $y'\in\cx$ with
$d(y,y')\le(2A_0)^{-1}[1+d(x_0,y)]$,
\begin{align*}
&\lf|\lf[f(y)-f_j(y)\r]-
\lf[f(y')-f_j(y')\r]\r|\\
&\quad\ls j^{\gz-\eta}\|f\|_{\cg(\eta,\eta)}
\lf[\frac{d(y,y')}{1+d(x_0,y)}\r]^\bz
G_{1,\gz}(x_0,y),
\end{align*}
which, combined with \eqref{alds1}, implies
that, for any $j\in\nn$,
$$
\lf\|f-f_j\r\|_{\cg(\bz,\gz)}
\ls j^{\gz-\eta}\|f\|_{\cg(\eta,\eta)}.
$$
Letting $j\to\fz$, we then complete
the proof of $\eqref{eqds1}$.

By an argument similar to that used in the proof
of \eqref{eqds1}, we find that, for any $j\in\nn$,
$f-f_j\in\cg(\eta,\eta)$, which, together with
$f\in \cg(\eta,\eta)$, implies that $f_j\in\cg(\eta,\eta)$.
Obviously, for any $j\in\nn$, $f_j$ supports in
$B(x_0,2A_0j)$. Thus,
$\{f_j\}_{j\in\nn}\subset \cg_{\mathrm{b}}(\eta,\eta)$.
This finishes the proof of Proposition \ref{prd}.
\end{proof}

Let $\bz$, $\gz\in (0,\eta)$, $f\in (\gs)'$,
and $\{P_k\}_{k\in\zz}$ be a 1-exp-ATI as in
Definition \ref{dee1}. For any $k\in\zz$, the operator
$P^*_k$ on $\gs$ is defined by setting,
for any $\varphi\in\gs$ and $x\in\cx$,
$$P^*_k(\varphi)(x):
=\int_\cx P_k(z,x)\varphi(z)\,d\mu(z).$$
Recall that, by an argument similar to that
used in the proof of \cite[Lemma 4.14]{hlyy19},
we find that, for any $k\in\zz$,
$P_k^*$ is bounded on $\gs$.
From this, it follows that, for any $k\in\zz$,
$P_k(f)$ is a bounded linear functional on $\gs$, where
$P_k(f)$ is defined by setting,
for any $\varphi\in\gs$,
$$
\lf\langle P_k(f),\varphi\r\rangle:
=\lf\langle f,P^*_k(\varphi)\r\rangle.
$$
For any $k\in\zz$, the function $\widetilde{P_k(f)}$ on $\cx$
is defined by setting, for any $x\in\cx$,
$$
\widetilde{P_k(f)}(x):=
\lf\langle f, P_k(x,\cdot)\r\rangle.
$$
Then we obtain the relation between $P_k(f)$
and $\widetilde{P_k(f)}$
as in the following proposition.

\begin{proposition}\label{prdc}
Let $\bz$, $\gz\in (0,\eta)$
with $\eta$ as in \eqref{eqwa}, and
$\{P_k\}_{k\in\zz}$ be a 1-$\exp$-\rm{ATI}. Then,
 for any $k\in\zz$,
$f\in (\gs)'$, and
$\psi\in\cg_{\mathrm{b}}(\eta,\eta)$,
$$
\lf\langle P_k(f),\psi\r\rangle
=\int_\cx \widetilde{P_k(f)}(x)
\psi(x)\,d\mu(x).
$$
\end{proposition}

\begin{proof}
We show the present proposition by
borrowing some ideas from the proof of \cite[Lemma 3.12]{hmy08}.
Let $\bz$, $\gz$, $\eta$, and
$\{P_k\}_{k\in\zz}$ be as in the
present proposition. Without loss of generality, we may only
prove that, for any $f\in (\gs)'$ and
$\psi\in\cg_{\mathrm{b}}(\eta,\eta)$,
\begin{equation*}
\lf\langle P_0(f),\psi\r\rangle=
\int_\cx \widetilde{P_0(f)}(x)
\psi(x)\,d\mu(x).
\end{equation*}

Now, we fix $f\in (\gs)'$ and
$\psi\in\cg_{\mathrm{b}}(\eta,\eta)$.
Let $L\in \nn$ be such that
$$
\supp(\psi):
=\{x\in\cx:\ \psi(x)\neq 0\}
\subset B(x_0, L),
$$
where $x_0\in\cx$
is the fixed basepoint.
By this and the size condition of $\psi$, we conclude that,
for any $\ez\in(0,\fz)$, there exists a positive constant $C$
such that, for any $x\in\cx$,
\begin{equation}\label{eqdc5}
|\psi(x)|\le CG_{1,\ez}(x_0,x),
\end{equation}
where $G_{1,\ez}$ is as in \eqref{eqg}
with $r=1$ and $\gz=\ez$.
Let $\{Q^k_\az\}_{k\in\zz,\az\in\ca_k}$ be
the dyadic cubes and $\{z^k_\az\}_{k\in\zz,\az\in\ca_k}$
their centers as in Lemma \ref{ledy}. For any $k\in\zz$, let
$$
\ca^{(L)}_k:=\lf\{\az\in\ca_k:\
Q^k_\az\cap B(x_0,L)\neq \emptyset\r\}.
$$
Obviously, for any $k\in\zz$, $\ca^{(L)}_k$
is a finite set.
For any $k\in\zz$ and $y\in\cx$, let
$$
\psi_k(y):
=\sum_{\az\in\ca^{(L)}_k}\int_{Q^k_\az}
P_0(z^k_\az,y)\psi(x)\,d\mu(x).
$$
Then we show that
\begin{equation}\label{eqdc3}
\lim_{k\to\fz}\lf\|P^*_0(\psi)
-\psi_k\r\|_{\cg(\bz,\gz)}=0.
\end{equation}

To prove \eqref{eqdc3}, we first show that
$P^*_0(\psi)-\psi_k$ satisfies
the size condition of $\cg(\bz,\gz)$.
From Lemma \ref{ledy}(v), it follows that,
for any $k\in\zz$, $\az\in\ca_k$, and $x\in Q^k_\az$,
\begin{equation}\label{eqdc2}
d(z^k_\az,x)\le C_L\dz_0^k,
\end{equation}
where $C_L$ is as in Lemma \ref{ledy}. Observe that, for any $y\in\cx$,
\begin{equation}\label{eqdc1}
\cx=\lf\{x\in\cx:\ d(x,y)\ge(2A_0)^{-1}d(x_0,y)\r\}
\bigcup\lf\{x\in\cx:\ d(x,x_0)\ge(2A_0)^{-1}d(x_0,y)\r\}.
\end{equation}
By the definition of $\psi_k$,
Remark \ref{rede1}(ii) together with \eqref{eqdc2},
\eqref{eqdc5} with $\ez=\gz$, \eqref{eqdc1}, and both
(i) and (ii) of Lemma \ref{lees}, we conclude that,
for any $k\in\zz$ with $C_L\dz_0^k\le (2A_0)^{-1}$,
and $y\in\cx$,
\begin{align}\label{aldc6}
&\lf|P^*_0(\psi)(y)-\psi_k(y)\r|\\
&\quad\le\sum_{\az\in\ca^{(L)}_k}\int_{Q^k_\az}
\lf|P_0(x,y)-P_0(z^k_\az,y)\r|
|\psi(x)|\,d\mu(x)\notag\\
&\quad\ls\sum_{\az\in\ca^{(L)}_k}
\int_{Q^k_\az}\dz_0^{k\eta}
G_{1,\gz}(x,y)|\psi(x)|\,d\mu(x)\notag\\
&\quad\ls\int_{\cx}\dz_0^{k\eta}
G_{1,\gz}(x,y)G_{1,\gz}(x_0,x)\,d\mu(x)\notag\\
&\quad\ls \int_{d(x,y)\ge (2A_0)^{-1}d(x_0,y)}
\dz_0^{k\eta}
G_{1,\gz}(x,y)G_{1,\gz}(x_0,x)\,d\mu(x)+\int_{d(x,x_0)\ge (2A_0)^{-1}d(x_0,y)}
\cdots\notag\\
&\quad\ls \dz_0^{k\eta}G_{1,\gz}(x_0,y),\notag
\end{align}
here and in the remainder of the present proof,
 $G_{1,\gz}$ is as in
\eqref{eqg} with $r=1$.
Thus, we obtain the size condition of
$P^*_0(\psi)-\psi_k$.

Next, we prove that $P^*_0(\psi)-\psi_k$ satisfies
the regularity condition of $\cg(\bz,\gz)$;
more precisely, we show that, for any $k\in\zz$
with $C_L\dz_0^k\le (2A_0)^{-2}$, and
$y$, $y'\in\cx$ with $d(y,y')<(2A_0)^{-1}[1+d(x_0,y)]$,
\begin{align}\label{eqdc7}
Y_k(y,y'):
&=\lf|\lf[P^*_0(\psi)(y)-\psi_k(y)\r]
-\lf[P^*_0(\psi)(y')-\psi_k(y')\r]\r|\\
&\ls \dz_0^{k\eta}
\lf[\frac{d(y,y')}{1+d(x_0,y)}\r]^\bz
G_{1,\gz}(x_0,y).\notag
\end{align}

Observe that, for any $y$, $y'\in\cx$ with
$d(y,y')<(2A_0)^{-1}[1+d(x_0,y)]$,
\begin{equation}\label{eqdc4}
1+d(x_0,y)\sim 1+d(x_0,y').
\end{equation}
Then we estimate $Y_k(y,y')$ by considering
the following two cases on $d(y,y')$.

Case 1) $(2A_0)^{-2}[1+d(x_0,y)]<
d(y,y')\le (2A_0)^{-1}[1+d(x_0,y)]$.
In this case,
using \eqref{aldc6} and Lemma \ref{lees}(i)
combined with \eqref{eqdc4},
we conclude that, for any $k\in\zz$ with
$C_L\dz_0^k\le (2A_0)^{-1}$,
\begin{align*}
Y_k(y,y')
&\le\lf|P^*_0(\psi)(y)-\psi_k(y)\r|
+\lf|P^*_0(\psi)(y')-\psi_k(y')\r|\\
&\ls \dz_0^{k\eta}G_{1,\gz}(x_0,y)
+\dz_0^{k\eta}G_{1,\gz}(x_0,y')
\sim\dz_0^{k\eta}G_{1,\gz}(x_0,y)\\
&\ls \dz_0^{k\eta}
\lf[\frac{d(y,y')}{1+d(x_0,y)}\r]^\bz
G_{1,\gz}(x_0,y).
\end{align*}

Case 2) $d(y,y')\le (2A_0)^{-2}[1+d(x_0,y)]$.
In this case, we first claim that,
for any $k\in\zz$ with $C_L\dz_0^k\le (2A_0)^{-2}$,
$\az\in\ca^{(L)}_k$, $x\in Q^k_\az$, and
$y$, $y'\in\cx$,
\begin{align}\label{aldc7}
Z(x,y,z^k_\az,y'):
&=\lf|\lf[P_0(x,y)-P_0(z^k_\az,y)\r]
-\lf[P_0(x,y')-P_0(z^k_\az,y')\r]\r|\\
&\ls\dz_0^{k\eta}
\lf[\frac{d(y,y')}{1+d(x,y)}\r]^\eta
\lf[G_{1,\gz}(x,y)+G_{1,\gz}(x,y')\r].\notag
\end{align}
Indeed, if $d(y,y')\le (2A_0)^{-2}[1+d(x,y)]$,
by Remark \ref{rede1}(iii) and \eqref{eqdc2}, we have
\begin{align*}
Z(x,y,z^k_\az,y')&
\ls \lf[\frac{d(x,z^k_\az)}{1+d(x,y)}\r]^\eta
\lf[\frac{d(y,y')}{1+d(x,y)}\r]^\eta
G_{1,\gz}(x,y)\\
&\ls\dz_0^{k\eta}
\lf[\frac{d(y,y')}{1+d(x,y)}\r]^\eta
\lf[G_{1,\gz}(x,y)+G_{1,\gz}(x,y')\r]
\end{align*}
with the implicit positive constants
independent of $k$, $\az$, $x$, $y$, and $y'$;
if $d(y,y')> (2A_0)^{-2}[1+d(x,y)]$,
by Remark \ref{rede1}(ii) and \eqref{eqdc2},
we find that
\begin{align*}
Z(x,y,z^k_\az,y')
&\le \lf|P_0(x,y)-P_0(z^k_\az,y)\r|
+\lf|P_0(x,y')-P_0(z^k_\az,y')\r|\\
&\ls \lf[\frac{d(x,z^k_\az)}{1+d(x,y)}\r]^\eta
G_{1,\gz}(x,y)+\lf[\frac{d(x,z^k_\az)}{1+d(x,y')}\r]^\eta
G_{1,\gz}(x,y')\\
&\ls\dz_0^{k\eta}
\lf[\frac{d(y,y')}{1+d(x,y)}\r]^\eta
\lf[G_{1,\gz}(x,y)+G_{1,\gz}(x,y')\r]
\end{align*}
with the implicit positive constants
independent of $k$, $\az$, $x$, $y$, and $y'$.
This finishes the proof of the claim \eqref{aldc7}.

Now, we estimate $Y_k(y,y')$ in Case 2).
From \eqref{aldc7}, \eqref{eqdc5} with $\ez=\eta+\gz$, and
\eqref{eqdc1},
we deduce that, for any $k\in\zz$ with $C_L\dz_0^k\le (2A_0)^{-2}$,
and any $y$, $y'\in\cx$ as in case 2),
\begin{align}\label{aldc1}
Y_k(y,y')
&\le\sum_{\az\in\ca^{(L)}_k}\int_{Q^k_\az}
Z(x,y,z^k_\az,y')|\psi(x)|\,d\mu(x)\\
&\ls\int_{\cx}\dz_0^{k\eta}
\lf[\frac{d(y,y')}{1+d(x,y)}\r]^\eta
\lf[G_{1,\gz}(x,y)+G_{1,\gz}(x,y')\r]
|\psi(x)|\,d\mu(x)\notag\\
&\ls\int_{\cx}\dz_0^{k\eta}
\lf[\frac{d(y,y')}{1+d(x,y)}\r]^\eta
\lf[G_{1,\gz}(x,y)+G_{1,\gz}(x,y')\r]
G_{1,\eta+\gz}(x_0,x)\,d\mu(x)\notag\\
&\ls \int_{d(x,y)\ge(2A_0)^{-1}d(x_0,y)}\dz_0^{k\eta}
\lf[\frac{d(y,y')}{1+d(x,y)}\r]^\eta
\lf[G_{1,\gz}(x,y)+G_{1,\gz}(x,y')\r]\notag\\
&\quad\times G_{1,\eta+\gz}(x_0,x)\,d\mu(x)
+\int_{d(x,x_0)\ge (2A_0)^{-1}d(x_0,y)}\cdots\notag\\
&=:I_1+I_2,\notag
\end{align}
where $G_{1,\eta+\gz}$ is as in
\eqref{eqg} with $r$ and $\gz$
replaced, respectively, by $1$ and $\eta+\gz$.

We first estimate $I_1$.
By \eqref{eqA0}, we find that, for any $y$, $y'\in\cx$
as in Case 2), and $x\in\cx$ with
$d(x,y)\ge(2A_0)^{-1}d(x_0,y)$,
$$
1+d(x,y')\ge 1+(A_0)^{-1}d(x,y)-d(y,y')
\ge 1-(2A_0)^{-2}+(2A_0)^{-2}d(x_0,y).
$$
From this, \eqref{eqoz}, both (i) and
(ii) of Lemma \ref{lees}, and $\bz<\eta$,
we deduce that
\begin{align*}
I_1
&\ls\dz_0^{k\eta}
\lf[\frac{d(y,y')}{1+d(x_0,y)}\r]^\eta
G_{1,\gz}(x_0,y)
\int_\cx G_{1,\eta+\gz}(x_0,x)\,d\mu(x)\\
&\ls \dz_0^{k\eta}
\lf[\frac{d(y,y')}{1+d(x_0,y)}\r]^\bz
G_{1,\gz}(x_0,y),
\end{align*}
where all the implicit positive constants
are independent of $k$, $y$, and $y'$.
As for $I_2$, using \eqref{eqoz}, both
(i) and (ii) of Lemma \ref{lees}, and
$\bz<\eta$, we have
\begin{align*}
I_2
&\ls \dz_0^{k\eta}
\lf[\frac{d(y,y')}{1+d(x_0,y)}\r]^\eta
G_{1,\gz}(x_0,y)\int_\cx
\lf[G_{1,\gz}(x,y)+G_{1,\gz}(x,y')\r]\,d\mu(x)\\
&\ls \dz_0^{k\eta}
\lf[\frac{d(y,y')}{1+d(x_0,y)}\r]^\bz
G_{1,\gz}(x_0,y),
\end{align*}
where all the implicit positive constants
are independent of $k$, $y$, and $y'$.

From \eqref{aldc1} and the estimates of
$I_1$ and $I_2$, it follows that,
for any $k\in\zz$ with $C_L\dz_0^k\le (2A_0)^{-2}$,
and any $y$, $y'\in\cx$ as in case 2),
\begin{equation*}
Y_k(y,y')
\ls \dz_0^{k\eta}
\lf[\frac{d(y,y')}{1+d(x_0,y)}\r]^\bz
G_{1,\gz}(x_0,y).
\end{equation*}
This finishes the proof of the case 2).

Using Case 1) and Case 2), we obtain \eqref{eqdc7}.
By \eqref{aldc6} and \eqref{eqdc7}, we find that,
for any $k\in\zz$ with $C_L\dz_0^k\le (2A_0)^{-2}$,
$$
\lf\|P^*_0(\psi)-\psi_k\r\|_{\cg(\bz,\gz)}
\ls \dz_0^k.
$$
Letting $k\to\fz$, we obtain \eqref{eqdc3}.

By an argument similar to that
used in the proof of \cite[Lemma 4.14]{hlyy19},
we find that $P^*_0(\psi)\in \cg(\eta,\eta)\subset \gs$.
Observe that, for any $k\in\zz$ and $\az\in\ca^{(L)}_k$,
$P_0(z^k_\az,\cdot)\in \cg(\eta,\eta)$.
Thus, for any $k\in\zz$, $\psi_k\in\cg(\eta,\eta)\subset\gs$.
From this, \eqref{eqdc3},
and the estimate that $\# \ca^{(L)}_k<\fz$
for any $k\in\zz$,
we deduce that
\begin{align*}
\lf\langle P_0(f),\psi\r\rangle
&=\lf\langle f,P^*_0(\psi)\r\rangle\\
&=\lim_{k\to\fz}\lf\langle f,
\sum_{\az\in\ca^{(L)}_k}
\int_{Q^k_\az}P_0(z^k_\az,\cdot)
\psi(x)\,d\mu(x)\r\rangle\\
&=\lim_{k\to\fz}\sum_{\az\in\ca^{(L)}_k}
\int_{Q^k_\az}
\widetilde{P_0(f)}(z^k_\az)
\psi(x)\,d\mu(x).
\end{align*}
Thus, to complete the proof of the present proposition,
it remains to show that
\begin{equation}\label{eqdc9}
\lim_{k\to\fz}\sum_{\az\in\ca^{(L)}_k}
\int_{Q^k_\az}
\widetilde{P_0(f)}(z^k_\az)
\psi(x)\,d\mu(x)
=\int_{\cx}\widetilde{P_0(f)}(x)\psi(x)\,d\mu(x).
\end{equation}

To this end, we need to prove that the right-hand side
of \eqref{eqdc9} is well defined; namely,
 we show that $\widetilde{P_0(f)}\psi$ is integrable on $\cx$.
Observe that, for any $x\in\cx$,
\begin{equation}\label{eqdc8}
\lf|\widetilde{P_0(f)}(x)\r|
=\lf|\lf\langle f,P_0(x,\cdot)\r\rangle\r|
\le \|f\|_{(\gs)'}
\lf\|P_0(x,\cdot)\r\|_{\cg(\bz,\gz)}.
\end{equation}
Then we estimate $\lf\|P_0(x,\cdot)\r\|_{\cg(\bz,\gz)}$.
Indeed, from Remark \ref{rede1}(i),
Lemma \ref{lees}(i) combined with \eqref{eqoz},
and the estimate that $d(x_0,y)\le A_0[d(x_0,x)+d(x,y)]$
for any $x$, $y\in\cx$ [which is deduced from \eqref{eqA0}],
it follows that, for any $x$, $y\in\cx$,
\begin{align}\label{aldc2}
|P_0(x,y)|
&\ls  G_{1,\gz}(x,y)\ls G_{1,\gz}(x_0,y)
\max\lf\{1,\lf[\frac{1+d(x_0,y)}
{1+d(x,y)}\r]^{\oz+\gz}\r\}\\
&\ls G_{1,\gz}(x_0,y)
\lf[1+d(x_0,x)\r]^{\oz+\gz}\notag
\end{align}
with the implicit positive constants independent of $x$ and $y$.
Thus, $P_0(x,\cdot)$
satisfies the size condition of $\cg(\bz,\gz)$.

Next, we consider the regularity condition of
$P_0(x,\cdot)$. To this end, we first claim
that, for any $x$, $y$, $y'\in\cx$ with
$d(y,y')\le (2A_0)^{-1}[1+d(x_0,y)]$,
\begin{align}\label{eqd1}
\lf|P_0(x,y)-P_0(x,y')\r|
\ls\lf[\frac{d(y,y')}{1+d(x,y)}\r]^\bz
\lf[G_{1,\gz}(x,y)+G_{1,\gz}(x,y')\r].
\end{align}
Indeed, if $d(y,y')\le (2A_0)^{-1}[1+d(x,y)]$,
by Remark \ref{rede1}(i), we have
\begin{align*}
\lf|P_0(x,y)-P_0(x,y')\r|
&\ls\lf[\frac{d(y,y')}{1+d(x,y)}\r]^\bz
G_{1,\gz}(x,y)\\
&\ls \lf[\frac{d(y,y')}{1+d(x,y)}\r]^\bz
\lf[G_{1,\gz}(x,y)+G_{1,\gz}(x,y')\r]
\end{align*}
with the implicit positive constants
independent of $x$, $y$, and $y'$;
if $d(y,y')> (2A_0)^{-1}[1+d(x,y)]$,
by Remark \ref{rede1}(ii), we obtain
\begin{align*}
\lf|P_0(x,y)-P_0(x,y')\r|
&\le \lf|P_0(x,y)\r|+\lf|P_0(x,y')\r|\\
&\ls\lf[G_{1,\gz}(x,y)+G_{1,\gz}(x,y')\r]\\
&\ls \lf[\frac{d(y,y')}{1+d(x,y)}\r]^\bz
\lf[G_{1,\gz}(x,y)+G_{1,\gz}(x,y')\r]
\end{align*}
with the implicit positive constants
independent of $x$, $y$, and $y'$.
This finishes the proof of the claim \eqref{eqd1}.

Using \eqref{eqd1},
Lemma \ref{lees}(i) together with \eqref{eqoz},
the estimates that $d(x_0,y)\le A_0[d(x_0,x)+d(x,y)]$
and $d(x_0,y')\le A_0[d(x_0,x)+d(x,y')]$
for any $x$, $y$, $y'\in\cx$
[which are deduced from \eqref{eqA0}], and
Lemma \ref{lees}(i) with \eqref{eqdc4},
we find that, for any $x$, $y$, $y'\in\cx$ with
$d(y,y')\le (2A_0)^{-1}[1+d(x_0,y)]$,
\begin{align}\label{aldc3}
&\lf|P_0(x,y)-P_0(x,y')\r|\\
&\quad\ls\lf[\frac{d(y,y')}{1+d(x,y)}\r]^\bz
\lf[G_{1,\gz}(x,y)+G_{1,\gz}(x,y')\r]\notag\\
&\quad\ls \lf[\frac{d(y,y')}{1+d(x_0,y)}\r]^\bz
\lf[\frac{1+d(x_0,y)}{1+d(x,y)}\r]^\bz
\lf[G_{1,\gz}(x_0,y)
\max\lf\{1,\lf[\frac{1+d(x_0,y)}
{1+d(x,y)}\r]^{\oz+\gz}\r\}\r.\notag\\
&\quad\quad\lf.+G_{1,\gz}(x_0,y')
\max\lf\{1,\lf[\frac{1+d(x_0,y')}
{1+d(x,y')}\r]^{\oz+\gz}\r\}\r]\notag\\
&\quad\ls \lf[\frac{d(y,y')}{1+d(x_0,y)}\r]^\bz
\lf[G_{1,\gz}(x_0,y)+G_{1,\gz}(x_0,y')\r]
\lf[1+d(x_0,x)\r]^{\bz+\oz+\gz}\notag\\
&\quad\sim \lf[\frac{d(y,y')}{1+d(x_0,y)}\r]^\bz
G_{1,\gz}(x_0,y)
\lf[1+d(x_0,x)\r]^{\bz+\oz+\gz}.\notag
\end{align}
From this and \eqref{aldc2}, we deduce that, for any $x\in\cx$,
\begin{equation*}
\|P_0(x,\cdot)\|_{\cg(\bz,\gz)}
\ls [1+d(x_0,x)]^{\bz+\oz+\gz},
\end{equation*}
which, together with \eqref{eqdc8}, implies that,
for any $x\in\cx$,
\begin{align*}
\lf|\widetilde{P_0(f)}(x)\r|
\le \|f\|_{(\gs)'}
\|P_0(x,\cdot)\|_{\cg(\bz,\gz)}
\ls \|f\|_{(\gs)'}
\lf[1+d(x_0,x)\r]^{\bz+\oz+\gz}.
\end{align*}
By this and $\psi\in\cg_{\mathrm{b}}(\eta,\eta)$,
we find that $\widetilde{P_0(f)}\psi$ is integrable on $\cx$
and the right-hand side of \eqref{eqdc9} is well defined.

Now, we prove \eqref{eqdc9}. To this end,
we show that, for any $x$, $x'\in\cx$
with $d(x,x')\le (2A_0)^{-2}$,
\begin{align}\label{aldc5}
&\lf|\widetilde{P_0(f)}(x)-
\widetilde{P_0(f)}(x')\r|
\ls \|f\|_{(\gs)'}
\lf[d(x,x')\r]^\eta
\lf[1+d(x_0,x)\r]^{\bz+\oz+\gz}.
\end{align}
Observe that, for any $x$, $x'\in\cx$,
\begin{align}\label{aldc4}
&\lf|\widetilde{P_0(f)}(x)-
\widetilde{P_0(f)}(x')\r|\\
&\quad=\lf|\langle f,P_0(x,\cdot)
-P_0(x',\cdot)\rangle\r|
\le \|f\|_{(\gs)'}
\lf\|P_0(x,\cdot)-P_0(x',\cdot)
\r\|_{\cg(\bz,\gz)}.\notag
\end{align}

Next, we estimate
$\|P_0(x,\cdot)-P_0(x',\cdot)\|_{\cg(\bz,\gz)}$.
To estimate the size condition of
$P_0(x,\cdot)-P_0(x',\cdot)$,
from Remark \ref{rede1}(ii) and some arguments
similar to those used in the estimation of \eqref{aldc2}, it follows that,
for any $x$, $x'$, $y\in\cx$
with $d(x,x')<(2A_0)^{-1}$,
\begin{align*}
&\lf|P_0(x,y)-P_0(x',y)\r|\\
&\quad\ls \lf[d(x,x')\r]^\eta
G_{1,\gz}(x,y)
\ls \lf[d(x,x')\r]^\eta
G_{1,\gz}(x_0,y)
\lf[1+d(x_0,x)\r]^{\oz+\gz}.
\end{align*}
By some arguments similar to those used
in the estimations of \eqref{aldc7} and \eqref{aldc3},
 we conclude that, for any $x$, $x'$,
$y$, $y'\in\cx$ with
$d(x,x')\le (2A_0)^{-2}$ and
$d(y,y')\le (2A_0)^{-1}[1+d(x_0,y)]$,
\begin{align*}
&\lf|\lf[P_0(x,y)-P_0(x',y)\r]
-\lf[P_0(x,y')-P_0(x',y')\r]\r|\\
&\quad\ls \lf[d(x,x')\r]^\eta
\lf[\frac{d(y,y')}{1+d(x,y)}\r]^\eta
\lf[G_{1,\gz}(x,y)+G_{1,\gz}(x,y')\r]\\
&\quad\ls \lf[d(x,x')\r]^\eta
\lf[\frac{d(y,y')}{1+d(x_0,y)}\r]^\bz
G_{1,\gz}(x_0,y)
\lf[1+d(x_0,x)\r]^{\bz+\oz+\gz}.
\end{align*}
From this and the size condition of
$P_0(x,\cdot)-P_0(x',\cdot)$, we deduce that,
for any $x$, $x'\in\cx$ with
$d(x,x')\le (2A_0)^{-2}$,
$$
\lf\|P_0(x,\cdot)-P_0(x',\cdot)
\r\|_{\cg(\bz,\gz)}
\ls \lf[d(x,x')\r]^\eta
\lf[1+d(x_0,x)\r]^{\bz+\oz+\gz}.
$$
Using this and \eqref{aldc4}, we obtain
\eqref{aldc5}.

By \eqref{eqdc2} and \eqref{aldc5},
we conclude that,
for any $k\in\zz$ with
$C_L\dz_0^k\le (2A_0)^{-2}$,
\begin{align*}
&\lim_{k\to\fz}\lf|\sum_{\az\in\ca^{(L)}_k}
\int_{Q^k_\az}\widetilde{P_0(f)}(z^k_\az)
\psi(x)\,d\mu(x)
-\int_{\cx}\widetilde{P_0(f)}(x)\psi(x)\,d\mu(x)\r|\\
&\quad\le \sum_{\az\in\ca^{(L)}_k}\int_{Q^k_\az}
\lf|\widetilde{P_0(f)}(z^k_\az)
-\widetilde{P_0(f)}(x)\r|
|\psi(x)|\,d\mu(x)\\
&\quad\ls \|f\|_{(\gs)'}
\dz_0^{k\eta}\int_\cx
\lf[1+d(x_0,x)\r]^{\bz+\oz+\gz}
|\psi(x)|\,d\mu(x).
\end{align*}
Letting $k\to\fz$ and using the fact
$\psi\in \cg_{\mathrm{b}}(\eta,\eta)$,
we obtain \eqref{eqdc9}. This finishes the proof
of Proposition \ref{prdc}.
\end{proof}
\begin{remark}\label{rewan}
Let $\{P_k\}_{k\in\zz}$ be a 1-exp-ATI as in
Definition \ref{dee1} and
$f\in(\gs)'$.
Using Propositions \ref{prd}
and \ref{prdc} together with a
standard density argument,
we find that $P_k(f)$ and $\widetilde{P_k(f)}$
coincide as the distribution of $(\gs)'$;
more precisely, for any $k\in\zz$ and $\varphi\in\gs$,
$$
\langle P_k(f),\varphi\rangle
=\lim_{n\to\fz}\langle P_k(f),\varphi_n\rangle
=\lim_{n\to\fz}\int_{\cx}\widetilde{P_k(f)}(x)
\varphi_n(x)\,d\mu(x)
=:\lf\langle \widetilde{P_k(f)},\varphi\r\rangle,
$$
where $\{\varphi_n\}_{n\in\nn}\subset \cg_{\mathrm{b}}(\eta,\eta)$
and $\lim_{n\to\fz}\|\varphi_n-\varphi\|_{\cg(\bz,\gz)}= 0$.
In what follows, by the abuse of the notation,
we identify $P_k(f)$ with $\widetilde{P_k(f)}$.
\end{remark}
\subsection{Weak Ball Quasi-Banach
Function Spaces \label{ssbb}}
Now, we recall the following
ball (quasi-)Banach function space on $\cx$, which
is just \cite[Definition 2.5]{yhyy21a} (see also
\cite[Definition 2.2 and (2.3)]{shyy17}).

\begin{definition}\label{debb}
A quasi-normed linear space
$X(\cx)\subset\mathscr{M}(\cx)$,
which is
equipped with a quasi-norm $\|\cdot\|_{X(\cx)}$,
is called a \emph{ball quasi-Banach function space} (for short, BQBF \emph{space}) on $\cx$
if it satisfies the following conditions:
\begin{enumerate}
\item [(i)]
$\|f\|_{X(\cx)}=0$ if and only if
$f=0$ $\mu$-almost everywhere.
\item [(ii)]
$|g|\le|f|$ $\mu$-almost everywhere implies
that $\|g\|_{X(\cx)}\le \|f\|_{X(\cx)}$.
\item [(iii)]
$0\le f_n\uparrow f$ $\mu$-almost everywhere
indicates that $\|f_n\|_{X(\cx)}\uparrow \|f\|_{X(\cx)}$.
\item [(iv)]
for any ball $B\subset\cx$, $\mathbf{1}_B\in X(\cx)$.
\end{enumerate}

A normed linear space $X(\cx)\subset \mathscr{M}(\cx)$
is called a \emph{ball Banach function space} (for short, BBF \emph{space}) on $\cx$
if it satisfies (i)-(iv) and
\begin{enumerate}
\item [(v)]
for any ball $B\subset\cx$,
there exists a positive constant $C$, depending
only on $B$, such that, for any $f\in X(\cx)$,
$$ \int_{B}\lf|f(x)\r|\,d\mu(x)
\le C\lf\|f\r\|_{X(\cx)}.$$
\end{enumerate}
\end{definition}

\begin{remark}\label{refz}
\begin{itemize}
\item[\textup{(i)}]
Let $X(\cx)$ be a BQBF space.
By \cite[Theorem 2]{dfmn21}, we find that both
(ii) and (iii) of Definition \ref{debb} imply
that $X(\cx)$ is complete and hence
a quasi-Banach space.
\item[\textup{(ii)}]
In Definition \ref{debb}, if we replace any ball $B$
by any $\mu$-measurable set $E$ with $\mu(E)<\fz$,
we obtain the definition of (quasi-)Banach function
spaces on $\cx$, which can be found in
\cite[p.\,3, Definition 1.3]{bs88} and
\cite[Definition 2.4]{yhyy21a}; if we replace any ball $B$
by any bounded $\mu$-measurable set $E$,
we obtain an equivalent definition of
ball (quasi-)Banach function spaces on $\cx$.
Obviously, a (quasi-)Banach function space on $\cx$
is surely a ball (quasi-)Banach function space on $\cx$.
\item[\textup{(iii)}]
Let $X(\cx)$ be a BQBF
space. If $f\in X(\cx)$,
then $\mu(\{x\in\cx:\ |f(x)|=\fz\})=0$.
Indeed, from Definition \ref{debb}(ii)
and $\|f\|_{X(\cx)}<\fz$, it follows that
\begin{equation*}
\lf\|\mathbf{1}_{\{x\in\cx:\ |f(x)|=\fz\}}\r\|_{X(\cx)}
\le \inf_{\lz\in(0,\fz)}\lf[
\lf\|\mathbf{1}_{\{x\in\cx:\ |f(x)|>\lz\}}\r\|_{X(\cx)}\r]
\le \inf_{\lz\in(0,\fz)}\lf\{
\lz^{-1}\|f\|_{X(\cx)}\r\}=0.
\end{equation*}
By this and Definition \ref{debb}(i),
we conclude that
$\mu(\{x\in\cx:\ |f(x)|=\fz\})=0$.
\end{itemize}
\end{remark}

A BQBF
space $X(\cx)$ is said to have an \emph{absolutely continuous
quasi-norm} if, for any $f\in X(\cx)$ and any sequence
$\{E_n\}_{n\in\nn}$ of $\mu$-measurable sets of $\cx$ with
$E_n\downarrow\emptyset$,
$\|f\mathbf{1}_{E_n}\|_{X(\cx)}\downarrow 0$ as $n\to\fz$;
see, for instance, \cite[p.\,14, Definition 3.1]{bs88}.
The \emph{associate space} $X'(\cx)$
of a BBF space $X(\cx)$
is defined by setting
\begin{equation*}
X'(\cx):=\lf\{f\in\mathscr{M}
(\cx):\ \|f\|_{X'(\cx)}<\fz\r\},
\end{equation*}
where, for any $f\in\mathscr{M}(\cx)$,
$$\|f\|_{X'(\cx)}:
=\sup\lf\{\|fg\|_{L^1(\cx)}: \ g\in X(\cx),\
\|g\|_{X(\cx)}=1\r\};$$
see \cite[Chapter 1, Section 2]{bs88}
for more details on associate spaces.
It is well known that, if
$X(\cx)$ is a BBF space,
then $X'(\cx)$ is also a BBF space;
see, for instance, \cite[Lemma 2.19]{yhyy21a}
or \cite[Proposition 2.3]{shyy17}.
Next, we recall the concept of the convexification,
which previously appeared, for instance,
in \cite[p.\,53]{lt79} and
\cite[Definition 2.6]{shyy17}.

\begin{definition}\label{decon}
Let $p\in(0,\fz)$.
The \emph{p-convexification}
$X^p(\cx)$ of a BQBF
space $X(\cx)$ is defined by setting
\begin{equation*}
X^p(\cx):=\lf\{f\in\mathscr{M}(\cx):\
\lf\|f\r\|_{X^p(\cx)}:=
\lf\|\lf|f\r|^p\r\|_{X(\cx)}^{1/p}<\fz\r\}.
\end{equation*}
\end{definition}
\begin{remark}\label{recon}
Let $X^p(\cx)$ be as in Definition \ref{decon}.
Observe that $X^p(\cx)$ is also a BQBF space. In what follows,
for any $p\in(0,\fz)$, we always let
\begin{equation}\label{eqsigm}
\sigma_p:=\inf\lf\{C\in[1,\fz):\
\|f+g\|_{X^p(\cx)}\le C[\|f\|_{X^p(\cx)}
+\|g\|_{X^p(\cx)}],\ \forall\,
f,\ g\in X^p(\cx)\r\}.
\end{equation}
\end{remark}

The next proposition provides a
frequently used estimate suitable for
any given BQBF space
and its convexification, which
is an immediate corollary of
the Aoki--Rolewicz theorem
(see \cite{a42,r57}, or \cite[Exercise 1.4.6]{g14}
with a detailed hint on its proof);
we omit the details here.

\begin{proposition}\label{prtz}
Let $X(\cx)$ be a \emph{BQBF} space, $p\in (0,\fz)$, and
$\tz_p:=\frac{\log 2}{\log(2\sigma_p)}$
with $\sigma_p\in[1,\fz)$ as in
\eqref{eqsigm}. Then, for any
$\{f_{j}\}_{j\in\nn}\subset\mathscr{M}(\cx)$,
\begin{equation*}
\lf\|\sum_{j\in\nn}\lf|f_j\r|\r\|_{X^p(\cx)}
\le 4^{1/\tz_p}\lf[\sum_{j\in\nn}\lf\|f_j
\r\|_{X^p(\cx)}^{\tz_p}\r]^{1/\tz_p}.
\end{equation*}
\end{proposition}

Now, we turn to introduce the weak BQBF space on $\cx$; see
\cite[Definition 2.8]{zyyw20} for the
Euclidean space case.

\begin{definition}\label{dewb}
Let $X(\cx)$ be a BQBF space.
Then the \emph{weak ball quasi-Banach function space} (for short, \emph{weak} BQBF \emph{space})
$WX(\cx)$ associated with $X(\cx)$
is defined to be the set of all the $\mu$-measurable
functions $f$ on $\cx$ such that
$$
\|f\|_{WX(\cx)}:=\sup_{\lz\in(0,\fz)}
\lf\{\lz\lf\|\mathbf{1}_{\{x\in\cx:
\ |f(x)|>\lz\}}\r\|_{X(\cx)}
\r\}<\fz.
$$
\end{definition}

The following proposition indicates that
the weak BQBF
space is also a BQBF space.
Its proof is similar to those of
\cite[Lemma 2.13 and Remark 2.9]{zyyw20};
we omit the details.

\begin{proposition}\label{prwb}
Let $X(\cx)$ be a \emph{BQBF} space.
Then the weak \emph{BQBF} space
$WX(\cx)$ is also a \emph{BQBF} space
and, moreover, $X(\cx)$ continuously embeds into $WX(\cx)$,
namely, if $f\in X(\cx)$, then $f\in WX(\cx)$
and $\|f\|_{WX(\cx)}\le\|f\|_{X(\cx)}$.
\end{proposition}

\begin{remark}
Let $p\in(0,\fz)$, $X(\cx)$ be a BQBF
space, $W(X^p)(\cx)$ the weak BQBF space
associated with $X^p(\cx)$, and $(WX)^p(\cx)$
the $p$-convexification of $WX(\cx)$. We claim that
$W(X^p)(\cx)=(WX)^p(\cx)$.
Indeed, for any $f\in\mathscr{M}(\cx)$, we have
\begin{align*}
\|f\|_{W(X^p)(\cx)}
&=\sup_{\lz\in(0,\fz)}\lf\{
\lz\lf\|\mathbf{1}_{\{x\in\cx:
\ |f(x)|>\lz\}}\r\|_{X^p(\cx)}\r\}\\
&=\sup_{\lz\in(0,\fz)}\lf\{
\lz^{1/p}\lf\|\mathbf{1}_{\{x\in\cx:
\ |f(x)|^p>\lz\}}\r\|^{1/p}_{X(\cx)}\r\}
=\|f\|_{(WX)^p(\cx)},
\end{align*}
which proves the above claim.
In what follows, we use
$WX^p(\cx)$ to denote either $W(X^p)(\cx)$
or $(WX)^p(\cx)$.
\end{remark}

Next, we establish the following useful proposition
on the estimate of the truncation of functions in $WX(\cx)$.

\begin{proposition}\label{prin}
Let $p\in(0,\fz)$ and $X(\cx)$ be a \emph{BQBF} space.
\begin{itemize}
\item[\textup{(i)}]
If $p_1\in(0,p)$,
then there exists a positive
constant $C$ such that, for any $\lz\in(0,\fz)$
and $f\in WX^p(\cx)$,
$$
\lf\|f\mathbf{1}_{\{x\in\cx:\
|f(x)|>\lz\}}\r\|^{p_1/p}_{X^{p_1}(\cx)}
\le C\lz^{p_1/p-1}\|f\|_{WX^p(\cx)}.
$$
\item[\textup{(ii)}]
If $p_2\in(p,\fz)$,
then there exists a positive
constant $C$ such that, for any $\lz\in(0,\fz)$
and $f\in WX^p(\cx)$,
$$
\lf\|f\mathbf{1}_{\{x\in\cx:\
|f(x)|\le\lz\}}\r\|^{p_2/p}_{X^{p_2}(\cx)}
\le C\lz^{p_2/p-1}\|f\|_{WX^p(\cx)}.
$$
\end{itemize}
\end{proposition}

\begin{proof}
Let $p$, $X(\cx)$, $p_1$, and $p_2$
be as in the present proposition.

We first prove (i).
From Remark \ref{refz}(iii) with $X(\cx)$ replaced by
$WX^p(\cx)$, Definition \ref{debb}(ii) with
$X(\cx)$ replaced by $X^p(\cx)$,
Proposition \ref{prtz},
the definition of $WX^p(\cx)$,
and $p_1/p<1$, it follows that,
for any $\lz\in(0,\fz)$
and $f\in WX^p(\cx)$,
\begin{align*}
&\lf\|f\mathbf{1}_{\{x\in\cx:\
|f(x)|>\lz\}}\r\|^{p_1/p}_{X^{p_1}(\cx)}\\
&\quad=\lf\|\sum^{\fz}_{j=0}
\lf|f\r|^{p_1/p}\mathbf{1}
_{\{x\in\cx:\ 2^j\lz<|f(x)|\le 2^{j+1}\lz\}}
\r\|_{X^p(\cx)}\\
&\quad\le \lf\|\sum^{\fz}_{j=0}
\lf(2^{j+1}\lz\r)^{p_1/p}\mathbf{1}
_{\{x\in\cx:\ 2^j\lz<|f(x)|\le 2^{j+1}\lz\}}
\r\|_{X^p(\cx)}\\
&\quad\ls \lf[\sum^{\fz}_{j=0}
\lf(2^j\lz\r)^{\tz_pp_1/p}\lf\|\mathbf{1}
_{\{x\in\cx:\ 2^j\lz<|f(x)|\le 2^{j+1}\lz\}}
\r\|^{\tz_p}_{X^p(\cx)}\r]^{1/\tz_p}\\
&\quad\ls \lf[\sum^{\fz}_{j=0}
\lf(2^j\lz\r)^{\tz_p(p_1/p-1)}\r]^{1/\tz_p}\|f\|_{WX^p(\cx)}
\sim \lz^{p_1/p-1}\|f\|_{WX^p(\cx)},
\end{align*}
where $\tz_p$ is as in Proposition \ref{prtz}.
This finishes the proof of (i).

Then we prove (ii).
By Definition \ref{debb}(ii) with $X(\cx)$
replaced by $X^p(\cx)$,
Proposition \ref{prtz}, the definition of
$WX^p(\cx)$, and $p_2/p>1$, we conclude that,
for any $\lz\in(0,\fz)$
and $f\in WX^p(\cx)$,
\begin{align*}
&\lf\|f\mathbf{1}_{\{x\in\cx:\
|f(x)|\le\lz\}}\r\|^{p_2/p}_{X^{p_2}(\cx)}\\
&\quad= \lf\|\sum^{0}_{j=-\fz}\lf|f\r|^{p_2/p}
\mathbf{1}_{\{x\in\cx:
\ 2^{j-1}\lz<|f(x)|\le 2^j\lz\}}\r\|_{X^p(\cx)}\\
&\quad\le \lf\|\sum^{0}_{j=-\fz}\lf(2^j\lz\r)^{p_2/p}
\mathbf{1}_{\{x\in\cx:
\ 2^{j-1}\lz<|f(x)|\le 2^j\lz\}}\r\|_{X^p(\cx)}\\
&\quad\ls \lf[\sum^{0}_{j=-\fz}\lf(2^j\lz\r)^{\tz_pp_2/p}
\lf\|\mathbf{1}_{\{x\in\cx:
\ 2^{j-1}\lz<|f(x)|\le 2^j\lz\}}\r\|
^{\tz_p}_{X^p(\cx)}\r]^{1/\tz_p}\\
&\quad\ls \lf[\sum^{0}_{j=-\fz}\lf(2^j\lz\r)^{
\tz_p(p_2/p-1)}\r]^{1/\tz_p}\|f\|_{WX^p(\cx)}
\sim \lz^{p_2/p-1}\|f\|_{WX^p(\cx)},
\end{align*}
where $\tz_p$ is as in Proposition \ref{prtz}.
This finishes the proof of (ii), and
hence of Proposition \ref{prin}.
\end{proof}

\section{Maximal Function Characterizations\label{sm}}

In this section, we establish various
maximal function characterizations of
the weak BQBF Hardy space associated with
the BQBF space $X(\cx)$.

Now, we recall the definitions of several
maximal functions;
see, for instance, \cite[p.\,2209]{hhllyy}.
Let $\bz$, $\gz\in(0,\eta)$
with $\eta$ as in \eqref{eqwa},
$f\in (\gs)'$, and $\{P_k\}_{k\in\zz}$
be a 1-exp-ATI as in Definition \ref{dee1}.
The \emph{radial maximal function}
$\rmf$ of $f$ is defined by setting,
for any $x\in\cx$,
\begin{equation*}
\rmf(x):=\sup_{k\in\zz}\lf|P_k(f)(x)\r|.
\end{equation*}
The \emph{non-tangential
maximal function} $\ntm$ of $f$
with $\tz\in(0,\fz)$
is defined by setting, for any $x\in\cx$,
\begin{equation*}
\ntm(x)=\sup_{k\in\zz}\sup_{y\in B(x,\tz\dz_0^k)}
\lf|P_k(f)(y)\r|.
\end{equation*}
The \emph{grand maximal function}
$f^*$ of $f$ is defined by setting,
for any $x\in\cx$,
\begin{align*}
f^*(x)&:=f^{(*,\bz,\gz)}(x)&\\
&:=\sup\lf\{|\langle f,\phi\rangle|:
\ \phi\in\gs\ \hbox{and}
\ \|\phi\|_{\cg(x.r,\bz,\gz)}
\le 1\ \hbox{for some}\ r\in(0,\fz)\r\}.
\end{align*}

Let $\tz\in(0,\fz)$.
It is well known that there exists a
positive constant $C$ such that,
for any $f\in(\gs)'$,
\begin{equation}\label{eqms}
\rmf\le\ntm\le Cf^*;
\end{equation}
see, for instance, \cite[(3.1)]{hhllyy}.
Let $X(\cx)$ be a BQBF space
and $WX(\cx)$ the weak BQBF space associated with $X(\cx)$.
From \eqref{eqms}, Proposition \ref{prwb}, and
Definition \ref{debb}(ii) with $X(\cx)$ replaced
by $WX(\cx)$, we deduce that
there exists a
positive constant $C$ such that,
for any $f\in(\gs)'$,
\begin{equation}\label{eqma}
\lf\|\rmf\r\|_{WX(\cx)}
\le\lf\|\ntm\r\|_{WX(\cx)}
\le C \lf\|f^*\r\|_{WX(\cx)}.
\end{equation}
Furthermore, if $X(\cx)$ satisfies some assumptions,
we have the following theorem, which can directly be
deduced from Proposition \ref{prwb} and
\cite[Theorems 3.14--3.16]{yhyy21a};
here we omit the details.

\begin{theorem}\label{thm}
Suppose that $X(\cx)$ is a \emph{BQBF} space
and that there exists a $\underline{p}\in
(\omega/(\omega+\eta),1]$ with
$\oz$ as in \eqref{eqoz} and
$\eta$ as in \eqref{eqwa}
such that, for any
$t\in (0,\underline{p})$, the Hardy--Littlewood
maximal operator $\cm$ is
bounded on $WX^{1/t}(\cx)$.
\begin{enumerate}
\item[\textup{(i)}]
Let $\tz\in (0,\fz)$. If $f\in(\cg^{\eta}_0(\bz,\gz))'$
with $\bz$, $\gz\in(\oz[1/\underline{p}-1],\eta)$, then
\begin{equation*}
\lf\|\rmf\r\|_{WX(\cx)}\sim
\lf\|\ntm\r\|_{WX(\cx)}
\sim \lf\|f^*\r\|_{WX(\cx)},
\end{equation*}
where the positive equivalence
constants are independent of $f$.
\item[\textup{(ii)}]
Let $\bz_1$, $\gz_1$, $\bz_2$, $\gz_2\in
(\oz[1/\underline{p}-1],\eta)$.
If $f\in (\cg^\eta_0(\bz_1,\gz_1))'$ and
$f^{(*,\bz_1,\gz_1)}\in WX(\cx)$, then
$f\in (\cg^\eta_0(\bz_2,\gz_2))'$,
$f^{(*,\bz_2,\gz_2)}\in WX(\cx)$, and
there exists a constant $C\in[1,\fz)$,
independent of $f$, such that
\begin{equation*}
C^{-1}\lf\|f^{(*,\bz_1,\gz_1)}\r\|_{WX(\cx)}
\le \lf\|f^{(*,\bz_2,\gz_2)}\r\|_{WX(\cx)}
\le C\lf\|f^{(*,\bz_1,\gz_1)}\r\|_{WX(\cx)}.
\end{equation*}
\end{enumerate}
\end{theorem}
Let $\bz$, $\gz\in(0,\eta)$ with $\eta$
as in \eqref{eqwa}, and $X(\cx)$
be a given BQBF space.
Recall that, in \cite[(3.7)]{yhyy21a},
the \emph{Hardy space}
$H_X(\cx)$ associated with $X(\cx)$
is defined by setting
\begin{equation*}
H_X(\cx):=\lf\{f\in(\gs)':\
\lf\|f\r\|_{H_X(\cx)}:=\lf\|f^*\r\|_{X(\cx)}<\fz\r\}.
\end{equation*}
Then we introduce the following
weak Hardy space
associated with $X(\cx)$.
\begin{definition}\label{de}
Let $X(\cx)$ be a BQBF
space and $\bz,\,\gamma\in(0,\eta)$ with
$\eta$ as in \eqref{eqwa}. The \emph{weak Hardy
space $WH_X(\cx)$} associated with $X(\cx)$
is defined by setting
$$
WH_X(\cx):=\lf\{f\in(\gs)':\ \|f\|_{WH_X(\cx)}
:=\lf\|f^*\r\|_{WX(\cx)}<\fz\r\}.
$$
\end{definition}
\begin{remark}\label{rede}
\begin{itemize}
\item[(i)] Let $WH_X(\cx)$ and $(\gs)'$
be as in Definition \ref{de}.
By Proposition \ref{prwb} and \cite[Lemma 3.10]{yhyy21a},
we find that $WH_X(\cx)$ continuously embeds into $(\gs)'$.
\item[(ii)]
Let $X(\cx)$ and $\underline{p}$ be as in Theorem \ref{thm}
and $\bz$, $\gz\in(\oz[1/\underline{p}-1],\eta)$.
From Theorem \ref{thm},
it follows that the weak Hardy space $WH_X(\cx)$ is
independent of the choice of $(\gs)'$.
\end{itemize}
\end{remark}
Embedding $WX(\cx)$ into a weighted Lebesgue space
and borrowing some ideas from the proof of
\cite[Theorem 3.12]{yhyy21a},
we can establish the following
relation between the weak BQBF space $WX(\cx)$ and
the weak Hardy space $WH_X(\cx)$.
\begin{theorem}\label{thmw}
Suppose that $X(\cx)$ is a \emph{BQBF}
space and there exists a $t\in (1,\fz)$
such that the Hardy--Littlewood maximal
operator $\cm$ is bounded on $WX^{1/t}(\cx)$.
Let $\bz$, $\gz\in (0,\eta)$
with $\eta$ as in \eqref{eqwa}.
\begin{enumerate}
\item[\textup{(i)}]
If $f\in WX(\cx)$,
then $f\in (\gs)'\cap WH_X(\cx)$ and
$\|f\|_{WH_X(\cx)}\le C\|f\|_{WX(\cx)},$
where $C$ is a positive constant independent of $f$.
\item[\textup{(ii)}]
If $f\in(\gs)'$
satisfies $\rmf\in WX(\cx)$,
then there exists a function
$\widetilde{f}\in WX(\cx)$ such that
$f=\widetilde{f}$ in $(\gs)'$, and
$\|\widetilde{f}\|_{WX(\cx)}\le \|\rmf\|_{WX(\cx)}$.
\end{enumerate}
Furthermore, $WH_X(\cx)=WX(\cx)$ with equivalent quasi-norms.
\end{theorem}

\begin{proof}
By Proposition \ref{prwb} and
\cite[Theorem 3.12]{yhyy21a}, we obtain (i).
If (ii) holds true, by (i), (ii), and
\eqref{eqma}, we conclude that
$WH_X(\cx)=WX(\cx)$ with equivalent quasi-norms.
Thus, to complete the proof of the present theorem,
it remains to show (ii).
To this end, let $X(\cx)$, $t$, $\bz$, and
$\gz$ be as in the present theorem.

Let $x_0\in\cx$ be the fixed basepoint and,
for any $y\in\cx$, let $w(y):=G_{1,\gz}(x_0,y)$
be as in \eqref{eqg} with $r=1$ and $x$ replaced by $x_0$. The
space $L^t_w(\cx)$ is defined as in \eqref{eqweight}
with $p$ replaced by $t$.
We claim that
$WX(\cx)$ continuously embeds into $L^t_w(\cx)$.
Indeed, from \eqref{eqoz}, the definition of $\cm$,
$\gz>0$, Proposition \ref{prwb},
Definition \ref{debb}(ii) with $X(\cx)$ replaced by $WX(\cx)$,
and the boundedness of $\cm$ on $WX^{1/t}(\cx)$,
it follows that, for any $g\in WX(\cx)$,
\begin{align*}
\|g\|_{L^t_w(\cx)}
&\ls \lf[\frac{1}{V_1(x_0)}
\int_{B(x_0,1)}|g(z)|^t\,d\mu(z)
+\sum^\fz_{k=1}\frac{2^{-k\gz}}{V_{2^k}(x_0)}
\int_{B(x_0,2^k)\setminus B(x_0,2^{k-1})}
|g(z)|^t\,d\mu(z)\r]^{1/t}\\
&\ls\lf[\sum^\fz_{k=0}2^{-k\gz}\inf_{x\in B(x_0,1)}
\lf\{\cm\lf(|g|^t\r)(x)\r\}\r]^{1/t}
\ls\inf_{x\in B(x_0,1)}
\lf\{\lf[\cm\lf(|g|^t\r)(x)\r]^{1/t}\r\}\\
&\ls \lf\|\mathbf{1}_{B(x_0,1)}\r\|^{-1}_{WX(\cx)}
\lf\|\cm\lf(|g|^t\r)\r\|^{1/t}_{WX^{1/t}(\cx)}
\ls \|g\|_{WX(\cx)},
\end{align*}
which proves the above claim.

Let $1/t+1/t'=1$ and $L^{t'}_{w^{1-t'}}(\cx)$ be
as in \eqref{eqweight} with $p$ and
$w$ replaced, respectively, by
$t'$ and $w^{1-t'}$. Observe that the dual space
of $L^{t'}_{w^{1-t'}}(\cx)$ is $L^t_w(\cx)$; more precisely,
$\mathcal{L}$ is a bounded linear functional on
$L^{t'}_{w^{1-t'}}(\cx)$ if and
only if there exists a unique $f_{\mathcal{L}}\in L^t_w(\cx)$
such that, for any $g\in L^{t'}_{w^{1-t'}}(\cx)$,
\begin{equation}\label{eqdu}
\mathcal{L}(g)=\int_\cx f_{\mathcal{L}}(z)g(z)\,d\mu(z).
\end{equation}

Let $f\in (\gs)'$ and $\rmf\in WX(\cx)$.
Suppose that $\{P_k\}_{k\in\zz}$ is a
1-exp-ATI as in Definition \ref{dee1}.
By the definition of $\rmf$ and
the embedding $WX(\cx)\subset L^t_w(\cx)$,
we conclude that
$$
\sup_{k\in\zz}\lf\|P_k(f)\r\|_{L^t_w(\cx)}
\le\lf\|\rmf\r\|_{L^t_w(\cx)}
\ls \lf\|\rmf\r\|_{WX(\cx)}.
$$
From this, the Alaoglu theorem
(see, for instance, \cite[Theorem 3.17]{r91}),
and \eqref{eqdu},
it follows that there exist some $\widetilde{f}
\in L^t_w(\cx)$ and a sequence $\{k_j\}_{j\in\nn}
\subset\nn$ such that $k_j\to\fz$ as $j\to\fz$, and,
for any $g\in L^{t'}_{w^{1-t'}}(\cx)$,
\begin{equation}\label{tmwe1}
\int_{\cx}\widetilde{f}(z)g(z)\,d\mu(z)
=\lim_{j\to\fz}\int_{\cx}
P_{k_j}(f)(z)g(z)\,d\mu(z).
\end{equation}
By this, the Lebesgue differentiation theorem
(see, for instance, \cite[Theorem 1.8]{h01}),
and the definition of $\rmf$,
we find that, for $\mu$-almost every point $x\in\cx$,
\begin{align*}
\lf|\widetilde{f}(x)\r|
&=\lim_{r\to 0}\frac{1}{V_r(x)}\int_{B(x,r)}
\widetilde{f}(z)h(z)\mathbf{1}_{B(x,r)}(z)\,d\mu(z)\\
&=\lim_{r\to 0}\lim_{j\to\fz}
\frac{1}{V_r(x)}\int_{B(x,r)}
P_{k_j}(f)(z)h(z)\mathbf{1}_{B(x,r)}(z)\,d\mu(z)\\
&\le\lim_{r\to 0}
\frac{1}{V_r(x)}\int_{B(x,r)}
\rmf(z)\,d\mu(z)=\rmf(x),
\end{align*}
where, for any $z\in\cx$,
\begin{align*}
h(z):=
\begin{cases}
\displaystyle\frac{|\widetilde{f}(z)|}
{\widetilde{f}(z)}\displaystyle\ \
&\text{if}\ \widetilde{f}(z)\neq 0,\\
\displaystyle 0\ \
&\text{if}\  \widetilde{f}(z)= 0.
\end{cases}
\end{align*}
Using this, Proposition \ref{prwb},
and Definition \ref{debb}(ii) with $X(\cx)$
replaced by $WX(\cx)$,
we conclude that
$\|\widetilde{f}\|_{WX(\cx)}\le
\|\rmf\|_{WX(\cx)}$.

By \cite[Lemma 3.2(iii)]{hhllyy},
Proposition \ref{prdc} combined
with Remark \ref{rewan}, and
\eqref{tmwe1} with $g=\psi$,
we find that,
for any $\psi\in \cg_{\mathrm{b}}(\eta,\eta)$,
\begin{align*}
\lf\langle f,\psi\r\rangle
&=\lim_{j\to\fz}\lf\langle P_{k_j}(f),\psi\r\rangle\\
&=\lim_{j\to\fz}\int_\cx P_{k_j}(f)(z)\psi(z)\,d\mu(z)
=\int_\cx\widetilde{f}(z)\psi(z)\,d\mu(z).
\end{align*}
Using this, Proposition \ref{prd},
and a standard density argument, we conclude that
$f=\widetilde{f}$ in $(\gs)'$.
This finishes the proof of (ii),
and hence of Theorem \ref{thmw}.
\end{proof}

\begin{remark}\label{remax}
\begin{itemize}
\item[(i)]
If $\cx=\rn$, then
the corresponding result of Theorem \ref{thm}
can be found in \cite[Theorem 3.2]{zyyw20}.
In Theorem \ref{thm}, we assume that
$\underline{p}$ has a lower bound
$\oz/(\oz+\eta)$.
Since the regularity condition of test functions on $\cx$
is at most $\eta$ and there exists no ``derivatives''
and ``polynomials'' on $\cx$, it is well known that
$\oz/(\oz+\eta)$ is the known best lower bound
of $\underline{p}$ in Theorem \ref{thm}
(see, for instance, \cite{hhllyy}).

\item[(ii)]
If $\cx=\rn$, then
the corresponding result
of Theorem \ref{thmw}
can be found in \cite[Theorem 3.4]{zyyw20}
and, moreover,
we find that the assumptions needed
in Theorem \ref{thmw} coincide with
those needed in \cite[Theorem 3.4]{zyyw20}.
\end{itemize}
\end{remark}
\section{Two Crucial Assumptions\label{sa}}
The purpose of this section is to give two basic
assumptions and the related conclusions induced by them,
which play pivotal roles in the later sections,
including the atomic characterization and
the real interpolation of $WH_X(\cx)$, and the boundedness of
Calder\'{o}n--Zygmund operators from $H_X(\cx)$ to $WH_X(\cx)$.
We list these assumptions as follows.
\begin{assumption}\label{asfs}
Let $X(\cx)$ be a BQBF space.
Assume that there exists a positive constant
$p_-$ such that, for any given $t\in (0,p_-)$ and $s\in(1,\fz)$,
there exists a positive constant $C_{(t,s)}$ such that,
for any sequence $\{f_j\}_{j\in\nn}\subset \mathscr{M}(\cx)$,
\begin{equation}\label{eqbfs}
\lf\|\lf\{\sum_{j\in\nn}\lf[\cm\lf(f_j\r)\r]^s\r\}
^{1/s}\r\|_{X^{1/t}(\cx)}
\le C_{(t,s)}\lf\|\lf(\sum_{j\in\nn}\lf|f_j\r|^s\r)
^{1/s}\r\|_{X^{1/t}(\cx)}.
\end{equation}
\end{assumption}
\begin{assumption}\label{asas}
Let $X(\cx)$ be a BQBF space.
Assume that there exist constants $s_0\in(0,\fz)$
and $p_0\in(s_0,\fz)$ such that
$X^{1/s_0}(\cx)$ is a BBF space and the
Hardy--Littlewood maximal operator $\cm$ is
bounded on the $\frac{1}{(p_0/s_0)'}$-convexification
of the associate space $(X^{1/s_0})'(\cx)$,
where $\frac{1}{(p_0/s_0)'}
+\frac{1}{(p_0/s_0)}=1$.
\end{assumption}

\begin{remark}\label{reas}
Let $p\in (0,\fz)$. By \cite[Theorem 1.2]{gly09},
we find that, if $X(\cx)=L^p(\cx)$ and $p_-=p$
in Assumption \ref{asfs},
then Assumption \ref{asfs} holds true and
\eqref{eqbfs} becomes the notable
Fefferman--Stein vector-valued maximal inequality.
By  \cite[p.\,10, Theorem 2.5]{bs88}
and the boundedness of $\cm$ on
$L^q(\cx)$ for any given $q\in(1,\fz]$,
we find that,
if $X(\cx)=L^p(\cx)$, $s_0\in(0,p]$,
and $p_0\in(p,\fz)$ in Assumption \ref{asas},
then Assumption \ref{asas} holds true.
Moreover, see Section \ref{sap} for
more function spaces
satisfying Assumptions \ref{asfs}
and \ref{asas}.
\end{remark}

The next theorem shows
that, if Assumption \ref{asfs} holds true,
then the following weak-type
Fefferman--Stein vector-valued maximal
inequality is also true for weak BQBF spaces.
\begin{theorem}\label{thwfs}
Let $X(\cx)$ be a \emph{BQBF} space
satisfying Assumption \ref{asfs} with $p_-\in (0,\fz)$.
Then, for any given $t\in (0,p_-)$ and $s\in(1,\fz)$,
there exists a positive constant $C$ such that,
for any sequence $\{f_j\}_{j\in\nn}\subset \mathscr{M}(\cx)$,
\begin{equation*}
\lf\|\lf\{\sum_{j\in\nn}\lf[\cm\lf(f_j\r)\r]^s\r\}
^{1/s}\r\|_{WX^{1/t}(\cx)}
\le C\lf\|\lf(\sum_{j\in\nn}\lf|f_j\r|^s\r)
^{1/s}\r\|_{WX^{1/t}(\cx)}.
\end{equation*}
\end{theorem}
To prove this theorem, we first establish an interpolation
result which might have independent interest.
Recall that an operator $T:\mathscr{M}(\cx)
\to\mathscr{M}(\cx)$ is said to be \emph{sublinear}
if, for any $\lz\in\cc$ and
$f$, $g\in\mathscr{M}(\cx)$,
$$
|T(\lz f)|=|\lz||T(f)|
\quad\mathrm{and}\quad
|T(f+g)|\le |T(f)|+|T(g)|.
$$
For any sets $X_1$,
$X_2\subset\mathscr{M}(\cx)$, the symbol $X_1+X_2$
denotes the set of all the functions $f\in\mathscr{M}(\cx)$
satisfying that there exist functions $f_1\in X_1$
and $f_2\in X_2$ such that $f=f_1+f_2$.
\begin{lemma}\label{lewi}
Let $0<p_1<p_2<\fz$, $X(\cx)$ be a \emph{BQBF} space,
and, for any $i\in\{1,2\}$,
$X^{p_i}(\cx)$ the $p_i$-convexification of $X(\cx)$.
Suppose that $T$ is a sublinear operator
from $X^{p_1}(\cx)+X^{p_2}(\cx)$ to
$\mathscr{M}(\cx)$ and,
for any given $i\in\{1,2\}$,
there exists a positive constant $C_i$ such that,
for any $f\in X^{p_i}(\cx)$,
\begin{equation}\label{eqt}
\lf\|T(f)\r\|_{WX^{p_i}(\cx)}
\le C_i\|f\|_{X^{p_i}(\cx)}.
\end{equation}
If $p\in(p_1,p_2)$, then $T$
is bounded on $WX^p(\cx)$ and its operator norm
$$
\|T\|_{WX^p(\cx)\to WX^p(\cx)}\le C
\lf[(C_1)^{p_1/p}+(C_2)^{p_2/p}\r],
$$
where $C$ is a positive constant depending
only on $p$, $p_1$, $p_2$, and $\sigma_p$ with
$\sigma_p\in[1,\fz)$ as in \eqref{eqsigm}.
\end{lemma}

\begin{proof}
Let $p_1$, $p_2$, $X(\cx)$, $T$, $C_1$, and $C_2$
be as in the present lemma, $p\in(p_1,p_2)$,
and $\sigma_p$ as in \eqref{eqsigm}.
To prove the present lemma,
it suffices to show that, for any
$f\in WX^p(\cx)$ and $\lz\in(0,\fz)$,
\begin{equation}\label{eqt0}
\lz\lf\|\mathbf{1}_{\{x\in\cx:\ |T(f)(x)|
>\lz\}}\r\|_{X^p(\cx)}
\ls \lf[(C_1)^{p_1/p}+(C_2)^{p_2/p}\r]
\|f\|_{WX^p(\cx)}
\end{equation}
with the implicit positive constant depending
only on $p$, $p_1$, $p_2$, and $\sigma_p$.
Now, for any $f\in WX^p(\cx)$ and $\lz\in(0,\fz)$,
we decompose
\begin{equation*}
f=f\mathbf{1}_{\{x\in\cx:\ |f(x)|>\lz\}}
+f\mathbf{1}_{\{x\in\cx:\ |f(x)|\le\lz\}}
=:f_{1,\lambda}+f_{2,\lambda}.
\end{equation*}
By the proof of Proposition \ref{prin},
we find that, for any given
$i\in\{1,2\}$, and any
$f\in WX^p(\cx)$ and $\lz\in(0,\fz)$,
\begin{equation}\label{eqt1}
\lf\|f_{i,\lambda}\r\|^{p_i/p}_{X^{p_i}(\cx)}
\ls \lz^{p_i/p-1}\|f\|_{WX^p(\cx)},
\end{equation}
where the implicit positive constant depends only
on $p$, $p_1$, $p_2$, and $\sigma_p$.

From the sublinearity of $T$, Definition \ref{debb}(ii)
with $X(\cx)$ replaced by $X^p(\cx)$,
Remark \ref{recon}, the definitions of
$WX^{p_i}(\cx)$ with $i\in\{1,2\}$,
\eqref{eqt}, and \eqref{eqt1}, we deduce that,
for any $f\in WX^p(\cx)$ and $\lz\in(0,\fz)$,
\begin{align*}
&\lz\lf\|\mathbf{1}_{\{x\in\cx:
\ |T(f)(x)|>\lz\}}\r\|_{X^p(\cx)}\\
&\quad\le \lz\lf\|\sum^{2}_{i=1}
\mathbf{1}_{\{x\in\cx:\ |
T(f_{i,\lz})(x)|>\lz/2\}}\r\|_{X^p(\cx)}
\ls \sum^{2}_{i=1}
\lz\lf\|\mathbf{1}_{\{x\in\cx:\|
T(f_{i,\lz})(x)|>\lz/2\}}\r\|_{X^p(\cx)}\\
&\quad\ls \sum^{2}_{i=1}\lz^{1-p_i/p}
\lf\|T(f_{i,\lz})\r\|^{p_i/p}_{WX^{p_i}(\cx)}
\ls \sum^{2}_{i=1}\lz^{1-p_i/p}
(C_i)^{p_i/p}\lf\|f_{i,\lz}\r\|^{p_i/p}_{X^{p_i}(\cx)}\\
&\quad\ls \lf[(C_1)^{p_1/p}+(C_2)^{p_2/p}\r]\|f\|_{WX^p(\cx)},
\end{align*}
where all the implicit positive constants
depend only on $p$, $p_1$, $p_2$, and $\sigma_p$.
This finishes the proof of \eqref{eqt0}, and hence
of Lemma \ref{lewi}.
\end{proof}

\begin{remark}
Let all the symbols be as in Lemma \ref{lewi}.
If $X(\cx)=L^1(\cx)$ in Lemma \ref{lewi},
by the Marcinkiewicz interpolation theorem
(see, for instance, \cite[Theorem 1.3.2]{g14}),
we then conclude that, for any given $p\in(p_1,p_2)$,
$T$ is bounded on $L^p(\cx)$.
It is unclear whether or not the Marcinkiewicz interpolation
theorem still holds true for an arbitrary BQBF space $X(\cx)$ because of the lack
of the explicit expression of $\|\cdot\|_{X(\cx)}$.
\end{remark}

\begin{proof}[Proof of Theorem \ref{thwfs}]
Let $X(\cx)$ and $p_-$ be as in the present theorem,
$t\in(0,p_-)$, $t_1\in (0,t)$,
$t_2\in(t,p_-)$, $s\in(1,\fz)$,
and $\{f_j\}_{j\in\nn}\subset \mathscr{M}(\cx)$.
Without loss of generality, we may assume that
$(\sum_{j\in\nn}|f_j|^s)^{1/s}\in WX^{1/t}(\cx)$.

Now, for any $f\in \mathscr{M}(\cx)$, let
$$
T(f):=\lf\{\sum_{j\in\nn}\lf[\cm\lf(fg_j\r)\r]^s
\r\}^{1/s},
$$
where, for any $j\in\nn$ and $x\in\cx$,
\begin{align*}
g_j(x):=
\begin{cases}
\displaystyle\frac{f_j(x)}
{[\sum_{j\in\nn}|f_j(x)|^s]^{1/s}}\ \
&\text{if}\ {\displaystyle\sum_{j\in\nn}}|f_j(x)|^s\neq 0,
\\
\displaystyle 0\ \
&\text{if}\  {\displaystyle\sum_{j\in\nn}}|f_j(x)|^s= 0.
\end{cases}
\end{align*}
From Proposition \ref{prwb} with $X(\cx)$ replaced
by $X^{1/t_i}(\cx)$, Assumption \ref{asfs}
with $t$ replaced by $t_i\in(0,p_-)$,
and the definition of $\{g_j\}_{j\in\nn}$,
we deduce that, for any $i\in\{1,2\}$
and $f\in X^{1/t_i}(\cx)$,
\begin{align*}
\|T(f)\|_{WX^{1/t_i}(\cx)}
&\le\|T(f)\|_{X^{1/t_i}(\cx)}\\
&\le C_{(t_i,s)}\lf\|\lf(\sum_{j\in\nn}
\lf|fg_j\r|^s\r)^{1/s}\r\|_{X^{1/t_i}(\cx)}
\le C_{(t_i,s)}\|f\|_{X^{1/t_i}(\cx)},
\end{align*}
where $C_{(t_i,s)}$ is as in \eqref{eqbfs} with
$t$ replaced by $t_i$.
Using this and Lemma \ref{lewi} with $p=1/t$
and $p_i=1/t_i$ for any $i\in\{1,2\}$, we find that
$T$ is bounded on $WX^{1/t}(\cx)$, and its operator
norm is independent of $\{f_j\}_{j\in\nn}$.
By this and the definition of $T$, we have
\begin{align*}
&\lf\|\lf\{\sum_{j\in\nn}\lf[
\cm\lf(f_j\r)\r]^s\r\}^{1/s}
\r\|_{WX^{1/t}(\cx)}\\
&\quad=\lf\|T\lf(\lf(\sum_{j\in\nn}
\lf|f_j\r|^s\r)^{1/s}\r)\r\|_{WX^{1/t}(\cx)}
\ls \lf\|\lf(\sum_{j\in\nn}
\lf|f_j\r|^s\r)^{1/s}\r\|_{WX^{1/t}(\cx)}
\end{align*}
with the implicit positive constant
independent of $\{f_j\}_{j\in\nn}$.
This finishes the proof of Theorem \ref{thwfs}.
\end{proof}

\begin{remark}\label{rem}
Let $X(\cx)$ be a BQBF
space satisfying Assumption \ref{asfs}
with $p_-\in(0,\fz)$. Then, for any $t\in(0,p_-)$,
$\cm$ is bounded on $X^{1/t}(\cx)$.
Moreover, by Theorem \ref{thwfs},
we conclude that, for any $t\in(0,p_-)$,
$\cm$ is bounded on $WX^{1/t}(\cx)$.
\end{remark}

Assumption \ref{asfs} also implies the
following useful proposition.

\begin{proposition}\label{prfs}
Let $X(\cx)$ and $p_-$ be as in
Assumption \ref{asfs}.
\begin{itemize}
\item[\textup{(i)}]
Let $t\in (0,\fz)$ and
$s\in(\max\{1,t/p_-\},\fz)$. Then there exists
a positive constant $C$ such that,
for any $\tau\in[1,\fz)$, any
$\{\lz_j\}_{j\in\nn}\subset (0,\fz)$,
and any sequence $\{B_j\}_{j\in\nn}$ of balls,
\begin{align*}
\lf\|\sum_{j\in\nn}\lz_j\mathbf{1}_{
\tau B_j}\r\|_{X^{1/t}(\cx)}
\le C\tau^{\oz s}\lf\|\sum_{j\in\nn}
\lz_j\mathbf{1}_{B_j}\r\|_{X^{1/t}(\cx)},
\end{align*}
where $\oz$ is as in \eqref{eqoz}.
\item[\textup{(ii)}]
Then $\|\mathbf{1}_{\cx}\|_{X(\cx)}=\fz$.
\end{itemize}
\end{proposition}

\begin{proof}
We first prove (i).
Let $X(\cx)$, $p_-$, $t$, $s$, $\tau$,
$\{\lz_j\}_{j\in\nn}$, and
$\{B_j\}_{j\in\nn}$ be as in (i).
By \eqref{eqoz}, we find that,
for any $j\in\nn$ and $x\in \tau B_j$,
$$
\cm\lf(\mathbf{1}_{B_j}\r)(x)\ge
\frac{\mu(B_j)}{\mu(\tau B_j)}\ge
\lf[C_{(\mu)}\r]^{-1}\tau^{-\oz}
$$
with $C_{(\mu)}$ and $\oz$ as in \eqref{eqoz}.
From this, Definition \ref{debb}(ii) with
$X(\cx)$ replaced by $X^{1/t}(\cx)$,
and Assumption \ref{asfs} with $t$
replaced by $t/s\in(0,p_-)$, we deduce that
\begin{align*}
\lf\|\sum_{j\in\nn}\lz_j\mathbf{1}_{
\tau B_j}\r\|_{X^{1/t}(\cx)}
\ls \lf\|\sum_{j\in\nn}\lz_j\lf[
\tau^{\oz}\cm\lf(\mathbf{1}_{B_j}\r)\r]
^s\r\|_{X^{1/t}(\cx)}
\ls\tau^{\oz s}\lf\|\sum_{j\in\nn}\lz_j
\mathbf{1}_{B_j}\r\|_{X^{1/t}(\cx)},
\end{align*}
where all the implicit positive constants
are independent of $\tau$, $\{\lz_j\}_{j\in\nn}$,
and $\{B_j\}_{j\in\nn}$.
This finishes the proof of (i).

Now, we prove (ii). Let $x_0\in\cx$ be the fixed
basepoint and
$r_0\in(0,\fz)$. From \cite[Lemma 2.23]{yhyy21a},
it follows that
there exists a sequence
$\{r_j\}_{j\in\nn}\subset (0,\fz)$
such that, for any $j\in\nn$,
$2r_{j-1}\le r_j$ and
\begin{equation}\label{eqadmiss}
2V_{r_{j-1}}(x_0)
\le V_{r_j}(x_0)
\le 2C_{(\mu)}V_{r_{j-1}}(x_0)
\end{equation}
with $C_{(\mu)}$ as in \eqref{eqoz}.
By the definition of $\cm$ and
 the first inequality of \eqref{eqadmiss},
we conclude that, for any
$x\in B(x_0,r_0)$ and $s\in (1,\fz)$,
\begin{align}\label{ala}
&\lf\{\sum_{j\in\nn}\lf[\cm\lf(\mathbf{1}_{B(x_0,r_j)
\setminus B(x_0,r_{j-1})}\r)(x)\r]^s\r\}^{1/s}\\
&\quad\ge \lf\{\sum_{j\in\nn}
\lf[\frac{\mu(B(x_0,r_j)\setminus B(x_0,r_{j-1}))}
{V_{r_j}(x_0)}\r]^s\r\}^{1/s}
\ge \lf[\sum_{j\in\nn}
\lf(\frac{1}{2}\r)^s\r]^{1/s}=\fz.\notag
\end{align}
Next, we choose a $t\in(0,p_-)$ and
an $s\in(1,\fz)$.
From Definition \ref{debb}(ii),
Assumption \ref{asfs}, \eqref{ala}, and
$\|\mathbf{1}_{B(x_0,r_0)}\|_{X^{1/t}(\cx)}>0$,
we deduce that
\begin{align*}
\|\mathbf{1}_{\cx}\|_{X(\cx)}
&\ge\lf\|\lf\{\sum_{j\in\nn}
\lf[\mathbf{1}_{B(x_0,r_j)
\setminus B(x_0,r_{j-1})}\r]^s\r\}^{1/s}
\r\|^{1/t}_{X^{1/t}(\cx)}\\
&\gtrsim \lf\|\lf\{\sum_{j\in\nn}
\lf[\cm\lf(\mathbf{1}_{B(x_0,r_j)
\setminus B(x_0,r_{j-1})}\r)\r]^s\r\}^{1/s}
\r\|^{1/t}_{X^{1/t}(\cx)}\\
&\gtrsim \lf\|\lf[\sum_{j\in\nn}
\lf(\frac{1}{2}\r)^s\r]^{1/s}
\mathbf{1}_{B(x_0,r_0)}
\r\|^{1/t}_{X^{1/t}(\cx)}=\fz.
\end{align*}
This finishes the proof of (ii), and
hence of Proposition \ref{prfs}.
\end{proof}

\begin{remark}
Observe that $L^\fz(\cx)$ is a BBF space
and $(L^\fz)^{1/t}(\cx)=L^\fz(\cx)$
for any given $t\in (0,\fz)$.
From Proposition \ref{prfs}(ii) and
$\|\mathbf{1}_{\cx}\|_{L^\fz(\cx)}=1$, we deduce
that Assumption \ref{asfs} does not
hold true when $X(\cx)=L^\fz(\cx)$; namely,
the Fefferman--Stein vector-valued maximal
inequality does not hold true for $L^\fz(\cx)$.
Otherwise, by Proposition \ref{prfs}(ii),
we have $\|\mathbf{1}_{\cx}\|_{L^\fz(\cx)}=\fz$,
which contradicts the fact that
$\|\mathbf{1}_\cx\|_{L^\fz(\cx)}=1$.
\end{remark}

Using Assumption \ref{asas}, we can
obtain the useful estimate
in the following proposition which
is just \cite[Lemma 4.3]{yhyy21a}; see also
 \cite[Lemma 4.8]{zyyw20} for the proof in the
Euclidean space case.

\begin{proposition}\label{pras}
Let $X(\cx)$ be a \emph{BQBF}
space satisfying Assumption \ref{asas}
with $s_0\in(0,\fz)$ and $p_0\in(s_0,\fz)$.
Then there exists a positive constant $C$ such that,
for any sequence $\{\lambda_j\}_{j\in\nn}\subset (0,\fz)$
and any sequence
$\{a_j\}_{j\in\nn}\subset L^{p_0}(\cx)$
supported, respectively, in balls
$\{B_j\}_{j\in\nn}$
with $\|a_j\|_{L^{p_0}(\cx)}\le [\mu(B_j)]^{1/p_0}$
for any $j\in\nn$,
\begin{equation*}
\lf\|\sum_{j\in\nn}\lf|\lambda_j
a_j\r|^{s_0}\r\|_{X^{1/s_0}(\cx)}
\le C\lf\|\sum_{j\in\nn}\lf|\lambda_j
\mathbf{1}_{B_j}\r|^{s_0}\r\|_{X^{1/s_0}(\cx)}.
\end{equation*}
\end{proposition}

\section{Atomic Characterizations\label{sat}}
In this section, we establish the atomic characterization
of the weak Hardy space $WH_X(\cx)$ associated with
a given BQBF space $X(\cx)$.
To this end, we first recall the definition of atoms
from \cite[Definition 4.1]{yhyy21a}.
\begin{definition}\label{deat}
Let $X(\cx)$ be a BQBF space  and $q\in[1,\fz]$.
A $\mu$-measurable function $a$ on $\cx$ is
called an \emph{$(X(\cx),q)$-atom} if
there exists a ball $B\subset\cx$ such that
\begin{enumerate}
\item[(i)]
$\supp(a):=\{x\in\cx:\,a(x)\neq 0\}\subset B$;
\item[(ii)]
$\|a\|_{L^q(\cx)}\le [\mu(B)]^{1/q}
\|\mathbf{1}_B\|^{-1}_{X(\cx)}$;
\item[(iii)]
$\int_\cx a(x)\,d\mu(x)=0.$
\end{enumerate}
\end{definition}

The atomic characterization of $WH_X(\cx)$
mainly includes the reconstruction theorem and
the decomposition theorem, which are stated,
respectively, in Subsections \ref{ssc} and \ref{ssd}.
\subsection{Reconstruction \label{ssc}}
Now, we state the reconstruction theorem.

\begin{theorem}\label{tham}
Let $\oz$ be as in \eqref{eqoz},
$\eta$ as in \eqref{eqwa}, and $X(\cx)$
a \emph{BQBF} space.
Suppose that $X(\cx)$ satisfies Assumption \ref{asfs}
with $p_-\in(\oz/(\oz+\eta),\fz)$, and Assumption
\ref{asas} with $s_0\in(0,1)$ and
$p_0\in(s_0,\fz)$.
Fix $c\in (0,1]$, $A,\ \widetilde{A}\in(0,\fz)$,
$q\in(p_0,\fz]\cap[1,\fz]$, and
$\bz$, $\gz\in(\oz[1/\min\{1,p_-\}-1],\eta)$.
Let $\{a_{i,j}\}_{i\in\zz,\,j\in\nn}$
be a sequence of $(X(\cx),q)$-atoms supported,
respectively, in balls
$\{B_{i,j}\}_{i\in\zz,\,j\in\nn}$ such that
$\sum_{j\in\nn}\mathbf{1}_{cB_{i,j}}\leq A$
for any $i\in\zz$,
and $\sup_{i\in\zz}\{2^i\|\sum_{j\in\nn}
\mathbf{1}_{B_{i,j}}
\|_{X(\cx)}\}<\fz$.
Then
$$f:=\sum_{i\in\zz}\sum_{j\in\nn}\widetilde{A}2^i
\|\mathbf{1}_{B_{i,j}}\|_{X(\cx)}a_{i,j}$$
converges in $(\gs)'$,
$f\in WH_X(\cx)$, and
$$
\lf\|f\r\|_{WH_X(\cx)}\ls
\sup_{i\in\zz}\lf\{2^i\lf\|\sum_{j\in\nn}
\mathbf{1}_{B_{i,j}}\r\|_{X(\cx)}\r\}
$$
with the implicit positive constant independent of
both $\{a_{i,j}\}_{i\in\zz,j\in\nn}$ and
$\{B_{i,j}\}_{i\in\zz,j\in\nn}$.
\end{theorem}

To prove Theorem \ref{tham}, we need the following lemma.

\begin{lemma}\label{lea}
Let $A_0$ be as in \eqref{eqA0},
$\oz$ as in \eqref{eqoz}, and
$X(\cx)$ a \emph{BQBF} space.
Then there exists a positive constant $C$ such that,
for any given $q\in[1,\fz]$, and any $(X(\cx),q)$-atom $a$
supported in a ball $B$, and any $x\in\cx$,
\begin{equation*}
a^*(x)\le C\cm(a)(x)\mathbf{1}_{2A_0B}(x)
+C\lf\|\mathbf{1}_B\r\|^{-1}_{X(\cx)}\lf[
\cm(\mathbf{1}_B)(x)\r]^{\frac{\oz+\bz}{\oz}}
\mathbf{1}_{(2A_0B)^{\complement}}(x),
\end{equation*}
where $a$ is regarded as a distribution on
$\gs$ with $\bz,\ \gamma\in(0,\eta)$.
\end{lemma}
Lemma \ref{lea} was proved in
\cite[(4.2) and (4.3)]{yhyy21a}.
Borrowing some ideas from the proof
of \cite[Theorem 4.7]{zyyw20},
and using the Aoki--Rolewicz theorem,
we now show Theorem \ref{tham}.

\begin{proof}[Proof of Theorem \ref{tham}]
Let all the symbols be as in the present theorem
and, for any $i\in\zz$ and $j\in\nn$, let
$\lz_{i,j}:=\widetilde{A}2^i\|
\mathbf{1}_{B_{i,j}}\|_{X(\cx)}$.
To obtain the convergence of $\sum_{i\in\zz}
\sum_{j\in\nn}\lz_{i,j}a_{i,j}$
in $(\gs)'$, we need to show that,
for any $\varphi\in \gs$,
\begin{equation}\label{t3eq3}
\sum_{i\in\zz}\sum_{j\in\nn}\lz_{i,j}
\lf|\lf\langle a_{i,j},\varphi\r\rangle\r|
\ls \|\varphi\|_{\cg(\bz,\gz)}.
\end{equation}
Indeed, from the estimate that
$\|\varphi\|_{\cg(x,1,\bz,\gz)}
\sim \|\varphi\|_{\cg(\bz,\gz)}$
for any $x\in B(x_0,1)$ and
$\varphi\in\gs$, Proposition \ref{prwb},
and Definition \ref{debb}(ii) with $X(\cx)$
replaced by $WX(\cx)$,
it follows that, for any $\varphi\in\gs$,
\begin{align}\label{t3a}
\sum_{i\in\zz}\sum_{j\in\nn}\lz_{i,j}
\lf|\lf\langle a_{i,j},\varphi\r\rangle\r|
&\ls \inf_{x\in B(x_0,1)}\lf\{
\sum_{i\in\zz}\sum_{j\in\nn}\lz_{i,j}
\lf(a_{i,j}\r)^*(x)\r\}\|\varphi\|_{\cg(\bz,\gz)}\\
&\ls \lf\|\mathbf{1}_{B(x_0,1)}\r\|_{WX(\cx)}
^{-1}\lf\|\sum_{i\in\zz}\sum_{j\in\nn}
\lz_{i,j}\lf(a_{i,j}\r)^*\r\|_{WX(\cx)}
\|\varphi\|_{\cg(\bz,\gz)}.\notag
\end{align}

Then we show that, for any $\lz\in (0,\fz)$,
\begin{equation}\label{t3eq1}
\lz\lf\|\mathbf{1}_{\lf\{x\in\cx:
	\ \sum_{i\in\zz}\sum_{j\in\nn}\lz_{i,j}
(a_{i,j})^*(x)>\lz\r\}}\r\|_{X(\cx)}\ls
\sup_{i\in\zz}\lf\{2^i\lf\|\sum_{j\in\nn}
\mathbf{1}_{B_{i,j}}\r\|_{X(\cx)}\r\}
\end{equation}
with the implicit positive constant
independent of $\lz$,
$\{a_{i,j}\}_{i\in\zz,j\in\nn}$, and
$\{B_{i,j}\}_{i\in\zz,j\in\nn}$.
To this end, we fix $\lz\in (0,\fz)$ and
let $i_0\in\zz$
be such that $2^{i_0}\le\lz<2^{i_0+1}$.
By \eqref{eqsigm} with $p=1$,
and Definition \ref{debb}(ii),
we find that
\begin{align}\label{t3eq2}
&(\sigma_1)^{-2}\lf\|\mathbf{1}_{\{x\in\cx:
\ \sum_{i\in\zz}\sum_{j\in\nn}\lz_{i,j}
(a_{i,j})^*(x)>\lz\}}\r\|_{X(\cx)}\\
&\quad\le\lf\|\mathbf{1}_{\{x\in\cx:
\ \sum^{i_0-1}_{i=-\fz}\sum_{j\in\nn}
\lz_{i,j}(a_{i,j})^*(x)\mathbf{1}_{2A_0B_{i,j}}
(x)>\lz/4\}}\r\|_{X(\cx)}\notag\\
&\quad\quad+\lf\|\mathbf{1}_{\{x\in\cx:
\ \sum^{i_0-1}_{i=-\fz}\sum_{j\in\nn}
\lz_{i,j}(a_{i,j})^*(x)\mathbf{1}
_{(2A_0B_{i,j})^{\complement}}(x)
>\lz/4\}}\r\|_{X(\cx)}\notag\\
&\quad\quad+\lf\|\mathbf{1}_{\{x\in\cx:
\ \sum^\fz_{i=i_0}\sum_{j\in\nn}
\lz_{i,j}(a_{i,j})^*(x)\mathbf{1}_{2A_0B_{i,j}}
(x)>\lz/4\}}\r\|_{X(\cx)}\notag\\
&\quad\quad+\lf\|\mathbf{1}_{\{x\in\cx:
\ \sum^\fz_{i=i_0}\sum_{j\in\nn}
\lz_{i,j}(a_{i,j})^*(x)\mathbf{1}
_{(2A_0B_{i,j})^{\complement}}(x)
>\lz/4\}}\r\|_{X(\cx)}\notag\\
&\quad=:\sum^4_{j=1}\mathrm{I}_j,\notag
\end{align}
where $\sigma_1\in[1,\fz)$ is as in
\eqref{eqsigm} with $p=1$.

We first estimate $\mathrm{I}_1$.
To achieve this,
let $q_0\in (1,1/s_0]\cap(1,q/p_0)$.
Next, we claim that, for any $i\in\zz$
and $j\in\nn$,
\begin{equation}\label{t3eq7}
\lf\|\lf[\lf\|\mathbf{1}_{B_{i,j}}
\r\|_{X(\cx)}\cm\lf(a_{i,j}\r)\r]^{q_0}
\mathbf{1}_{2A_0B_{i,j}}\r\|_{L^{p_0}(\cx)}
\ls\lf[\mu\lf(2A_0B_{i,j}\r)\r]^{1/p_0}.
\end{equation}
Indeed, if $q\in(1,\fz]$, from the H\"{o}lder inequality,
the boundedness of $\cm$ on $L^q(\cx)$, and
Definition \ref{deat}(ii) with $a=a_{i,j}$
and $B=B_{i,j}$,
it follows that, for any $i\in\zz$ and $j\in\nn$,
\begin{align*}
&\lf\|\lf[\lf\|\mathbf{1}_{B_{i,j}}
\r\|_{X(\cx)}\cm\lf(a_{i,j}\r)\r]^{q_0}
\mathbf{1}_{2A_0B_{i,j}}\r\|_{L^{p_0}(\cx)}\\
&\quad\le \lf\|\mathbf{1}_{B_{i,j}}
\r\|_{X(\cx)}^{q_0}\lf\|\cm\lf(a_{i,j}\r)
\r\|^{q_0}_{L^q(\cx)}\lf[\mu\lf(2A_0B_{i,j}\r)
\r]^{1/p_0-q_0/q}\\
&\quad\ls \lf\|\mathbf{1}_{B_{i,j}}
\r\|_{X(\cx)}^{q_0}\lf\|
a_{i,j}\r\|^{q_0}_{L^q(\cx)}
\lf[\mu\lf(2A_0B_{i,j}\r)\r]^{1/p_0-q_0/q}
\ls\lf[\mu\lf(2A_0B_{i,j}\r)\r]^{1/p_0}.
\end{align*}
If $q=1$, by the layer cake representation
(see, for instance, \cite[Proposition 1.1.4]{g14}),
the boundedness of $\cm$ from $L^1(\cx)$
to $WL^1(\cx)$, $q_0p_0\in(0,1)$, and
Definition \ref{deat}(ii) with $a=a_{i,j}$
and $B=B_{i,j}$, we conclude that,
for any $i\in\zz$ and $j\in\nn$,
\begin{align*}
&\lf\|\lf[\lf\|\mathbf{1}_{B_{i,j}}
\r\|_{X(\cx)}\cm\lf(a_{i,j}\r)\r]^{q_0}
\mathbf{1}_{2A_0B_{i,j}}
\r\|_{L^{p_0}(\cx)}^{p_0}\\
&\quad=\int^{\fz}_0q_0p_0
\lz^{q_0p_0-1}\mu\lf(\lf\{
x\in\cx:\ \lf\|\mathbf{1}_{B_{i,j}}
\r\|_{X(\cx)}\cm\lf(a_{i,j}\r)(x)
\mathbf{1}_{2A_0B_{i,j}}(x)>\lz\r\}\r)\,d\lz\\
&\quad\sim\int^{1}_0
\lz^{q_0p_0-1}\mu\lf(\lf\{
x\in 2A_0B_{i,j}:\ \cm\lf(a_{i,j}\r)(x)
>\frac{\lz}{\|\mathbf{1}_{B_{i,j}}\|_{X(\cx)}}
\r\}\r)\,d\lz+\int^\fz_1\cdots\\
&\quad\ls\int^1_0\lz^{q_0p_0-1}
\mu\lf(2A_0B_{i,j}\r)\,d\lz+
\int^{\fz}_1 \lz^{q_0p_0-2}
\lf\|\mathbf{1}_{B_{i,j}}
\r\|_{X(\cx)}\lf\|a_{i,j}\r\|
_{L^1(\cx)}\,d\lz\\
&\quad\ls \mu\lf(2A_0B_{i,j}\r).
\end{align*}
This finishes the proof of the above claim \eqref{t3eq7}.
Observe that
\begin{equation*}
(\lambda/4)^{q_0}\mathbf{1}_{\{x\in\cx:
\ \sum^{i_0-1}_{i=-\fz}\sum_{j\in\nn}
\lz_{i,j}(a_{i,j})^*(x)\mathbf{1}_{2A_0B_{i,j}}
(x)>\lz/4\}}
\le \lf[\sum^{i_0-1}_{i=-\fz}\sum_{j\in\nn}
\lz_{i,j}(a_{i,j})^*\mathbf{1}_{2A_0B_{i,j}}\r]^{q_0}.
\end{equation*}
By this, Definition \ref{debb}(ii),
Lemma \ref{lea} with $a=a_{i,j}$ and
$B=B_{i,j}$, Lemma \ref{lees}(iv) with $s=q_0s_0$,
the assumption that $X^{1/s_0}(\cx)$
is a Banach space, \eqref{t3eq7},
and Proposition \ref{pras} with
$B_j$ and $a_j$ replaced, respectively,
by $2A_0B_{i,j}$ and
a constant multiple of
$[\|\mathbf{1}_{B_{i,j}}\|_{X(\cx)}
\cm(a_{i,j})\mathbf{1}_{2A_0B_{i,j}}]^{q_0}$,
we conclude that
\begin{align*}
\mathrm{I}_1
&\ls \lambda^{-q_0}\lf\|\lf[
\sum^{i_0-1}_{i=-\fz}\sum_{j\in\nn}
\lz_{i,j}\lf(a_{i,j}\r)^*\mathbf{1}_{2A_0B_{i,j}}
\r]^{q_0}\r\|_{X(\cx)}\\
&\ls \lambda^{-q_0}
\lf\|\lf[\sum^{i_0-1}_{i=-\fz}\sum_{j\in\nn}
2^i\lf\|\mathbf{1}_{B_{i,j}}
\r\|_{X(\cx)}\cm\lf(a_{i,j}\r)
\mathbf{1}_{2A_0B_{i,j}}\r]^{q_0s_0}
\r\|_{X^{1/s_0}(\cx)}^{1/s_0}\\
&\ls \lambda^{-q_0}
\lf\|\sum^{i_0-1}_{i=-\fz}\sum_{j\in\nn}
\lf[2^i\lf\|\mathbf{1}_{B_{i,j}}
\r\|_{X(\cx)}\cm\lf(a_{i,j}\r)\r]^{q_0s_0}
\mathbf{1}_{2A_0B_{i,j}}\r\|_{X^{1/s_0}
(\cx)}^{1/s_0}\\
&\ls \lambda^{-q_0}\lf\{\sum^{i_0-1}
_{i=-\fz}2^{iq_0s_0}\lf\|
\sum_{j\in\nn}\lf[\lf\|
\mathbf{1}_{B_{i,j}}\r\|_{X(\cx)}
\cm\lf(a_{i,j}\r)\r]^{q_0s_0}\mathbf{1}_{2A_0B_{i,j}}\r\|
_{X^{1/s_0}(\cx)}\r\}^{1/s_0}\\
&\ls \lambda^{-q_0}\lf[
\sum^{i_0-1}_{i=-\fz}2^{iq_0s_0}\lf\|
\sum_{j\in\nn}\mathbf{1}_{2A_0B_{i,j}}\r\|
_{X^{1/s_0}(\cx)}\r]^{1/s_0}
\end{align*}
with all the implicit positive constants
independent of $\lz$,
$\{a_{i,j}\}_{i\in\zz,j\in\nn}$, and
$\{B_{i,j}\}_{i\in\zz,j\in\nn}$.
Using this, Proposition \ref{prfs}(i) with $t$,
$\tau$, $\lz_j$, and $B_j$ replaced, respectively,
by $s_0$, $2A_0/c$, $1$, and $cB_{i,j}$, the assumption that
$\sum_{j\in\nn}\mathbf{1}_{cB_{i,j}}\le A$
for any $i\in\zz$,
$c\in(0,1]$, $q_0\in(1,\fz)$, and $\lz\sim 2^{i_0}$,
we obtain
\begin{align*}
\mathrm{I}_1
&\ls \lambda^{-q_0}\lf[
\sum^{i_0-1}_{i=-\fz}2^{iq_0s_0}
\lf\|\sum_{j\in\nn}\mathbf{1}_{cB_{i,j}}
\r\|_{X^{1/s_0}(\cx)}\r]^{1/s_0}\\
&\ls \lambda^{-q_0}\lf[
\sum^{i_0-1}_{i=-\fz}2^{iq_0s_0}
\lf\|\sum_{j\in\nn}\mathbf{1}_{cB_{i,j}}\r\|
^{s_0}_{X(\cx)}\r]^{1/s_0}\\
&\ls \lambda^{-q_0}\lf[
\sum^{i_0-1}_{i=-\fz}2^{iq_0s_0}
\lf\|\sum_{j\in\nn}\mathbf{1}_{B_{i,j}}\r\|
^{s_0}_{X(\cx)}\r]^{1/s_0}\\
&\ls \lambda^{-q_0}\lf(
\sum^{i_0-1}_{i=-\fz}2^{iq_0s_0}
2^{-is_0}\r)^{1/s_0}\sup_{i\in\zz}\lf\{2^i
\lf\|\sum_{j\in\nn}\mathbf{1}_{B_{i,j}}
\r\|_{X(\cx)}\r\}\\
&\sim \lz^{-1}\sup_{i\in\zz}\lf\{2^i
\lf\|\sum_{j\in\nn}\mathbf{1}_{B_{i,j}}
\r\|_{X(\cx)}\r\},
\end{align*}
where all the implicit positive constants
are independent of $\lz$,
$\{a_{i,j}\}_{i\in\zz,j\in\nn}$, and
$\{B_{i,j}\}_{i\in\zz,j\in\nn}$.

Then it comes to estimate $\mathrm{I}_2$.
Let $q_1\in (\max\{\frac{\oz}
{(\oz+\bz)p_-},1\},\fz)$.
Observe that
\begin{equation*}
(\lz/4)^{q_1}\mathbf{1}_{\{x\in\cx:
\ \sum^{i_0-1}_{i=-\fz}\sum_{j\in\nn}
\lz_{i,j}(a_{i,j})^*(x)\mathbf{1}_{
(2A_0B_{i,j})^\complement}
(x)>\lz/4\}}
\le\lf[\sum^{i_0-1}_{i=-\fz}\sum_{j\in\nn}
\lz_{i,j}(a_{i,j})^*\mathbf{1}_{
(2A_0B_{i,j})^\complement}\r]^{q_1}.
\end{equation*}
From this, Definition \ref{debb}(ii),
Lemma \ref{lea} with $a=a_{i,j}$ and $B=B_{i,j}$,
Proposition \ref{prtz} with $p=q_1$, and
Assumption \ref{asfs} with
$t=\frac{\oz}{(\oz+\bz)q_1}\in(0,p_-)$
and $s=\frac{\oz+\bz}{\oz}\in(1,\fz)$,
we deduce that
\begin{align*}
\mathrm{I}_2
&\ls \lz^{-q_1}\lf\|\lf[\sum^{i_0-1}_{i=-\fz}
\sum_{j\in\nn}\lz_{i,j}\lf(a_{i,j}\r)^*
\mathbf{1}_{(2A_0B_{i,j})^{\complement}}\r]^{q_1}\r\|_{X(\cx)}\\
&\ls \lz^{-q_1}\lf\|\sum^{i_0-1}
_{i=-\fz}\sum_{j\in\nn}2^i
\lf[\cm\lf(\mathbf{1}_{B_{i,j}}\r)
\r]^{\frac{\oz+\bz}{\oz}}\r\|^{q_1}_{X^{q_1}(\cx)}\\
&\ls \lz^{-q_1}\lf\{\sum^{i_0-1}_{i=-\fz}2^{i\tz_{q_1}}
\lf\|\sum_{j\in\nn}\lf[\cm\lf(\mathbf{1}_{
B_{i,j}}\r)\r]^{\frac{\oz+\bz}{\oz}}
\r\|^{\tz_{q_1}}_{X^{q_1}(\cx)}\r\}^{q_1/\tz_{q_1}}\\
&\ls \lz^{-q_1}\lf[\sum^{i_0-1}_{i=-\fz}2^{i\tz_{q_1}}
\lf\|\sum_{j\in\nn}\mathbf{1}_{B_{i,j}}
\r\|^{\tz_{q_1}}_{X^{q_1}(\cx)}\r]^{q_1/\tz_{q_1}},
\end{align*}
where $\tz_{q_1}\in(0,1]$ is as in Proposition
\ref{prtz} with $p=q_1$, and
all the implicit positive constants are
independent of $\lz$,
$\{a_{i,j}\}_{i\in\zz,j\in\nn}$, and
$\{B_{i,j}\}_{i\in\zz,j\in\nn}$.
By this, Proposition \ref{prfs}(i) with $t$, $\tau$,
$\lz_j$, and $B_j$ replaced, respectively,
by $1/q_1$, $1/c$, $1$, and $cB_{i,j}$, the assumption
that $\sum_{j\in\nn}\mathbf{1}_{cB_{i,j}}\le A$
for any $i\in\zz$,
$c\in(0,1]$, $q_1\in(1,\fz)$, and $\lambda\sim 2^{i_0}$,
we conclude that
\begin{align*}
\mathrm{I}_2
&\ls \lz^{-q_1}\lf[\sum^{i_0-1}_{i=-\fz}2^{i\tz_{q_1}}
\lf\|\sum_{j\in\nn}\mathbf{1}_{cB_{i,j}}
\r\|^{\tz_{q_1}}_{X^{q_1}(\cx)}\r]^{q_1/\tz_{q_1}}\\
&\ls \lz^{-q_1}\lf[\sum^{i_0-1}_{i=-\fz}2^{i\tz_{q_1}}
\lf\|\sum_{j\in\nn}\mathbf{1}_{cB_{i,j}}
\r\|^{\tz_{q_1}/q_1}_{X(\cx)}\r]^{q_1/\tz_{q_1}}\\
&\ls \lz^{-q_1}\lf[\sum^{i_0-1}_{i=-\fz}
2^{i\tz_{q_1}(1-1/q_1)}\r]^{q_1/\tz_{q_1}}
\sup_{i\in\zz}\lf\{2^i\lf\|\sum_{j\in\nn}
\mathbf{1}_{B_{i,j}}\r\|_{X(\cx)}\r\}\\
&\sim \lambda^{-1}\sup_{i\in\zz}\lf\{2^{i}\lf\|
\sum_{j\in\nn}\mathbf{1}_{B_{i,j}}\r\|_{X(\cx)}\r\}
\end{align*}
with all the implicit positive constants
independent of $\lz$,
$\{a_{i,j}\}_{i\in\zz,j\in\nn}$, and
$\{B_{i,j}\}_{i\in\zz,j\in\nn}$.

Next, we estimate $\mathrm{I_3}$. Observe that
\begin{align*}
\mathbf{1}_{\{x\in\cx:
\ \sum^\fz_{i=i_0}\sum_{j\in\nn}
\lz_{i,j}(a_{i,j})^*(x)\mathbf{1}_{2A_0B_{i,j}}(x)>\lz/4\}}
\le \mathbf{1}_{\bigcup^\fz_{i=i_0}\bigcup_{j\in\nn}2A_0B_{i,j}}
\le \sum^\fz_{i=i_0}\sum_{j\in\nn}
\mathbf{1}_{2A_0B_{i,j}}.
\end{align*}
From this, Definition \ref{debb}(ii),
Proposition \ref{prtz} with $p=1$,
Proposition \ref{prfs}(i) with $t$,
$\tau$, $\lz_j$, and $B_j$ replaced, respectively,
by $1$, $2A_0$, $1$, and $B_{i,j}$,
and $\lz\sim 2^{i_0}$, we deduce that
\begin{align*}
\mathrm{I}_3
&\le \lf\|\sum^{\fz}_{i=i_0}\sum_{j\in\nn}
\mathbf{1}_{2A_0B_{i,j}}\r\|_{X(\cx)}
\ls \lf[\sum^{\fz}_{i=i_0}\lf\|\sum_{j\in\nn}
\mathbf{1}_{2A_0B_{i,j}}\r\|^{\tz_1}_{X(\cx)}
\r]^{1/\tz_1}\\
&\ls \lf[\sum^{\fz}_{i=i_0}\lf\|\sum_{j\in\nn}
\mathbf{1}_{B_{i,j}}\r\|^{\tz_1}_{X(\cx)}
\r]^{1/\tz_1}
\ls \lf(\sum^{\fz}_{i=i_0}2^{-i\tz_1}\r)
^{1/\tz_1}\sup_{i\in\zz}\lf\{2^i\lf\|
\sum_{j\in\nn}\mathbf{1}_{B_{i,j}}\r\|_{X(\cx)}\r\}\\
&\sim \lz^{-1} \sup_{i\in\zz}\lf\{2^i\lf\|
\sum_{j\in\nn}\mathbf{1}_{B_{i,j}}\r\|_{X(\cx)}\r\},
\end{align*}
where $\tz_1\in(0,1]$ is as in Proposition \ref{prtz}
with $p=1$,
and all the implicit positive constants
are independent of $\lz$,
$\{a_{i,j}\}_{i\in\zz,j\in\nn}$, and
$\{B_{i,j}\}_{i\in\zz,j\in\nn}$.

Finally, we estimate $\mathrm{I}_4$.
Let $q_2\in (\frac{\oz}{(\oz+\bz)\min\{1,p_-\}},1)$.
Observe that
\begin{equation*}
(\lz/4)^{q_2}\mathbf{1}_{\{x\in\cx:
\ \sum^\fz_{i=i_0}\sum_{j\in\nn}
\lz_{i,j}(a_{i,j})^*(x)\mathbf{1}
_{(2A_0B_{i,j})^{\complement}}(x)
>\lz/4\}}
\le \lf[\sum^{\fz}_{i=i_0}\sum_{j\in\nn}
\lz_{i,j}\lf(a_{i,j}\r)^*\mathbf{1}_{(2A_0
B_{i,j})^{\complement}}\r]^{q_2}.
\end{equation*}
By this, Definition \ref{debb}(ii),
Lemma \ref{lea} with $a=a_{i,j}$ and
$B=B_{i,j}$, Lemma \ref{lees}(iv) with $s=q_2$,
Proposition \ref{prtz} with $p=1$,
Assumption \ref{asfs} with
$t=\frac{\oz}{(\oz+\bz)q_2}\in(0,p_-)$
and $s=\frac{(\oz+\bz)q_2}{\oz}\in(1,\fz)$,
$q_2\in(0,1)$, and $\lz\sim 2^{i_0}$,
we have
\begin{align*}
\mathrm{I}_4
&\ls \lz^{-q_2}\lf\|\lf[\sum^{\fz}_{i=i_0}\sum_{j\in\nn}
\lz_{i,j}\lf(a_{i,j}\r)^*\mathbf{1}_{(2A_0
B_{i,j})^{\complement}}\r]^{q_2}\r\|_{X(\cx)}\\
&\ls \lz^{-q_2}\lf\|\lf\{
\sum^{\fz}_{i=i_0}\sum_{j\in\nn}2^i
\lf[\cm\lf(\mathbf{1}_{B_{i,j}}\r)
\r]^{\frac{(\oz+\bz)}{\oz}}\r\}^{q_2}
\r\|_{X(\cx)}\\
&\ls \lz^{-q_2}\lf\|\sum^{\fz}_{i=i_0}
2^{iq_2}\sum_{j\in\nn}
\lf[\cm\lf(\mathbf{1}_{B_{i,j}}\r)
\r]^{\frac{(\oz+\bz)q_2}{\oz}}
\r\|_{X(\cx)}\\
&\ls \lz^{-q_2}\lf\{\sum^{\fz}_{i=i_0}
2^{iq_2\tz_1}\lf\|\sum_{j\in\nn}
\lf[\cm\lf(\mathbf{1}_{B_{i,j}}\r)
\r]^{\frac{(\oz+\bz)q_2}{\oz}}
\r\|^{\tz_1}
_{X(\cx)}\r\}^{1/\tz_1}\\
&\ls \lz^{-q_2}\lf[\sum^{\fz}_{i=i_0}
2^{iq_2\tz_1-i\tz_1}2^{i\tz_1}\lf\|\sum_{j\in\nn}
\mathbf{1}_{B_{i,j}}
\r\|^{\tz_1}_{X(\cx)}\r]^{1/\tz_1}
\ls \lambda^{-1}\sup_{i\in\zz}\lf\{2^i\lf\|\sum_{j\in\nn}
\mathbf{1}_{B_{i,j}}\r\|_{X(\cx)}\r\},
\end{align*}
where $\tz_1\in(0,1]$ is as in Proposition \ref{prtz}
with $p=1$, and all the implicit positive constants
are independent of $\lz$,
$\{a_{i,j}\}_{i\in\zz,j\in\nn}$, and
$\{B_{i,j}\}_{i\in\zz,j\in\nn}$.

Using \eqref{t3eq2} and the estimates of
$\{\mathrm{I}_j\}^4_{j=1}$,
we obtain \eqref{t3eq1}, which, combined with
\eqref{t3a} and the definition of $WX(\cx)$,
implies \eqref{t3eq3}.
Thus, for any $\varphi\in\gs$,
$$
\langle f,\varphi\rangle:=
\sum_{i\in\zz}\sum_{j\in\nn}
\lf\langle\lz_{i,j}a_{i,j},\varphi\r\rangle
$$
is well defined and $f=\sum_{i\in\zz}\sum_{j\in\nn}
\lz_{i,j}a_{i,j}$ in $(\gs)'$. By the Fatou lemma,
it is easy to see that
$f^*\le \sum_{i\in\zz}\sum_{j\in\nn}
\lz_{i,j}(a_{i,j})^*$ which, together with
\eqref{t3eq1} and the definition of $WX(\cx)$,
further implies that
$f\in WH_X(\cx)$ and
\begin{equation*}
\|f\|_{WH_X(\cx)}
\le \lf\|\sum_{i\in\zz}\sum_{j\in\nn}
\lz_{i,j}\lf(a_{i,j}\r)^*\r\|_{WX(\cx)}
\ls \sup_{i\in\zz}\lf\{2^i\lf\|\sum_{j\in\nn}
\mathbf{1}_{B_{i,j}}\r\|_{X(\cx)}\r\},
\end{equation*}
where the implicit positive constant is
independent of $\{a_{i,j}\}_{i\in\zz,j\in\nn}$
and $\{B_{i,j}\}_{i\in\zz,j\in\nn}$.
This finishes the proof of Theorem \ref{tham}.
\end{proof}
\begin{remark}\label{ream}
In Theorem \ref{tham}, if $q=\fz$, then,
to obtain the conclusion of Theorem \ref{tham},
we only need to assume that $X(\cx)$ satisfies
Assumption \ref{asfs} with $p_-\in(\oz/(\oz+\eta),\fz)$.
To show this,
we use the same symbols as in the proof of
Theorem \ref{tham}. Indeed, this proof
is similar to that of Theorem \ref{tham} and
the only difference is the estimation of $\mathrm{I}_1$,
which we present now.
Let $q_3\in (1,\fz)$.
By Definition \ref{debb}(ii),
Lemma \ref{lea} with $a=a_{i,j}$ and
$B=B_{i,j}$, the estimate that
$\|\mathbf{1}_{B_{i,j}}\|_{X(\cx)}
\|\cm(a_{i,j})\|_{L^\fz(\cx)}\le
\|\mathbf{1}_{B_{i,j}}\|_{X(\cx)}
\|a_{i,j}\|_{L^\fz(\cx)}\le 1$
for any $i\in\zz$ and $j\in\nn$,
and Proposition \ref{prtz} with
$p=q_3$,
we conclude that
\begin{align*}
\mathrm{I}_1
&\ls \lambda^{-q_3}\lf\|\lf[
\sum^{i_0-1}_{i=-\fz}\sum_{j\in\nn}
\lz_{i,j}\lf(a_{i,j}\r)^*\mathbf{1}_{2A_0B_{i,j}}
\r]^{q_3}\r\|_{X(\cx)}\\
&\ls \lambda^{-q_3}
\lf\|\sum^{i_0-1}_{i=-\fz}\sum_{j\in\nn}
2^i\lf\|\mathbf{1}_{B_{i,j}}
\r\|_{X(\cx)}\cm\lf(a_{i,j}\r)
\mathbf{1}_{2A_0B_{i,j}}
\r\|^{q_3}_{X^{q_3}(\cx)}\\
&\ls \lambda^{-q_3}
\lf\|\sum^{i_0-1}_{i=-\fz}\sum_{j\in\nn}
2^i\mathbf{1}_{2A_0B_{i,j}}
\r\|^{q_3}_{X^{q_3}(\cx)}\\
&\ls \lambda^{-q_3}\lf[\sum^{i_0-1}_{i=-\fz}
2^{i\tz_{q_3}}\lf\|\sum_{j\in\nn}
\mathbf{1}_{2A_0B_{i,j}}
\r\|^{\tz_{q_3}}_{X^{q_3}(\cx)}
\r]^{q_3/\tz_{q_3}},
\end{align*}
where $\tz_{q_3}\in(0,1]$ is as in Proposition \ref{prtz}
with $p=q_3$, and all the implicit positive
constants are independent of $\lz$,
$\{a_{i,j}\}_{i\in\zz,j\in\nn}$, and
$\{B_{i,j}\}_{i\in\zz,j\in\nn}$.
Using this, Proposition \ref{prfs}(i) with $t$,
$\tau$, $\lz_j$, and $B_j$ replaced, respectively,
by $1/q_3$, $2A_0/c$, $1$, and $cB_{i,j}$, the assumption that
$\sum_{j\in\nn}\mathbf{1}_{cB_{i,j}}\le A$
for any $i\in\zz$,
$c\in(0,1]$, $q_3\in(1,\fz)$, and $\lz\sim 2^{i_0}$,
we obtain
\begin{align*}
\mathrm{I}_1
&\ls \lambda^{-q_3}\lf[\sum^{i_0-1}_{i=-\fz}
2^{i\tz_{q_3}}\lf\|\sum_{j\in\nn}
\mathbf{1}_{cB_{i,j}}
\r\|^{\tz_{q_3}}_{X^{q_3}(\cx)}\r]^{q_3/\tz_{q_3}}\\
&\ls \lambda^{-q_3}\lf[\sum^{i_0-1}_{i=-\fz}
2^{i\tz_{q_3}}\lf\|\sum_{j\in\nn}
\mathbf{1}_{cB_{i,j}}
\r\|^{\tz_{q_3}/q_3}_{X(\cx)}\r]^{q_3/\tz_{q_3}}\\
&\ls \lambda^{-q_3}\lf[\sum^{i_0-1}_{i=-\fz}
2^{i\tz_{q_3}(1-1/q_3)}\r]^{q_3/\tz_{q_3}}
\sup_{i\in\zz}\lf\{2^i
\lf\|\sum_{j\in\nn}\mathbf{1}_{B_{i,j}}
\r\|_{X(\cx)}\r\}\\
&\sim \lz^{-1}\sup_{i\in\zz}\lf\{2^i
\lf\|\sum_{j\in\nn}\mathbf{1}_{B_{i,j}}
\r\|_{X(\cx)}\r\},
\end{align*}
where all the implicit positive constants
are independent of $\lz$,
$\{a_{i,j}\}_{i\in\zz,j\in\nn}$, and
$\{B_{i,j}\}_{i\in\zz,j\in\nn}$.
This finishes the estimation of $\mathrm{I}_1$,
and hence the proof of the present remark.
\end{remark}

\subsection{Decomposition\label{ssd}}

The decomposition theorem is as follows.

\begin{theorem}\label{thma}
Let $\oz$ be as in \eqref{eqoz},
$\eta$ as in \eqref{eqwa},
and $X(\cx)$ a \emph{BQBF} space satisfying
Assumption \ref{asfs} with
$p_-\in(\oz/(\oz+\eta),\fz)$.
If $\bz$, $\gz\in (\oz[1/\min\{1,p_-\}-1],\eta)$,
then there exist positive constants
$A$, $\widetilde{A}$, and $C$ such that,
for any $f\in WH_X(\cx)\cap (\gs)'$, there exists
a sequence $\{a_{i,j}\}_{i\in\zz,j\in\nn}$
of $(X(\cx),\fz)$-atoms supported, respectively,
in balls $\{B_{i,j}\}_{i\in\zz,j\in\nn}$
such that
$\sum_{j\in\nn}\mathbf{1}_{B_{i,j}}\le A$
for any $i\in\zz$,
$f=\sum_{i\in\zz}\sum_{j\in\nn}
\widetilde{A}2^i
\|\mathbf{1}_{B_{i,j}}\|_{X(\cx)}
a_{i,j}$ in $(\gs)'$, and
$$
\sup_{i\in\zz}\lf\{2^i\lf\|\sum_{j\in\nn}
\mathbf{1}_{B_{i,j}}\r\|_{X(\cx)}\r\}
\le C\lf\|f\r\|_{WH_X(\cx)}.
$$
\end{theorem}
To prove this theorem, we need to make some preparation.
We first establish the following technical lemma which is a
slight different variant of \cite[Proposition 4.4]{hhllyy}.
\begin{lemma}\label{lewh1}
Let $A_0$ be as in \eqref{eqA0}, and $\Oz\subset\cx$ a proper
open subset. Then there exists a sequence
$\{x_j\}_{j\in\nn}\subset \Oz$ such that
\begin{enumerate}
\item[\textup{(i)}]
$\{B(x_j,(5A_0^3)^{-1}r_j)\}_{j\in\nn}$
is pairwise disjoint, where, for any $j\in\nn$,
$r_j:=(2A_0)^{-5}d(x_j,\Oz^{\complement})$;
\item[\textup{(ii)}]
$\bigcup_{j\in\nn}B(x_j,r_j)=\Oz$;
\item[\textup{(iii)}]
for any $j\in\nn$ and $x\in B(x_j,r_j)$,
$(2A_0)^4r_j\le d(x,\Oz^{\complement})\le 48A_0^6r_j$;
\item[\textup{(iv)}]
for any $j\in\nn$, there exists a
$y_j\in \Oz^\complement$ such that
$d(x_j,y_j)<48A^5_0r_j$;
\item[\textup{(v)}]
for any $j\in\nn$,
$\# \{l\in \nn:\ (2A_0)^4B(x_l.r_l)
\cap (2A_0)^4B(x_j,r_j)\neq\emptyset\}\le L_0$,
where $L_0$ is a positive integer
independent of $j$ and $\Oz$.
\end{enumerate}
\end{lemma}

Lemma \ref{lewh1} when $\mu(\Oz)<\fz$ is just \cite[Proposition 4.4]{hhllyy}.
Observe that the condition that $\mu(\Oz)<\fz$
in the proof of \cite[Proposition 4.4]{hhllyy}
is only used to show
that the set of the chosen points,
which was denoted by $\{x_j\}_{j\in I}$
with $I$ being a set of indices therein,
is countable, but this
can be deduced directly from both the pairwise disjointness
of balls, $\{B(x_j,(5A_0^3)^{-1}r_j)\}_{j\in I}$, and the doubling
condition of $\mu$. This implies that Lemma \ref{lewh1}
still holds true even when $\Omega$
is just a proper subset of $\cx$.
We omit the details.

Using Lemma \ref{lewh1},
we can establish the following lemma which is a slight
different variant of \cite[Proposition 4.5]{hhllyy};
we also omit the details.
\begin{lemma}\label{lewh2}
Let $\Oz$, $\{x_j\}_{j\in\nn}$, and $\{r_j\}_{j\in\nn}$
be as in Lemma \ref{lewh1}. For any $j\in\nn$, let
$$B_j:=(2A_0)^4B(x_j,r_j).$$
Then there exists a sequence $\{\phi_j\}_{j\in\nn}$
of non-negative functions such that
\begin{enumerate}
\item[\textup{(i)}]
for any $j\in\nn$, $0\le \phi_j\le 1$ and
$\supp(\phi_j)\subset (2A_0)^{-3}B_j$;
\item[\textup{(ii)}]
$\sum_{j\in\nn}\phi_j=\mathbf{1}_{\Oz}$;
\item[\textup{(iii)}]
for any $j\in\nn$ and $x\in(2A_0)^{-4}B_j$,
$\phi_j(x)\ge L_0^{-1}$,
where $L_0$ is as in Lemma \ref{lewh1};
\item[\textup{(iv)}]
for any $j\in\nn$,
$\|\phi_j\|_{\cg(x_j,r_j,\eta,\eta)}\le CV_{r_j}(x_j)$,
where $C$ is a positive constant independent of
both $j$ and $\Oz$.
\end{enumerate}
\end{lemma}

Observe that, for any $f\in WH_X(\cx)$
and $\lz\in (0,\fz)$, the level set
$\left\{x\in\cx:\ f^*(x)>\lz\right\}$
is not necessarily open.
To overcome this difficulty,
we borrow some ideas from \cite[Theorem 2]{ms79a}
and \cite[Definition 2.5]{gly08}.
Indeed, from the proof of \cite[Theorem 2]{ms79a},
it follows that
there exist a metric $\rho$ and constants
$\ez\in(0,1]$ and $C\in(0,\fz)$ such that,
for any $x$, $y\in\cx$,
\begin{equation}\label{eqdr}
C^{-1}[d(x,y)]^{\varepsilon}
\le \rho(x,y)
\le C [d(x,y)]^{\varepsilon}.
\end{equation}
For any $x\in\cx$ and
$r\in (0,\fz)$, we define the \emph{$\rho$-ball}
$B_\rho(x,r):=\lf\{y\in\cx:\ \rho(x,y)<r\right\}$.
Then $(\cx,\rho,\mu)$ is a doubling metric measure space and,
for any $x,\ y\in\cx$ and $r\in (0,\fz)$, we have
\begin{equation}\label{eqbdr}
\mu(B(y,r+d(x,y)))
\sim\mu\lf(B_\rho\lf(y,[r+d(x,y)]^{\varepsilon}\r)\r)
\sim \mu\lf(B_\rho\lf(y,r^{\varepsilon}+\rho(x,y)\r)\r),
\end{equation}
where the positive equivalence constants are
independent of $x$, $y$, and $r$.
We now recall the following concept of test functions
in terms of the metric $\rho$;
see, for instance, \cite[Definition 2.5]{gly08}.
\begin{definition}
Let $x\in\cx$, $r\in(0,\fz)$, $\bz\in(0,1]$,
and $\gz\in (0,\fz)$.
The \emph{space $\cg_\rho(x,r,\bz,\gz)$ of test functions}
is defined to be the set of all the functions $f$ satisfying
that there exists a positive constant $C$ such that
\begin{enumerate}
\item[(i)] (the \emph{size condition})
for any $y\in\cx$,
\begin{equation*}
|f(y)|\le C\frac{1}{\mu(B_\rho(y,r+\rho(x,y)))}
\lf[\frac{r}{r+\rho(x,y)}\r]^{\gz};
\end{equation*}
\item[(ii)] (the \emph{regularity condition})
for any $y$, $y'\in\cx$
satisfying $\rho(y,y')\le [r+\rho(x,y)]/2$,
\begin{equation*}
\lf|f(y)-f(y')\r|\le C\lf[\frac{\rho(y,y')}
{r+\rho(x,y)}\r]^{\bz}
\frac{1}{\mu(B_\rho(y,r+\rho(x,y)))}
\lf[\frac{r}{r+\rho(x,y)}\r]^{\gz}.
\end{equation*}
\end{enumerate}
Moreover, for any $f\in \cg_\rho(x,r,\bz,\gz)$, let
\begin{equation*}
\|f\|_{\cg_\rho(x,r,\bz,\gz)}:=
\inf \{C\in (0,\fz):\ C\
\mathrm{satisfies}\ \hbox{(i) and (ii)}\}.
\end{equation*}
\end{definition}

By \eqref{eqdr}, \eqref{eqbdr}, and Lemma \ref{lees}(i),
we conclude that
\begin{equation}\label{eqgr}
\cg(x,r,\bz,\gz)
=\cg_\rho(x,r^\varepsilon,
\bz/\varepsilon,\gz/\varepsilon)
\end{equation}
with equivalent quasi-norms, where the positive
equivalence constants are
independent of both $x$ and $r$.
Recall that, for any $f\in(\gs)'$,
the \emph{modified grand maximal function $\fs$}
of $f$ is defined by setting, for any $x\in\cx$,
\begin{equation}\label{eqfs}
\fs(x):=\sup\lf\{|\langle f,\phi\rangle|:\
\phi\in \gs\ \mathrm{with}\
\|\phi\|_{\cg_\rho(x,r^\varepsilon,
\bz/\varepsilon,\gz/\varepsilon)}\le 1\
\mathrm{for\ some}\ r\in (0,\fz)\r\}.
\end{equation}
By \eqref{eqgr} and an argument
similar to that used in the proof of
\cite[Remark 2.9(iii)]{gly08},
 we can obtain the following lemma;
we omit the details here.

\begin{lemma}\label{lefs}
Let $\eta$ be as in \eqref{eqwa} and
$\bz$, $\gz\in(0,\eta)$.
Then, for any $x\in\cx$ and $f\in (\gs)'$,
$f^*(x)\sim \fs(x)$ with the positive
equivalence constants
independent of both $f$ and $x$.
Moreover, for any $f\in (\gs)'$
and $\lz\in(0,\fz)$, the level set
$\{x\in\cx:\ \fs(x)>\lz\}$
is open under the topology induced by $d$.
\end{lemma}

Using the above preparation, the
Calder\'{o}n--Zygmund decomposition, and
the relation between $WX(\cx)$ and the convexification
of $X(\cx)$, we can now show Theorem \ref{thma}.

\begin{proof}[Proof of Theorem \ref{thma}]
Let $X(\cx)$, $p_-$, $\bz$, and $\gz$
be as in the present theorem,
$\{P_m\}_{m\in\zz}$ a 1-exp-ATI as
in Definition \ref{dee1}, and
$f\in WH_X(\cx)\cap(\gs)'$.
By \cite[Lemma 3.2(iii)]{hhllyy},
we find that, for any $\varphi\in\gs$,
\begin{equation}\label{t2e3}
\lim_{m\to\fz}\langle P_m(f),\varphi\rangle
=\langle f,\varphi\rangle.
\end{equation}
Next, we divide the remainder of
the present proof into three steps.

Step 1) In this step, we show that, for any $m\in\nn$,
there exists a sequence
$\{h^{(m)}_{i,j}\}_{i\in\zz,j\in\nn}
\subset L^\fz(\cx)$ supported, respectively, in balls
$\{B_{i,j}\}_{i\in\zz,j\in\nn}$ such that, for any
$\psi\in \cg_{\mathrm{b}}(\eta,\eta)$,
\begin{equation*}
\lf\langle P_m(f),\psi\r\rangle
=\sum_{i\in\zz}\sum_{j\in\nn}
\lf\langle h^{(m)}_{i,j},\psi\r\rangle.
\end{equation*}
Moreover, for any $i\in\zz$ and $j\in\nn$,
$\int_{\cx}h^{(m)}_{i,j}(z)\,d\mu(z)=0$ and
$\|h^{(m)}_{i,j}\|_{L^{\fz}(\cx)}\le \widetilde{A}2^i$
with $\widetilde{A}$ being a positive constant
independent of $f$, $m$, $i$, and $j$.
Furthermore, for any $i\in\zz$,
$\sum_{j\in\nn}\mathbf{1}_{B_{i,j}}\le L_0$
with $L_0$ as in Lemma \ref{lewh1}.

By Lemma \ref{lefs}, we find that,
for any $i\in\zz$, the set
$$\Omega_i:=\lf\{x\in\cx: \fs(x)>2^i\r\}$$
is open, where $\fs$ is as in \eqref{eqfs}.
From the definition of $WX(\cx)$,
$\fs\sim f^*$, Proposition \ref{prwb},
Definition \ref{debb}(ii) with $X(\cx)$
replaced by $WX(\cx)$,
and $f^*\in WX(\cx)$,
it follows that, for any $i\in\zz$,
\begin{equation*}
\lf\|\mathbf{1}_{\Oz_i}\r\|_{X(\cx)}
\le 2^{-i}\lf\|\fs\r\|_{WX(\cx)}
\sim 2^{-i}\lf\|f^*\r\|_{WX(\cx)}
<\fz,
\end{equation*}
which, together with Proposition \ref{prfs}(ii), further
implies that $\Oz_i$ is a proper subset of $\cx$.
Using Lemmas \ref{lewh1} and \ref{lewh2} with
$\Oz$ replaced by $\Oz_i$,  we find that,
for any $i\in\zz$, there
exist a sequence of points,
$\{x_{i,j}\}_{j\in\nn}\subset\cx$, and
a sequence of non-negative functions,
$\{\phi_{i,j}\}_{j\in\nn}$, satisfying
all the conclusions of these lemmas.
For any $i\in\zz$ and $j\in\nn$, let
$$
r_{i,j}:=\lf(2A_0\r)^{-5}
d\lf(x_{i,j},\Omega_i^\complement\r)
\quad \mathrm{and}\quad
B_{i,j}:=(2A_0)^4B\lf(x_{i,j},r_{i,j}\r),
$$
where $A_0$ is as in \eqref{eqA0}.
Observe that, for any $i\in\zz$,
\begin{equation}\label{eqbl0}
\sum_{j\in\nn}\mathbf{1}_{B_{i,j}}\le L_0\mathbf{1}_{\Oz_i}
\end{equation}
with $L_0$ as in Lemma \ref{lewh1}.

For any $m$, $j$, $l\in\nn$, $i\in\zz$,
and $x\in\cx$, let
$$W^{(m)}_{i,j}:=\lf\|\phi_{i,j}\r\|^{-1}_{L^1(\cx)}
\int_{\cx}P_m(f)(z)\phi_{i,j}(z)\,d\mu(z),$$
\begin{equation*}
b^{(m)}_{i,j}(x):=\lf[P_m(f)(x)-W^{(m)}_{i,j}\r]
\phi_{i,j}(x),
\end{equation*}
and
\begin{equation*}
L^{(m)}_{i,j,l}:=\lf\|\phi_{i+1,l}\r\|_{L^1(\cx)}^{-1}
\int_{\cx}b^{(m)}_{i+1,l}(z)\phi_{i,j}(z)\,d\mu(z).
\end{equation*}

By some arguments similar to those used in the estimations
of \cite[(4.25) and (4.30)]{zhy20}, we find that,
for any $m$, $j$, $l\in\nn$ and $i\in\zz$,
\begin{equation}\label{t2eq3}
\lf|W^{(m)}_{i,j}\r|+\lf|L^{(m)}_{i,j,l}\r|
\ls 2^i
\end{equation}
with the implicit positive constant independent
of $f$, $m$, $i$, $j$, and $l$.
Applying an argument similar to that used
in the estimation of
\cite[(4.21)]{zhy20}, we conclude that,
for any $m$, $j\in\nn$ and $i\in\zz$,
\begin{align}\label{t2e7}
\lf[b^{(m)}_{i,j}\r]^*\ls
f^*\mathbf{1}_{B_{i,j}}+2^i
\frac{V_{r_{i,j}}(x_{i,j})}
{V_{r_{i,j}}(x_{i,j})+V(x_{i,j},x)}
\lf[\frac{r_{i,j}}{r_{i,j}+d(x_{i,j},x)}\r]^{\bz}
\mathbf{1}_{(B_{i,j})^{\complement}}
\end{align}
with the implicit positive constant independent
of $f$, $m$, $i$, and $j$.
From Lemma \ref{lewh2}(i) with $\Oz$ replaced by $\Oz_i$,
it follows that, for any $m$, $j\in\nn$ and $i\in\zz$,
\begin{equation}\label{eqsupp}
\supp(b^{(m)}_{i,j})
\subset \supp(\phi_{i,j})
\subset (2A_0)^{-3}B_{i,j}.
\end{equation}

By \eqref{eqbl0}, we find that,
for any $i\in\zz$ and $x\in\cx$,
\begin{equation*}
\#\mathcal{L}_i(x):
=\#\lf\{l\in\nn:\
x\in (2A_0)^{-3}B_{i+1,l}\r\}
\le L_0.
\end{equation*}
Using \eqref{eqsupp},
we conclude that, for any $m\in\nn$,
$i\in\zz$, $x\in\cx$,
and $l\in \nn\setminus \mathcal{L}_i(x)$,
\begin{equation*}
b^{(m)}_{i+1,l}(x)=\phi_{i+1,l}(x)=0.
\end{equation*}
Thus, for any $m$, $j\in\nn$,
$i\in\zz$, and $x\in\cx$,
\begin{align}\label{al2}
h^{(m)}_{i,j}(x):
&=b^{(m)}_{i,j}(x)-\sum_{l\in\nn}
\lf[b^{(m)}_{i+1,l}(x)\phi_{i,j}(x)-
L^{(m)}_{i,j,l}\phi_{i+1,l}(x)\r]\\
&=b^{(m)}_{i,j}(x)-\sum_{l\in\mathcal{L}_i(x)}
\lf[b^{(m)}_{i+1,l}(x)\phi_{i,j}(x)-
L^{(m)}_{i,j,l}\phi_{i+1,l}(x)\r]\notag
\end{align}
is well defined. From this, the definitions
of both $b^{(m)}_{i,j}$ and $b^{(m)}_{i+1,l}$,
$\sum_{l\in\nn}\phi_{i+1,l}=\mathbf{1}_{\Oz_{i+1}}$
for any $i\in\zz$,
$|P_m(f)|\ls f^*\sim\fs$ for any $m\in\nn$, and \eqref{t2eq3},
we find that,
for any $m$, $j\in\nn$, $i\in\zz$,
and $x\in\cx$,
\begin{align}\label{al1}
\lf|h^{(m)}_{i,j}(x)\r|
&=\lf|\lf[P_m(f)(x)-W^{(m)}_{i,j}\r]\phi_{i,j}(x)
-\sum_{l\in\mathcal{L}_{i}(x)}\lf[P_m(f)(x)
-W^{(m)}_{i+1,l}\r]\r.\\
&\quad\lf.\times\phi_{i+1,l}(x)\phi_{i,j}(x)
+\sum_{l\in\mathcal{L}_i(x)}L^{(m)}_{i,j,l}
\phi_{i+1,l}(x)\r|\notag\\
&=\lf|P_m(f)(x)\r|\mathbf{1}_{(\Oz_{i+1})^{\complement}}(x)
\phi_{i,j}(x)+\lf|W^{(m)}_{i,j}\r|\phi_{i,j}(x)\notag\\
&\quad+\sum_{l\in\mathcal{L}_{i}(x)}
\lf|W^{(m)}_{i+1,l}\r|\phi_{i+1,l}(x)\phi_{i,j}(x)+
\sum_{l\in\mathcal{L}_{i}(x)}
\lf|L^{(m)}_{i,j,l}\r|\phi_{i+1,l}(x)\notag\\
&\le \widetilde{A}2^{i},\notag
\end{align}
where $\widetilde{A}$ is a positive constant
independent of $f$, $m$, $i$, and $j$.
Thus, for any $m$, $j\in\nn$ and $i\in\zz$,
\begin{equation}\label{eqlfz}
\lf\|h^{(m)}_{i,j}\r\|_{L^{\fz}(\cx)}\le
\widetilde{A}2^i.
\end{equation}
Then we show that,
for any $m$, $j\in\nn$ and $i\in\zz$,
\begin{equation}\label{t2eq2}
\supp\lf(h^{(m)}_{i,j}\r)\subset B_{i,j}
\end{equation}
\ and
\begin{equation}\label{t2e6}
\int_{\cx}h^{(m)}_{i,j}(z)\,d\mu(z)=0.
\end{equation}

Let $m$, $j\in\nn$ and $i\in\zz$.
We first prove \eqref{t2eq2}.
By the definition of $L^{(m)}_{i,j,l}$,
and \eqref{eqsupp},
we find that, for any $l\in\nn$ with
$(2A_0)^{-3}B_{i+1,l}
\cap (2A_0)^{-3}B_{i,j}=\emptyset$,
\begin{equation}\label{t2e5}
L^{(m)}_{i,j,l}
=\lf\|\phi_{i+1,l}\r\|_{L^1(\cx)}^{-1}
\int_{\cx}b^{(m)}_{i+1,l}(z)
\mathbf{1}_{(2A_0)^{-3}B_{i+1,l}}(z)
\phi_{i,j}(z)
\mathbf{1}_{(2A_0)^{-3}B_{i,j}}(z)
\,d\mu(z)=0.
\end{equation}
We now claim that, for any $l\in\nn$ with
$(2A_0)^{-3}B_{i+1,l}
\cap (2A_0)^{-3}B_{i,j}
\neq\emptyset$,
\begin{equation}\label{t2eq5}
(2A_0)^{-3}B_{i+1,l}\subset B_{i,j}.
\end{equation}
Indeed, from \eqref{eqA0}, it follows that,
for some $y\in (2A_0)^{-3}B_{i+1,l}
\cap (2A_0)^{-3}B_{i,j}$,
\begin{equation}\label{t2eq1}
d\lf(x_{i+1,l},x_{i,j}\r)
\le A_0\lf[d\lf(x_{i+1,l},y\r)
+d\lf(y,x_{i,j}\r)\r]
\le 2A_0^2\lf(r_{i+1,l}+r_{i,j}\r).
\end{equation}
Using Lemma \ref{lewh1}(iv) with $\Oz$
replaced by $\Oz_i$, we find that
there exists some $y_{i,j}\in (\Oz_i)^{\complement}$
such that $d(x_{i,j},y_{i,j})<48A_0^5r_{i,j}$.
From this, the definition of $r_{i+1,l}$,
$y_{i,j}\in(\Oz_{i+1})^{\complement}$,
\eqref{eqA0},
and \eqref{t2eq1}, we deduce that
\begin{align*}
32A^5_0r_{i+1,l}
&=d\lf(x_{i+1,l},(\Oz_{i+1})^{\complement}\r)
\le d\lf(x_{i+1,l},y_{i,j}\r)
\le A_0\lf[d\lf(x_{i+1,l},x_{i,j}\r)+
d\lf(x_{i,j},y_{i,j}\r)\r]\\
&\le 2A^3_0\lf(r_{i+1,l}+r_{i,j}
+24A^3_0r_{i,j}\r),
\end{align*}
which implies that $r_{i+1,l}\le
\frac{1+24A_0^3}{16A_0^2-1}r_{i,j}
<2A_0r_{i,j}$.
By this, \eqref{eqA0}, and \eqref{t2eq1},
we find that, for any $x\in (2A_0)^{-3}B_{i+1,l}$,
\begin{align*}
d\lf(x,x_{i,j}\r)
&\le A_0\lf[d\lf(x,x_{i+l,l}\r)+d\lf(x_{i+1,l},x_{i,j}\r)\r]\\
&\le 2A_0^2\lf(r_{i+l,l}+A_0r_{i+1,l}
+A_0r_{i,j}\r)<16A_0^4r_{i,j},
\end{align*}
which completes the proof of the above claim \eqref{t2eq5}.
From the definition of $h^{(m)}_{i,j}$,
\eqref{eqsupp}, \eqref{t2e5}, and \eqref{t2eq5}, we deduce that,
for any $x\in (B_{i,j})^\complement$,
\begin{align*}
h^{(m)}_{i,j}(x)
&=b^{(m)}_{i,j}(x)-\sum_{l\in\nn}
\lf[b^{(m)}_{i+1,l}(x)\phi_{i,j}(x)-
L^{(m)}_{i,j,l}\phi_{i+1,l}(x)\r]\\
&=\sum_{\{l\in\nn:\
(2A_0)^{-3}B_{i+1,l}
\cap (2A_0)^{-3}B_{i,j}
\neq\emptyset\}}
L^{(m)}_{i,j,l}
\phi_{i+1,l}(x)=0,
\end{align*}
which completes the proof of \eqref{t2eq2}.

Let $m$, $j\in\nn$
and $i\in\zz$.
Next, we prove \eqref{t2e6}.
Observe that $\int_{\cx}b^{(m)}_{i,j}(z)\,d\mu(z)=0$
and, for any $l\in\nn$,
\begin{equation*}
\int_{\cx}\lf[b^{(m)}_{i+1,l}(z)\phi_{i,j}(z)
-L^{(m)}_{i,j,l}\phi_{i+1,l}(z)\r]\,d\mu(z)=0.
\end{equation*}
From this, the estimate that
$\sum^\fz_{l=1}|b^{(m)}_{i+1,l}\phi_{i,j}
-L^{(m)}_{i,j,l}\phi_{i+1,l}|\ls
|P_m(f)|\phi_{i,j}+2^i\mathbf{1}_{B_{i,j}}\in L^1(\cx)$,
and the dominated convergence theorem, we deduce that
$$
\int_{\cx}h^{(m)}_{i,j}(z)\,d\mu(z)
=\int_{\cx}b^{(m)}_{i,j}(z)\,d\mu(z)
-\sum_{l\in\nn}
\int_{\cx}\lf[b^{(m)}_{i+1,l}(z)\phi_{i,j}(z)
-L^{(m)}_{i,j,l}\phi_{i+1,l}(z)\r]\,d\mu(z)=0.
$$
This finishes the proof of \eqref{t2e6}.

Using \eqref{eqbl0},
we conclude that, for any
 $i\in\zz$ and $x\in\cx$,
\begin{equation*}
\#\mathcal{J}_i(x):
=\#\lf\{j\in\nn:\
x\in B_{i,j}\r\}\le L_0.
\end{equation*}
By \eqref{t2eq2}, we find that, for any $m\in\nn$,
$i\in\zz$, $x\in\cx$, and $j\in \mathcal{J}_i(x)$,
$$h^{(m)}_{i,j}(x)=0.$$
Thus, for any $m\in\nn$, $i\in\zz$, and $x\in\cx$,
\begin{equation}\label{t1}
h^{(m)}_i(x):
=\sum_{j\in\nn}h^{(m)}_{i,j}(x)
=\sum_{j\in\mathcal{J}_i(x)}h^{(m)}_{i,j}(x)
\end{equation}
is well defined.
From \eqref{al1}, \eqref{t2eq2}, and \eqref{eqbl0},
it follows that,
for any $m\in\nn$, $i\in\zz$, and $x\in\cx$,
\begin{equation*}
\lf|h^{(m)}_i(x)\r|
\le \sum_{j\in\nn}\lf|h^{(m)}_{i,j}(x)\r|
\le \widetilde{A}2^{i}\sum_{j\in\nn}
\mathbf{1}_{B_{i,j}}(x)
\le \widetilde{A}L_02^{i}.
\end{equation*}
By this and the dominated convergence theorem,
we have, for any $m\in\nn$,
$i\in\zz$, and $\psi\in\cg_{\mathrm{b}}(\eta,\eta)$,
\begin{align}\label{als}
\lf\langle h^{(m)}_i,\psi\r\rangle
&=\int_\cx h^{(m)}_i(z)\psi(z)\,d\mu(z)
=\int_{\cx}\lf[\sum_{j\in\nn}h^{(m)}_{i,j}(z)\psi(z)\r]\,d\mu(z)\\
&=\sum_{j\in\nn}\int_{\cx}h^{(m)}_{i,j}(z)\psi(z)\,d\mu(z)
=\sum_{j\in\nn}\lf\langle h^{(m)}_{i,j},\psi\r\rangle.\notag
\end{align}
Thus, to complete the proof of Step 1),
it remains to show that,
for any $m\in\nn$ and
$\psi\in\cg_{\mathrm{b}}(\eta,\eta)$,
\begin{equation}\label{t2eq4}
\lf\langle P_m(f),\psi\r\rangle
=\sum_{i\in\zz}\lf\langle h^{(m)}_i,\psi\r\rangle.
\end{equation}

To prove \eqref{t2eq4},
we claim that, for any $m\in\nn$,
$i\in\zz$, and $x\in\cx$,
\begin{equation}\label{t}
h^{(m)}_{i}(x)
=\sum_{j\in\nn}b^{(m)}_{i,j}(x)
-\sum_{l\in\nn}b^{(m)}_{i+1,l}(x)
\end{equation}
pointwisely.
Indeed, from \eqref{t1}, \eqref{al2},
$\sum_{j\in\mathcal{J}_i(x)}\phi_{i,j}(x)
=\sum_{j\in\nn}\phi_{i,j}(x)=\mathbf{1}_{\Oz_i}(x)$
for any $x\in\cx$ and any given $i\in\zz$,
$\supp(b^{(m)}_{i+1,l})\subset\Oz_i$
for any $m$, $l\in\nn$ and
$i\in\zz$, and
the definition of $L^{(m)}_{i,j,l}$,
it follows that, for any $m\in\nn$,
$i\in\zz$, and $x\in\cx$,
\begin{align}\label{eqal1}
h^{(m)}_i(x)
&=\sum_{j\in\mathcal{J}_i(x)}
\lf\{b^{(m)}_{i,j}(x)-
\sum_{l\in\mathcal{L}_i(x)}
\lf[b^{(m)}_{i+1,l}(x)\phi_{i,j}(x)-
L^{(m)}_{i,j,l}\phi_{i+1,l}(x)\r]\r\}\\
&=\sum_{j\in\mathcal{J}_i(x)}b^{(m)}_{i,j}(x)
-\sum_{l\in\mathcal{L}_i(x)}
b^{(m)}_{i+1,l}(x)\mathbf{1}_{\Oz_i}(x)
+\sum_{l\in\mathcal{L}_i(x)}
\sum_{j\in\mathcal{J}_i(x)}
L^{(m)}_{i,j,l}\phi_{i+1,l}(x)\notag\\
&=\sum_{j\in\nn}b^{(m)}_{i,j}(x)
-\sum_{l\in\nn}b^{(m)}_{i+1,l}(x)
+\sum_{l\in\mathcal{L}_i(x)}
\lf\|\phi_{i+1,l}\r\|_{L^1(\cx)}^{-1}\notag\\
&\quad \times\int_{\cx}\lf[
\sum_{j\in\mathcal{J}_i(x)}
b^{(m)}_{i+1,l}(z)\phi_{i,j}(z)\r]
\,d\mu(z)\phi_{i+1,l}(x).\notag
\end{align}
Then we show that, for any $m\in\nn$,
 $i\in\zz$, $x\in\cx$, and
$l\in\mathcal{L}_i(x)$,
\begin{equation}\label{eqa1}
\int_{\cx}\lf[\sum_{j\in\mathcal{J}_i(x)}
b^{(m)}_{i+1,l}(z)\phi_{i,j}(z)\r]
\,d\mu(z)=0.
\end{equation}
Indeed, using \eqref{t2eq5}, we find that,
for any $i\in\zz$, $x\in\cx$,
$l\in \mathcal{L}_i(x)$, and
$j\in\nn\setminus\mathcal{J}_i(x)$,
\begin{align*}
(2A_0)^{-3}B_{i+1,l}
\cap (2A_0)^{-3}B_{i,j}=\emptyset.
\end{align*}
From this and \eqref{eqsupp},
we deduce that, for any $m\in\nn$,
$i\in\zz$, $x\in\cx$,
and $l\in\mathcal{L}_i(x)$,
\begin{equation*}
\int_{\cx}\lf[
\sum_{j\in\nn\setminus\mathcal{J}_i(x)}
b^{(m)}_{i+1,l}(z)\phi_{i,j}(z)\r]
\,d\mu(z)=0.
\end{equation*}
By this, $\sum_{j\in\nn}\phi_{i,j}=\mathbf{1}_{\Oz_i}$
for any $i\in\zz$, and $\supp(b^{(m)}_{i+1,l})\subset \Oz_i$ and
$\int_\cx b^{(m)}_{i+1,l}(z)\,d\mu(z)=0$ for any $m$,
$l\in\nn$ and $i\in\zz$,
we find that, for any $m\in\nn$, $i\in\zz$, $x\in\cx$,
and $l\in\mathcal{L}_i(x)$,
\begin{align*}
&\int_{\cx}\lf[\sum_{j\in\mathcal{J}_i(x)}
b^{(m)}_{i+1,l}(z)\phi_{i,j}(z)\r]
\,d\mu(z)\\
&\quad=\int_{\cx}\lf[\sum_{j\in\nn}
b^{(m)}_{i+1,l}(z)\phi_{i,j}(z)\r]
\,d\mu(z)
=\int_{\cx}b^{(m)}_{i+1,l}(z)
\mathbf{1}_{\Oz_i}(z)\,d\mu(z)=0.
\end{align*}
This finishes the proof of \eqref{eqa1}.
By \eqref{eqal1} and \eqref{eqa1},
we complete the proof of the above claim \eqref{t}.

From Proposition \ref{prdc} combined
with Remark \ref{rewan},
the first equation of \eqref{als},
and \eqref{t},
we infer that,
for any $m$, $N\in\nn$ and
$\psi\in\cg_{\mathrm{b}}(\eta,\eta)$,
\begin{align}\label{t2e1}
&\lf\langle P_m(f),\psi\r\rangle-
\sum^{N}_{i=-N}\lf\langle h^{(m)}_i,\psi\r\rangle\\
&\quad=\int_\cx\lf[P_m(f)(z)-\sum^{N}_{i=-N}h^{(m)}_i(z)\r]
\psi(z)\,d\mu(z)\notag\\
&\quad=\int_\cx\lf[P_m(f)(z)-\sum_{j\in\nn} b^{(m)}_{-N,j}(z)
+\sum_{j\in\nn} b^{(m)}_{N+1,j}(z)\r]\psi(z)\,d\mu(z).\notag
\end{align}
By the H\"{o}lder inequality,
the definition of $b^{(m)}_{-N,j}$,
$\sum_{j\in\nn}\phi_{-N,j}=\mathbf{1}_{\Oz_{-N}}$
for any $N\in\nn$,
the size condition of $\cg(\eta,\eta)$ together with
Lemma \ref{lees}(ii), the estimate that
$|P_m(f)|\ls f^*\sim \fs$ for any $m\in\nn$, and \eqref{t2eq3},
we conclude that,
for any $m$, $N\in\nn$ and
$\psi\in\cg_{\mathrm{b}}(\eta,\eta)$,
\begin{align*}
&\lf|\int_\cx \lf[P_m(f)(z)-\sum_{j\in\nn}
b^{(m)}_{-N,j}(z)\r]\psi(z)\,d\mu(z)\r|\\
&\quad\le \lf\|P_m(f)-\sum_{j\in\nn}
\lf[P_m(f)-W^{(m)}_{-N,j}\r]\phi_{-N,j}
\r\|_{L^{\fz}(\cx)}\|\psi\|_{L^1(\cx)}\\
&\quad\ls\lf[\lf\|P_m(f)\mathbf{1}_{(\Oz_{-N})
^{\complement}}\r\|_{L^{\fz}(\cx)}
+\lf\|\sum_{j\in\nn}W^{(m)}_{-N,j}
\phi_{-N,j}\r\|_{L^\fz(\cx)}\r]
\|\psi\|_{\cg(\bz,\gz)}\\
&\quad\ls 2^{-N}\|\psi\|_{\cg(\bz,\gz)}
\end{align*}
with all the implicit positive constants independent of
$m$, $N$, and $\psi$. Thus,
we have, for any $m\in\nn$ and $\psi\in
\cg_{\mathrm{b}}(\eta,\eta)$,
$$
\lim_{N\to\fz}
\lf|\int_\cx \lf[P_m(f)(z)-\sum_{j\in\nn}
b^{(m)}_{-N,j}(z)\r]\psi(z)\,d\mu(z)\r|=0.
$$
By this and $\eqref{t2e1}$,
to finish the proof of $\eqref{t2eq4}$,
it suffices to show that, for any
$m\in\nn$ and $\psi\in\cg_{\mathrm{b}}(\eta,\eta)$,
\begin{equation}\label{t5}
\lim_{N\to\fz}\lf|\int_\cx\lf[\sum_{j\in\nn}
b^{(m)}_{N+1,j}(z)\psi(z)\r]\,d\mu(z)\r|=0.
\end{equation}

Now, for any given $m$, $N\in\nn$ and
$\psi\in \cg_{\mathrm{b}}(\eta,\eta)$, let
$$
\mathrm{I}:=\int_\cx\lf[\sum_{j\in\nn}
b^{(m)}_{N+1,j}(z)\psi(z)\r]\,d\mu(z).
$$
From the estimate that $\sum_{j\in\nn}|b^{(m)}_{N+1,j}\psi|
\ls |P_m(f)\psi|+2^{N}|\psi|\in L^1(\cx)$,
the dominated convergence theorem,
and $b^{(m)}_{N+1,j}\in (\gs)'$, it follows that
\begin{equation}\label{eqi}
\mathrm{I}=\sum_{j\in\nn}\int_{\cx}
b^{(m)}_{N+1,j}(z)\psi(z)\,d\mu(z)
=\sum_{j\in\nn}\lf\langle b^{(m)}_{N+1,j},
\psi\r\rangle.
\end{equation}
Let $q_0\in(\frac{\oz}{p_-(\oz+\bz)},1)$.
By \eqref{eqi},
$\|\psi\|_{\cg(x,1,\bz,\gz)}\sim
\|\psi\|_{\cg(\bz,\gz)}$
for any $x\in B(x_0,1)$, Definition \ref{debb}(ii)
with $X(\cx)$ replaced by $X^{q_0}(\cx)$,
\eqref{t2e7} and $\eqref{eqbl0}$ with $i=N+1$,
and Remark \ref{recon} with $p$ replaced by $q_0$,
we conclude that
\begin{align*}
|\mathrm{I}|
&\le \sum_{j\in\nn}\lf|\lf\langle
b^{(m)}_{N+1,j},\psi\r\rangle\r|
\ls \inf_{x\in B(x_0,1)}
\lf\{\sum_{j\in\nn}
\lf[b^{(m)}_{N+1,j}\r]^*(x)
\r\}\|\psi\|_{\cg(\bz,\gz)}\\
&\ls \lf\|\mathbf{1}_{B(x_0,1)}
\r\|_{X^{q_0}(\cx)}^{-1}
\lf\|\sum_{j\in\nn}\lf[b^{(m)}_{N+1,j}\r]^*
\r\|_{X^{q_0}(\cx)}\|\psi\|_{\cg(\bz,\gz)}\\
&\ls \lf\|\sum_{j\in\nn}
f^*\mathbf{1}_{B_{N+1,j}}(x)+2^N
\sum_{j\in\nn}\frac{V_{r_{N+1,j}}(x_{N+1,j})}
{V_{r_{N+1,j}}(x_{N+1,j})+V(x_{N+1,j},x)}\r.\\
&\quad\lf.\times\lf[
\frac{r_{N+1,j}}{r_{N+1,j}+d(x_{N+1,j},x)}\r]^{\bz}
\mathbf{1}_{(B_{N+1,j})^{\complement}}(x)
\r\|_{X^{q_0}(\cx)}\|\psi\|_{\cg(\bz,\gz)}\\
&\ls \lf\{\lf\|f^*\mathbf{1}_{\Oz_{N+1}}\r\|_{X^{q_0}(\cx)}
+2^{N}\lf\|\sum_{j\in\nn}\frac{V_{r_{N+1,j}}(x_{N+1,j})}
{V_{r_{N+1,j}}(x_{N+1,j})+V(x_{N+1,j},x)}\r.\r.\\
&\quad\lf.\lf.\times\lf[\frac{r_{N+1,j}}
{r_{N+1,j}+d(x_{N+1,j},x)}\r]^{\bz}
\mathbf{1}_{(B_{N+1,j})^{\complement}}(x)
\r\|_{X^{q_0}(\cx)}\r\}\|\psi\|_{\cg(\bz,\gz)},
\end{align*}
where all the implicit positive constants are
independent of $m$, $N$, and $\psi$.
From this,
Definition \ref{debb}(ii) with $X(\cx)$
replaced by $X^{q_0}(\cx)$,
Lemma \ref{lees}(iii) with $B$
replaced by $(2A_0)^{-4}B_{N+1,j}$,
Assumption \ref{asfs} with
$t=\frac{\oz}{q_0(\oz+\bz)}\in (0,p_-)$
and $s=\frac{\oz+\bz}{\oz}\in(1,\fz)$,
\eqref{eqsupp}, $f^*\sim \fs$,
and Proposition \ref{prin}(i) with
$p=1$ and $p_1=q_0$, we deduce that
\begin{align*}
|\mathrm{I}|
&\ls \lf\{\lf\|f^*\mathbf{1}_{\Oz_{N+1}}\r\|_{X^{q_0}(\cx)}
+2^{N}\lf\|\sum_{j\in\nn}\lf[\cm\lf(\mathbf{1}
_{(2A_0)^{-4}B_{N+1,j}}\r)\r]^{\frac{\oz+\bz}
{\oz}}\r\|_{X^{q_0}(\cx)}\r\}\|\psi\|_{\cg(\bz,\gz)}\\
&\ls \lf[\lf\|f^*\mathbf{1}_{\Oz_{N+1}}\r\|_{X^{q_0}(\cx)}
+2^{N}\lf\|\sum_{j\in\nn}\mathbf{1}_{(2A_0)^{-4}B_{N+1,j}}
\r\|_{X^{q_0}(\cx)}\r]\|\psi\|_{\cg(\bz,\gz)}\\
&\ls \lf[\lf\|f^*\mathbf{1}_{\Oz_{N+1}}\r\|_{X^{q_0}(\cx)}
+2^{N}\lf\|\mathbf{1}_{\Oz_{N+1}}\r\|_{X^{q_0}(\cx)}\r]
\|\psi\|_{\cg(\bz,\gz)}\\
&\ls \lf\|\fs\mathbf{1}_{\Oz_{N+1}}\r\|_{X^{q_0}(\cx)}
\|\psi\|_{\cg(\bz,\gz)}
\ls 2^{(1-1/q_0)N}\lf\|\fs\r\|^{1/q_0}_{WX(\cx)}
\|\psi\|_{\cg(\bz,\gz)}\\
&\sim 2^{(1-1/q_0)N}\|f\|^{1/q_0}_{WH_X(\cx)}
\|\psi\|_{\cg(\bz,\gz)}
\end{align*}
with all the implicit positive constants
independent of $m$, $N$, and $\psi$.
Thus, we have, for
any $m\in\nn$ and $\psi\in\cg_{\mathrm{b}}(\eta,\eta)$,
\begin{equation*}
\lim_{N\to\fz}\lf|\lf\langle
\sum_{j\in\nn} b^{(m)}_{N,j},\psi\r\rangle\r|=0,
\end{equation*}
which completes the proof of \eqref{t5},
and hence of Step 1).

Step 2) In this step, we show that there
exists a sequence $\{a_{i,j}\}_{i\in\zz,j\in\nn}$
of $(X(\cx),\fz)$-atoms supported, respectively,
in balls $\{B_{i,j}\}_{i\in\zz,j\in\nn}$ such that,
for any $\psi\in\cg_{\mathrm{b}}(\eta,\eta)$,
\begin{equation}\label{t2e4}
\langle f,\psi\rangle
=\sum_{i\in\zz}\sum_{j\in\nn}
\widetilde{A}2^i\lf\|
\mathbf{1}_{B_{i,j}}\r\|_{X(\cx)}
\lf\langle a_{i,j},\psi\r\rangle,
\end{equation}
where $\widetilde{A}$ and
$\{B_{i,j}\}_{i\in\zz,j\in\nn}$
are as in Step 1).

Let $\{h^{(m)}_{i,j}\}_{i\in\zz,j\in\nn}$
be as in Step 1). For any given
$i\in\zz$ and $j\in\nn$, from the Alaoglu
theorem (see, for instance, \cite[Theorem 3.17]{r91}) and
the boundedness of
$\{\|h^{(m)}_{i,j}\|_{L^\fz(\cx)}\}_{m\in\nn}$,
we deduce that there exist a sequence
$\{m^{(i,j)}_l\}_{l\in\nn}\subset \nn$ and a function
$h^{(i,j)}\in L^\fz(\cx)$ such that
$m^{(i,j)}_l\to\fz$ as $l\to\fz$, and,
for any $g\in L^1(\cx)$,
$$
\lim_{l\to\fz}\lf\langle
h^{(m^{(i,j)}_l)}_{i,j},g\r\rangle
=\lf\langle h^{(i,j)},g\r\rangle.
$$
Using this and the diagonal rule for series, we
find a sequence
$\{m_l\}_{l\in\nn}\subset \nn$
and a sequence
$\{h_{i,j}\}_{i\in\zz,j\in\nn}\subset L^{\fz}(\cx)$
such that $m_l\to\fz$ as $l\to\fz$, and,
for any $i\in\zz$, $j\in\nn$,
and $g\in L^1(\cx)$,
\begin{equation}\label{t2eq8}
\lim_{l\to\fz}\lf\langle h^{(m_l)}_{i,j},g\r\rangle
=\lf\langle h_{i,j},g\r\rangle,
\end{equation}
which, together with \eqref{eqlfz}, \eqref{t2eq2},
and \eqref{t2e6}, implies that
$\|h_{i,j}\|_{L^{\fz}(\cx)}\le \widetilde{A} 2^i$,
$\supp(h_{i,j})\subset B_{i,j}$, and
$\int_{\cx}h_{i,j}(z)\,d\mu(z)=0$.
Thus, for any $i\in\zz$ and
$j\in\nn$, $a_{i,j}:=(\widetilde{A}2^i\|
\mathbf{1}_{B_{i,j}}\|_{X(\cx)})^{-1}h_{i,j}$ is
an $(X(\cx),\fz)$-atom supported in $B_{i,j}$.

Next, to finish the proof of
Step 2), it suffices to show that, for any
$\psi\in\cg_{\mathrm{b}}(\eta,\eta)$,
\begin{equation}\label{t2e2}
\lf\langle f,\psi\r\rangle
=\sum_{i\in\zz}\sum_{j\in\nn}
\lf\langle h_{i,j},\psi\r\rangle.
\end{equation}
To this end, fix $\psi\in\cg_{\mathrm{b}}(\eta,\eta)$.
Moreover, from \eqref{eqlfz}, it follows that,
for any $l$, $j\in\nn$ and $i\in\zz$,
$$
\lf|\lf\langle
h^{(m_l)}_{i,j},\psi\r\rangle\r|
\le \widetilde{A}2^i\int_{B_{i,j}}|
\psi(z)|\,d\mu(z).
$$
Using this, the estimate that
$\sum_{j\in\nn}\int_{B_{i,j}}|\psi(z)|\,d\mu(z)
\le L_0\|\psi\|_{L^1(\cx)}$ for any $i\in\zz$,
the dominated convergence theorem for series,
and \eqref{t2eq8} with $g=\psi$,
we conclude that, for any $i\in\zz$,
\begin{equation}\label{t2eq9}
\lim_{l\to\fz}\sum_{j\in\nn}\lf\langle
h^{(m_l)}_{i,j},\psi\r\rangle
=\sum_{j\in\nn}\lim_{l\to\fz}\lf\langle
h^{(m_l)}_{i,j},\psi\r\rangle
=\sum_{j\in\nn}\lf\langle h_{i,j},\psi\r\rangle.
\end{equation}
By the proof of \eqref{t2eq4}, we find that
there exists a positive constant $C$
such that, for any $l$, $N\in\nn$,
\begin{align*}
&\lf|\lf\langle P_{m_l}(f),\psi\r\rangle-
\sum^{N}_{i=-N}\sum_{j\in\nn}
\lf\langle h^{(m_l)}_{i,j},\psi\r\rangle\r|
\le C\|\psi\|_{\cg(\bz,\gz)}
\lf[2^{-N}+2^{(1-1/q_0)N}
\lf\|f\r\|^{1/q_0}_{WH_X(\cx)}\r],
\end{align*}
where $q_0\in(\frac{\oz}{p_-(\oz+\bz)},1)$.
From this, we deduce that, for any $\ez\in(0,\fz)$,
there exists an $N_{(\ez)}\in\nn$
such that, for any $l\in\nn$ and
$N\in\nn\cap[N_{(\ez)},\fz)$,
\begin{equation}\label{t2es1}
\lf|\lf\langle P_{m_l}(f),\psi\r\rangle-
\sum^{N}_{i=-N}\sum_{j\in\nn}
\lf\langle h^{(m_l)}_{i,j},\psi\r\rangle\r|<\ez,
\end{equation}
which further implies that, for any
$\ez\in(0,\fz)$, $l\in\nn$,
and $N_1$, $N_2\in\nn\cap[N_{(\ez)},\fz)$,
\begin{align*}
\lf|\sum^{N_1}_{i=-N_1}\sum_{j\in\nn}
\lf\langle h^{(m_l)}_{i,j},\psi\r\rangle-
\sum^{N_2}_{i=-N_2}\sum_{j\in\nn}
\lf\langle h^{(m_l)}_{i,j},\psi\r\rangle\r|<2\ez.
\end{align*}
Letting $l\to\fz$ and using \eqref{t2eq9},
we obtain, for any $\ez\in(0,\fz)$
and $N_1$, $N_2\in\nn\cap[N_{(\ez)},\fz)$,
\begin{align*}
\lf|\sum^{N_1}_{i=-N_1}\sum_{j\in\nn}
\lf\langle h_{i,j},\psi\r\rangle-
\sum^{N_2}_{i=-N_2}\sum_{j\in\nn}
\lf\langle h_{i,j},\psi\r\rangle\r|\le2\ez,
\end{align*}
which implies that
$\sum_{i\in\zz}\sum_{j\in\nn}
\langle h_{i,j},\psi\rangle$ converges
and, for any $\ez\in(0,\fz)$
and $N\in\nn\cap[N_{(\ez)},\fz)$,
\begin{align}\label{t2es2}
\lf|\sum^N_{i=-N}\sum_{j\in\nn}
\lf\langle h_{i,j},\psi\r\rangle-
\sum_{i\in\zz}\sum_{j\in\nn}
\lf\langle h_{i,j},\psi\r\rangle\r|\le2\ez.
\end{align}
From \eqref{t2eq9},
we deduce that, for any $\ez\in(0,\fz)$,
there exists an $L_{(\ez)}\in\nn$ such that,
for any $l\in \nn\cap [L_{(\ez)},\fz)$,
\begin{equation*}
\lf|\sum^{N_{(\ez)}}_{i=-N_{(\ez)}}
\sum_{j\in\nn}
\lf\langle h^{(m_l)}_{i,j},\psi\r\rangle-
\sum^{N_{(\ez)}}_{i=-N_{(\ez)}}\sum_{j\in\nn}
\lf\langle h_{i,j},\psi\r\rangle\r|<\ez.
\end{equation*}
By this, and \eqref{t2es1} and \eqref{t2es2}
with $N$ replaced by $N_{(\ez)}$,
 we conclude that, for any $\ez\in(0,\fz)$
and $l\in \nn\cap [L_{(\ez)},\fz)$,
\begin{align*}
&\lf|\langle P_{m_l}(f),\psi\rangle
-\sum_{i\in\zz}\sum_{j\in\nn}
\lf\langle h_{i,j},\psi\r\rangle\r|\\
&\quad\le\lf|\langle P_{m_l}(f),\psi\rangle
-\sum^{N_{(\ez)}}_{i=-N_{(\ez)}}\sum_{j\in\nn}
\lf\langle h^{(m_l)}_{i,j},\psi\r\rangle\r|
+\lf|\sum^{N_{(\ez)}}_{i=-N_{(\ez)}}\sum_{j\in\nn}
\lf\langle h^{(m_l)}_{i,j},\psi\r\rangle
-\sum^{N_{(\ez)}}_{i=-N_{(\ez)}}\sum_{j\in\nn}
\lf\langle h_{i,j},\psi\r\rangle\r|\\
&\quad\quad+\lf|\sum^{N_{(\ez)}}_{i=-N_{(\ez)}}
\sum_{j\in\nn}
\lf\langle h_{i,j},\psi\r\rangle
-\sum_{i\in\zz}\sum_{j\in\nn}
\lf\langle h_{i,j},\psi\r\rangle\r|\\
&\quad<4\ez.
\end{align*}
Thus, $\lim_{l\to\fz}
\langle P_{m_l}(f),\psi\rangle
=\sum_{i\in\zz}\sum_{j\in\nn}
\langle h_{i,j},\psi\rangle$,
which, combined with \eqref{t2e3}, further implies that
\begin{align*}
\langle f,\psi\rangle
=\lim_{l\to\fz}
\lf\langle P_{m_l}(f),\psi\r\rangle
=\sum_{i\in\zz}\sum_{j\in\nn}
\lf\langle h_{i,j},\psi\r\rangle.
\end{align*}
This finishes the proof of \eqref{t2e2},
and hence of Step 2).

Step 3) In this step, we show that
$f=\sum_{i\in\zz}\sum_{j\in\nn}\widetilde{A}2^i\|
\mathbf{1}_{B_{i,j}}\|_{X(\cx)}a_{i,j}$ in $(\gs)'$,
where $\widetilde{A}$ and $\{B_{i,j}\}_{i\in\zz,j\in\nn}$
are as in Step 1), and
$\{a_{i,j}\}_{i\in\zz,j\in\nn}$ as in Step 2).

Using Definition \ref{debb}(ii) together
with \eqref{eqbl0},
the definition of $WX(\cx)$,
and $f^*\sim\fs$,
we find that, for any $i\in\zz$,
\begin{align*}
2^i\lf\|\sum_{j\in\nn}
\mathbf{1}_{B_{i,j}}\r\|_{X(\cx)}
\ls 2^i\lf\|
\mathbf{1}_{\Oz_i}\r\|_{X(\cx)}
\ls \lf\|\fs\r\|_{WX(\cx)}
\sim \lf\|f^*\r\|_{WX(\cx)}
\end{align*}
with the implicit positive constants
independent of $f$.
Let $A:=L_0$ and $c:=1$. From Remark \ref{ream},
we deduce that $\sum_{i\in\zz}\sum_{j\in\nn}\widetilde{A}2^i\|
\mathbf{1}_{B_{i,j}}\|_{X(\cx)}a_{i,j}$
converges in $(\gs)'$. Using this,
\eqref{t2e4}, Proposition \ref{prd}, and
a density argument,
we conclude that, for any $\varphi\in\gs$,
$$
\langle f,\varphi\rangle
=\sum_{i\in\zz}\sum_{j\in\nn}
\widetilde{A}2^i\lf\|
\mathbf{1}_{B_{i,j}}\r\|_{X(\cx)}
\lf\langle a_{i,j},\varphi\r\rangle.
$$
This finishes the proof of Step 3),
and hence of Theorem \ref{thma}.
\end{proof}
\begin{remark}\label{reha}
\begin{itemize}
\item[(i)]
In both Theorems \ref{tham} and \ref{thma},
we need that the index $p_-$ in Assumption \ref{asfs}
is greater than $\oz/(\oz+\eta)$,
which might be the known best possible.
Some applications of Theorems \ref{tham} and \ref{thma}
are given in Section \ref{sap} below.

\item[(ii)]
If $\cx=\rn$, then the corresponding result of
Theorem \ref{tham} can be found in
\cite[Theorem 4.7]{zyyw20}.
The assumptions needed in Theorem \ref{tham}
are less than those needed in \cite[Theorem 4.7]{zyyw20};
for instance, in Theorem \ref{tham}, we use the
Aoki--Rolewicz theorem to remove
the assumption that, for any $r\in(0,\min\{1,p_-\})$,
$X^{1/r}(\cx)$ is a BBF space,
which is needed in \cite[Theorem 4.7]{zyyw20}
for the Euclidean space case.
Moreover, in Theorem \ref{tham},
we obtain the result even when
$q\in(p_0,\fz]\cap [1,\fz]$, which is better
than $q\in (p_0,\fz]\cap(1,\fz]$ needed
in \cite[Theorem 4.7]{zyyw20}.
\item[(iii)]
If $\cx=\rn$, then
the corresponding result of Theorem
\ref{thma} can be found in
\cite[Theorem 4.2]{zyyw20}.
We only have one
assumption in Theorem \ref{thma},
while, in \cite[Theorem 4.2]{zyyw20},
one needs, except an analogue one,
more additional assumptions.
\end{itemize}
\end{remark}

\section{Real Interpolation\label{sr}}

In this section, for a given BQBF space
$X(\cx)$, we obtain two real interpolation theorems:
one is between $X(\cx)$ and $L^\fz(\cx)$,
and the other one between $H_X(\cx)$ and $L^\fz(\cx)$.

Now, we state some basic concepts on the
real interpolation; see, for instance,
\cite{bl76}. For any quasi-Banach spaces
$X_1$ and $X_2$, the pair $(X_1,X_2)$ is said to be
\emph{compatible} if $X_1$ and $X_2$ continuously embed
into a Hausdorff topological vector space $X$.
Furthermore, for the compatible pair $(X_1,X_2)$, let
\begin{equation*}
X_1+X_2:=\lf\{f\in X:\
f=f_1+f_2\ \mathrm{with}\
f_1\in X_1\ \mathrm{and}\ f_2\in X_2\r\}.
\end{equation*}
For any $t\in (0,\fz)$ and
$f\in X_1+X_2$,
the \emph{Peetre K-functional}
$K(t,f;X_1,X_2)$ is defined by setting
\begin{equation*}
K(t,f;X_1,X_2):=\inf\lf\{\|f_1\|_{X_1}+
t\|f_2\|_{X_2}:\ f_1\in X_1,\
f_2\in X_2,\ \mathrm{and}\ f=f_1+f_2\r\}.
\end{equation*}
For any $\tz\in(0,1)$,
the \emph{real interpolation space
$(X_1,X_2)_{\tz,\fz}$} is defined by setting
\begin{equation*}
(X_1,X_2)_{\tz,\fz}:=\lf\{f\in X_1+X_2:\
\|f\|_{(X_1,X_2)_{\tz,\fz}}<\fz\r\},
\end{equation*}
where, for any $f\in X_1+X_2$,
$$
\|f\|_{(X_1,X_2)_{\tz,\fz}}:
=\sup_{t\in (0,\fz)}
\lf\{t^{-\tz}K(t,f;X_1,X_2)\r\}.
$$

Borrowing some ideas from the proof of
\cite[Theorem 4.1]{kv14}
(see also the proof of \cite[Theorem 4.1]{wyyz21}),
we obtain the following real interpolation theorem.
\begin{theorem}\label{thin}
Let $\tz\in (0,1)$ and $X(\cx)$ be a \emph{BQBF} space. Then
\begin{equation*}
(X(\cx), L^{\fz}(\cx))_{\tz,\fz}=
WX^{1/(1-\tz)}(\cx)
\end{equation*}
with equivalent quasi-norms.
\end{theorem}
\begin{proof}
Let all the symbols be as in the present theorem.
We first show that
\begin{equation}\label{tieq0}
WX^{1/(1-\tz)}(\cx)\subset
(X(\cx), L^{\fz}(\cx))_{\tz,\fz}.
\end{equation}
To this end,
let $f\in WX^{1/(1-\tz)}(\cx)$ and,
without loss of generality, we may assume
that $f$ is a non-zero function.
By Definition \ref{debb}(ii) and
Proposition \ref{prin}(i) with $p=1/(1-\tz)$
and $p_1=1$,
we conclude that $\lz^{-1}\|f\mathbf{1}_{
\{x\in\cx:\ |f(x)|>\lz\}}\|_{X(\cx)}$ is a
decreasing function on $\lz\in(0,\fz)$, and
$\lim_{\lz\to\fz}\lz^{-1}\|f\mathbf{1}_{
\{x\in\cx:\ |f(x)|>\lz\}}\|_{X(\cx)}=0$.
For any $t\in(0,\fz)$, let
\begin{equation*}
\Lambda(t):=\inf\lf\{\lz\in(0,\fz):\
\lz^{-1}\lf\|f\mathbf{1}_{
\{x\in\cx:\ |f(x)|>\lz\}
}\r\|_{X(\cx)}\le t\r\}.
\end{equation*}
Using Definition \ref{debb}(iii)
and the definition of $\Lambda(t)$,
we conclude that, for any $t\in(0,\fz)$,
\begin{align}\label{11}
&\lf\|f\mathbf{1}_{\{x\in\cx:
\ |f(x)|>\Lambda(t)\}}\r\|_{X(\cx)}\\
&\quad=\lim_{\lz\in(\Lambda(t),\fz),\lz\to \Lambda(t)}
\lf\|f\mathbf{1}_{\{x\in\cx:
\ |f(x)|>\lz\}}\r\|_{X(\cx)}
\le \lim_{\lz\in(\Lambda(t),\fz),\lz\to \Lambda(t)}
t\lz= t\Lambda(t).\notag
\end{align}
Observe that, for any $t\in(0,\fz)$,
$\Lambda(t)\neq 0$ and
\begin{equation}\label{eql2}
[\Lambda(t)/2]^{-1}
\lf\|f\mathbf{1}_{\{x\in\cx:\
|f(x)|>\Lambda(t)/2\}}
\r\|_{X(\cx)}>t.
\end{equation}
From the definition of
$K(t,f;X(\cx),L^{\fz}(\cx))$, \eqref{11}, \eqref{eql2},
and Proposition \ref{prin}(i) with $p$, $p_1$, and
 $\lz$ replaced, respectively, by $1/(1-\tz)$, $1$,
and $\Lambda(t)/2$, we infer that,
for any $t\in (0,\fz)$,
\begin{align*}
&t^{-\tz}K(t,f;X(\cx),L^{\fz}(\cx))\\
&\quad\le t^{-\tz}\lf[\lf\|f\mathbf{1}_{\{x\in\cx:\
|f(x)|>\Lambda(t)\}}\r\|_{X(\cx)}
+t\lf\|f\mathbf{1}_{\{x\in\cx:\
|f(x)|\le\Lambda(t)\}}\r\|_{L^{\fz}(\cx)}\r]
\ls t^{1-\tz}\Lambda(t)\\
&\quad\ls \lf\{[\Lambda(t)/2]^{-1}
\lf\|f\mathbf{1}_{\{x\in\cx:\
|f(x)|>\Lambda(t)/2\}}
\r\|_{X(\cx)}\r\}^{1-\tz}\Lambda(t)
\ls \|f\| _{WX^{1/(1-\tz)}(\cx)},
\end{align*}
where all the implicit positive constants
are independent of both $t$ and $f$.
Taking the supremum over $t\in(0,\fz)$,
we find that
$f\in (X(\cx), L^\fz(\cx))_{\tz,\fz}$ and
$\|f\|_{(X(\cx), L^\fz(\cx))_{\tz,\fz}}
\ls \|f\|_{WX^{1/(1-\tz)}(\cx)}$
with the implicit positive constant
independent of $f$.
This finishes the proof of \eqref{tieq0}.

Next, we prove
\begin{equation}\label{tieq2}
(X(\cx), L^{\fz}(\cx))_{\tz,\fz}
\subset WX^{1/(1-\tz)}(\cx).
\end{equation}
Before proving \eqref{tieq2},
we first claim that, for any $t\in(0,\fz)$
and $f\in (X(\cx), L^{\fz}(\cx))_{\tz,\fz}$,
\begin{equation}\label{tieq3}
K(t,f;X(\cx),L^\fz(\cx))
=\inf_{\tau\in[0,\fz)}
N_{t,f}(\tau),
\end{equation}
where, for any $\tau\in [0,\fz)$,
\begin{equation}\label{eqn}
N_{t,f}(\tau):=
\lf\|\max\{|f|-\tau,
0\}\r\|_{X(\cx)}
+t\lf\|\min\{|f|,\tau\}\r\|_{L^\fz(\cx)}.
\end{equation}
To this end, fix $t\in(0,\fz)$
and $f\in (X(\cx), L^{\fz}(\cx))_{\tz,\fz}$.
Now, we show that the left-hand side of \eqref{tieq3}
is not greater than the right-hand side.
Indeed, for any $\tau\in[0,\fz)$ and $x\in\cx$, let
\begin{align*}
f_{1,\tau}(x):=
\begin{cases}
\displaystyle\frac{f(x)}{|f(x)|}[|f(x)|-\tau]
\ \ &\text{if}\ \ |f(x)|>\tau,\\
\displaystyle 0
\ \ &\text{if}\ \ |f(x)|\le \tau,
\end{cases}
\end{align*}
and $f_{2,\tau}(x):=f(x)-f_{1,\tau}(x)$.
Observe that, for any $\tau\in[0,\fz)$,
\begin{align*}
K(t,f;X(\cx),L^\fz(\cx))
\le \lf\|f_{1,\tau}\r\|_{X(\cx)}
+t\lf\|f_{2,\tau}\r\|_{L^\fz(\cx)}
=N_{t,f}(\tau).
\end{align*}
Taking the infimum over $\tau\in[0,\fz)$, we obtain
the desired conclusion. On the other hand,
we show that the right-hand side of \eqref{tieq3} is not
greater than the left-hand side. Observe that,
for any $f_1\in X(\cx)$
and $f_2\in L^\fz(\cx)$ satisfying $f=f_1+f_2$,
\begin{align*}
\inf_{\tau\in[0,\fz)} N_{t,f}(\tau)
\le N_{t,f}(\|f_2\|_{L^\fz(\cx)})
\le \|f_1\|_{X(\cx)}+t\|f_2\|_{L^\fz(\cx)}.
\end{align*}
Taking the infimum over $f_1$ and $f_2$
as above, we obtain the desired conclusion.
Using these two estimates, we then complete the
proof of the above claim \eqref{tieq3}.

Next, we show that, for any given $\lz\in(0,\fz)$ and
$f\in (X(\cx), L^{\fz}(\cx))_{\tz,\fz}$,
\begin{equation}\label{eqdlz}
D_f(\lz):=\lf\|\mathbf{1}_{\{x\in\cx:\
|f(x)|>\lz\}}\r\|_{X(\cx)}<\fz.
\end{equation}
If this is not true, then there exists a $\lz_0\in (0,\fz)$
such that $D_f(\lz_0)=\fz$.
Observe that, for any $t\in (0,\fz)$,
 if $\tau\in[0,\lz_0)$, then
$$
N_{t,f}(\tau)
\ge\|\max\{|f|-\tau,0\}\|_{X(\cx)}
\ge (\lz_0-\tau)D_f(\lz_0)=\fz;
$$
if $\tau\in[\lz_0,\fz)$, then
$$
N_{t,f}(\tau)
\ge t\|\min\{|f|,\tau\}\|_{L^\fz(\cx)}\ge t\lz_0.
$$
Using \eqref{tieq3} and taking the infimum
over $\tau\in[0,\fz)$,
we conclude that, for any $t\in (0,\fz)$,
$$
K(t,f;X(\cx),L^\fz(\cx))
=\inf_{\tau\in [0,\fz)}N_{t,f}(\tau)\ge t\lz_0.
$$
Thus, $\|f\|_{(X(\cx),L^\fz(\cx))_{\tz,\fz}}
\ge\sup_{t\in(0,\fz)}t^{-\tz}t\lz_0=\fz$,
which is a contradiction.
This finishes the proof of \eqref{eqdlz}.

Observe that, for any
$f\in (X(\cx), L^{\fz}(\cx))_{\tz,\fz}$,
\begin{equation}\label{eqwtz}
\|f\|_{WX^{1/(1-\tz)}(\cx)}=\sup_{\lz\in(0,\fz)}
\lf\{\lz\lf[D_f(\lz)\r]^{1-\tz}\r\}.
\end{equation}
Now, we show that, for any
$f\in (X(\cx), L^{\fz}(\cx))_{\tz,\fz}$
and $\lz\in(0,\fz)$,
\begin{equation}\label{tieq1}
\lz\lf[D_f(\lz)\r]^{1-\tz}
\le 2\|f\|_{(X(\cx),L^\fz(\cx))_{\tz,\fz}}.
\end{equation}
To this end, without loss of generality,
we may assume that $D_f(\lz)\neq 0$.
If $\tau\in[0,\lz/2)$, from the definition of $D_f(\lz)$,
and Definition \ref{debb}(ii),
it follows that
\begin{align*}
\lz\lf[D_f(\lz)\r]^{1-\tz}
&\le \lz\lf[D_f(\lz)\r]^{-\tz}
\lf\|\mathbf{1}_{\{x\in\cx:\
|f(x)|>\tau+\lz/2\}}\r\|_{X(\cx)}\\
&= 2\lf[D_f(\lz)\r]^{-\tz}[\lz/2]
\lf\|\mathbf{1}_{\{x\in\cx:\
\max\{|f(x)|-\tau,0\}>\lz/2\}}\r\|_{X(\cx)}\\
&\le 2\lf[D_f(\lz)\r]^{-\tz}
\lf\|\max\{|f|-\tau,0\}\r\|_{X(\cx)}
\le 2\lf[D_f(\lz)\r]^{-\tz}N_{D_f(\lz),f}(\tau),
\end{align*}
where $N_{D_f(\lz),f}$ is as in \eqref{eqn}
with $t$ replaced by $D_f(\lz)$;
if $\tau\in [\lz/2,\fz)$,
from $D_f(\lz)\neq 0$, it follows that
$\mu(\{x\in\cx:\ |f(x)|>\lz\})\neq 0$ and
\begin{align*}
\lz\lf[D_f(\lz)\r]^{1-\tz}
\le 2\lf[D_f(\lz)\r]^{1-\tz}
\lf\|\min\{f,\tau\}\r\|_{L^\fz(\cx)}
\le 2\lf[D_f(\lz)\r]^{-\tz}N_{D_f(\lz),f}(\tau),
\end{align*}
where $N_{D_f(\lz),f}$ is as in \eqref{eqn}
with $t$ replaced by $D_f(\lz)$.
Taking the infimum over $\tau\in[0,\fz)$, and
using \eqref{tieq3}
with $t$ replaced by $D_f(\lz)$, we conclude that
\begin{align*}
\lz\lf[D_f(\lz)\r]^{1-\tz}
&\le 2\lf[D_f(\lz)\r]^{-\tz}
\inf_{\tau\in[0,\fz)}N_{D_f(\lz),f}(\tau)\\
&= 2\lf[D_f(\lz)\r]^{-\tz}
K(D_f(\lz),f;X(\cx),L^\fz(\cx))
\le 2\|f\|_{(X(\cx),L^\fz(\cx))_{\tz,\fz}},
\end{align*}
which implies \eqref{tieq1}.
From \eqref{eqwtz} and \eqref{tieq1},
we deduce that $f\in WX^{1/(1-\tz)}(\cx)$ and
$$\|f\|_{WX^{1/(1-\tz)}(\cx)}\le
2\|f\|_{(X(\cx),L^\fz(\cx))_{\tz,\fz}}.$$
This finishes the proof of \eqref{tieq2}, and
hence of Theorem \ref{thin}.
\end{proof}

Next, we establish the real interpolation
between the Hardy space $H_X(\cx)$ and the
Lebesgue space $L^\fz(\cx)$.

\begin{theorem}\label{thhin}
Let $\tz\in(0,1)$,
$\oz$ be as in \eqref{eqoz},
$\eta$ as in \eqref{eqwa}, and
$X(\cx)$  a \emph{BQBF} space
satisfying Assumption \ref{asfs} with $p_-
\in (\oz/(\oz+\eta),\fz)$, and Assumption \ref{asas}
with $s_0\in(\oz/(\oz+\eta),\min\{p_-,1\})$
and $p_0\in(s_0,\fz)$. Then
\begin{equation*}
(H_X(\cx), L^{\fz}(\cx))_{\tz,\fz}=
WH_{X^{1/(1-\tz)}}(\cx)
\end{equation*}
with equivalent quasi-norms as subspaces of
$(\gs)'$ with $\bz$, $\gz\in(\oz[1/s_0-1],\eta)$.
\end{theorem}

To prove this theorem, we need the atomic characterization
of $H_X(\cx)$; see, for instance, \cite[Section 4]{yhyy21a}.
Recall that, for any $\bz$, $\gz\in(0,\eta)$,
$q\in [1,\fz]$, $s\in(0,1]$, and any
BQBF space $X(\cx)$,
the \emph{atomic Hardy space}
$H^{X,q,s}_{\mathrm{at}}(\cx)$ is defined to be the set of
all functionals $f$ in $(\gs)'$ satisfying that there
exist  a sequence $\{\lz_j\}_{j\in\nn}\subset (0,\fz)$ and
a sequence $\{a_j\}_{j\in\nn}$ of
$(X(\cx),q)$-atoms supported,
respectively, in balls $\{B_j\}_{j\in\nn}$ such that
$$
\lf\|\lf\{\sum_{j\in\nn}\lf[\frac{\lz_j}
{\|\mathbf{1}_{B_j}\|_{X(\cx)}}\r]^{s}\mathbf{1}_{B_j}\r\}^{1/s}
\r\|_{X(\cx)}<\fz$$
and $f=\sum_{j\in\nn}\lz_ja_j$ in $(\gs)'$.
Furthermore, for any
$f\in H^{X,q,s}_{\mathrm{at}}(\cx)$, let
\begin{equation*}
\|f\|_{H^{X,q,s}_{\mathrm{at}}(\cx)}
:=\inf\lf\|\lf\{\sum_{j\in\nn}\lf[\frac{\lz_j}
{\|\mathbf{1}_{B_j}\|_{X(\cx)}}
\r]^{s}\mathbf{1}_{B_j}\r\}^{1/s}
\r\|_{X(\cx)},
\end{equation*}
where the infimum is taken over all the
decompositions of $f$ as above.
The following lemma is just
\cite[Theorem 4.19]{yhyy21a} which reveals
the relation between $H_X(\cx)$ and
$H^{X,q,s}_{\mathrm{at}}(\cx)$.

\begin{lemma}\label{leha}
Let $X(\cx)$, $p_-$, $s_0$, $p_0$, $\bz$, and
$\gz$ be as in Theorem \ref{thhin}, and $q\in(
\max\{1,p_0\},\fz]$. Then
$H_X(\cx)=H^{X,q,s_0}_{\mathrm{at}}(\cx)$
with equivalent quasi-norms
as subspaces of $(\gs)'$.
\end{lemma}

Mainly using the above lemma and Theorem \ref{thma},
we can now prove the next lemma.

\begin{lemma}\label{lehin}
Let $X(\cx)$, $p_-$, $s_0$, $p_0$, $\bz$, and
$\gz$ be as in Theorem \ref{thhin}, $\tz\in(0,1)$,
and $f\in (\gs)'\cap WH_{X^{1/(1-\tz)}}(\cx)$.
Then, for any $\lz\in(0,\fz)$, there exist a
$g_{\lz}\in L^{\fz}(\cx)$ and a $b_{\lz}\in H_X(\cx)$
such that $f=g_\lz+b_\lz$ in $(\gs)'$. Furthermore,
$\|g_\lz\|_{L^{\fz}(\cx)}\le C\lz$ and
$$\|b_\lz\|_{H_X(\cx)}\le C\lf\|\fs\mathbf{1}
_{\{x\in\cx:\ \fs(x)>\lz\}}\r\|_{X(\cx)},$$
where $\fs$ is as in \eqref{eqfs} and
$C$ a positive constant
independent of both $f$ and $\lz$.
\end{lemma}

\begin{proof}
Let all the symbols be as in the present lemma.
Observe that $X^{1/(1-\tz)}(\cx)$ satisfies
Assumption \ref{asfs} with $X(\cx)$ and $p_-$ replaced,
respectively, by $X^{1/(1-\tz)}(\cx)$ and $p_-/(1-\tz)$.
Using Theorem \ref{thma} with $X(\cx)$, $WX(\cx)$, and
$\widetilde{A}2^i\|\mathbf{1}_{B_{i,j}}\|_{X(\cx)}a_{i,j}$
replaced, respectively, by $X^{1/(1-\tz)}(\cx)$,
$WX^{1/(1-\tz)}(\cx)$, and $h_{i,j}$, we obtain
\begin{equation*}
f=\sum_{i\in\zz}\sum_{j\in\nn}h_{i,j}
\end{equation*}
in $(\gs)'$, where, for any $i\in\zz$ and $j\in\nn$,
$h_{i,j}$ supports in a ball $B_{i,j}$,
and $\|h_{i,j}\|_{L^\fz(\cx)}\le \widetilde{A}2^i$
with $\widetilde{A}$ being a positive constant
independent of $f$, $i$, and $j$;
furthermore, we have, for any $i\in\zz$,
\begin{equation}\label{lheq2}
\bigcup_{j\in\nn}B_{i,j}=\lf\{x\in\cx:\
\fs(x)>2^i\r\}=:\Omega_i
\end{equation}
with $\fs$ as in \eqref{eqfs}, and
\begin{equation}\label{theq1}
\sum_{j\in\nn}\mathbf{1}_{B_{i,j}}
\le A\mathbf{1}_{\Oz_i},
\end{equation}
where $A$ is a positive constant
independent of $i$ and $f$.

Now, we fix $\lz\in(0,\fz)$, let
$i_0\in\zz$ be such that
$2^{i_0-1}\le \lz<2^{i_0}$,
and decompose
\begin{equation*}
f=\sum^{i_0-1}_{i=-\fz}\sum_{j\in\nn}
h_{i,j}+
\sum^{\fz}_{i=i_0}\sum_{j\in\nn}
h_{i,j}
=:g_\lz+b_\lz
\end{equation*}
in $(\gs)'$.

We first estimate $\|g_\lz\|_{L^\fz(\cx)}$.
Indeed, from the estimate that $|h_{i,j}|\le
\widetilde{A}2^i\mathbf{1}_{B_{i,j}}$
for any $i\in\zz$ and $j\in\nn$,
\eqref{theq1}, and $2^{i_0}\sim\lz$,
it follows that
\begin{align*}
\lf\|g_\lz\r\|_{L^{\fz}(\cx)}
\ls \lf\|\sum^{i_0-1}_{i=-\fz}
\sum_{j\in\nn}2^i
\mathbf{1}_{B_{i,j}}\r\|
_{L^{\fz}(\cx)}
\ls \lf\|\sum^{i_0-1}_{i=-\fz}2^i
\mathbf{1}_{\Oz_i}\r\|
_{L^{\fz}(\cx)}
\ls \sum^{i_0-1}_{i=-\fz}2^{i}\sim \lz
\end{align*}
with all the implicit positive constants
independent of both $f$ and $\lz$.

Next, we estimate $\|b_\lz\|_{H_X(\cx)}$.
For any $i\in\zz$ and $j\in\nn$, let
$\lz_{i,j}:=\widetilde{A}2^i\|\mathbf{1}_{B_{i,j}}
\|_{X(\cx)}$. Observe that $(\lz_{i,j})^{-1}h_{i,j}$
is an $(X(\cx),\fz)$-atom supported in $B_{i,j}$.
By this, Lemma \ref{leha},
 \eqref{theq1}, Definition \ref{debb}(ii),
\eqref{lheq2},
and $\lz<2^{i_0}$, we conclude that
\begin{align*}
\lf\|b_\lz\r\|_{H_X(\cx)}
&\ls \lf\|\lf\{\sum^{\fz}_{i=i_0}
\sum_{j\in\nn}
\lf[\frac{\lz_{i,j}}{\|\mathbf{1}
_{B_{i,j}}\|_{X(\cx)}}\r]^{s_0}
\mathbf{1}_{B_{i,j}}
\r\}^{1/s_0}\r\|_{X(\cx)}\\
&\sim \lf\|\lf(\sum^{\fz}_{i=i_0}
\sum_{j\in\nn}2^{is_0}
\mathbf{1}_{B_{i,j}}
\r)^{1/s_0}\r\|_{X(\cx)}
\ls \lf\|\lf(\sum^{\fz}_{i=i_0}
2^{is_0}\mathbf{1}_{\Omega_i}
\r)^{1/s_0}\r\|_{X(\cx)}\\
&\sim \lf\|\lf(\sum^{\fz}_{i=i_0}
\sum^{\fz}_{j=i}
2^{is_0}\mathbf{1}_{\Omega_j
\setminus \Omega_{j+1}}
\r)^{1/s_0}\r\|_{X(\cx)}
\ls \lf\|\lf(\sum^{\fz}_{j=i_0}
2^{js_0}\mathbf{1}_{\Omega_j
\setminus \Omega_{j+1}}
\r)^{1/s_0}\r\|_{X(\cx)}\\
&\ls \lf\|\fs\mathbf{1}
_{\{x\in\cx:\ \fs (x)> 2^{i_0}\}}
\r\|_{X(\cx)}
\ls \lf\|\fs\mathbf{1}
_{\{x\in\cx:\ \fs (x)>\lz\}}
\r\|_{X(\cx)},
\end{align*}
where all the implicit positive constants
are independent of both $f$ and $\lz$.
This finishes
the proof of Lemma \ref{lehin}.
\end{proof}

Now, we use Lemma \ref{lehin}
and Theorem \ref{thin} to
prove Theorem \ref{thhin}.

\begin{proof}[Proof of Theorem \ref{thhin}]
Let all the symbols be as in the present theorem.
We first show that
\begin{equation}\label{tie0}
WH_{X^{1/(1-\tz)}}(\cx)\subset
(H_X(\cx), L^{\fz}(\cx))
_{\tz,\fz}.
\end{equation}
To this end, let $f\in WH_{X^{1/(1-\tz)}}(\cx)$ and,
without loss of generality,
we may assume that $f$ is a non-zero functional.
Recall that $\fs$ is as in \eqref{eqfs}.
From Definition \ref{debb}(ii) and
Proposition \ref{prin}(i) with $p$, $p_1$, and
$f$ replaced, respectively, by $1/(1-\tz)$, $1$, and $\fs$,
we deduce that
$$
\lz^{-1}\lf\|\fs\mathbf{1}_{
\{x\in\cx:\ |\fs(x)|>\lz\}}\r\|_{X(\cx)}
$$
is a decreasing function on $\lz\in(0,\fz)$, and
$\lim_{\lz\to\fz}\lz^{-1}\|\fs\mathbf{1}_{
\{x\in\cx:\ |\fs(x)|>\lz\}}\|_{X(\cx)}=0.$
For any $t\in(0,\fz)$, let
\begin{equation*}
\Lambda(t):=\inf\lf\{\lz\in(0,\fz):\
\lz^{-1}\lf\|\fs\mathbf{1}_{
\{x\in\cx:\ |\fs(x)|>\lz\}
}\r\|_{X(\cx)}\le t\r\}.
\end{equation*}
Observe that, for any $t\in(0,\fz)$,
we have $\|\fs\mathbf{1}_{
\{x\in\cx:\ |\fs(x)|>\Lambda(t)\}
}\|_{X(\cx)}\le t\Lambda(t)$,
$\Lambda(t)>0$, and
\begin{equation*}
[\Lambda(t)/2]^{-1}
\lf\|\fs\mathbf{1}_{\{x\in\cx:\
|\fs(x)|>\Lambda(t)/2\}}
\r\|_{X(\cx)}>t.
\end{equation*}
Using this, the definition of
$K(t,f;H_X(\cx),L^{\fz}(\cx))$,
Lemma \ref{lehin} with
$\lz$ replaced by $\Lambda(t)$,
Proposition \ref{prin}(i) with $p$, $p_1$, $\lz$,
and $f$ replaced, respectively, by $1/(1-\tz)$,
$1$, $\Lambda(t)/2$, and $\fs$, and
$f^*\sim \fs$,
we conclude that, for any $t\in(0,\fz)$,
\begin{align*}
&t^{-\tz}K(t,f;H_X(\cx),L^{\fz}(\cx))\\
&\quad\le t^{-\tz}\lf\|b_{\Lambda(t)}\r\|_{H_X(\cx)}
+t^{1-\tz}\lf\|g_{\Lambda(t)}\r\|_{L^{\fz}(\cx)}\\
&\quad\ls t^{-\tz}\lf\|\fs\mathbf{1}_{
\{x\in\cx:\ |\fs(x)|>\Lambda(t)\}
}\r\|_{X(\cx)}+t^{1-\tz}\Lambda(t)
\ls t^{1-\tz}\Lambda(t)\\
&\quad\ls
\lf\{\lf[\Lambda(t)/2\r]^{-1}
\lf\|\fs\mathbf{1}_{\{x\in\cx:\
|\fs(x)|>\Lambda(t)/2\}}\r\|_{X(\cx)}
\r\}^{1-\tz}\Lambda(t)\\
&\quad\ls \lf\|\fs\r\| _{WX^{1/(1-\tz)}(\cx)}
\sim \|f\| _{WH_{X^{1/(1-\tz)}}(\cx)},
\end{align*}
where $b_{\Lambda(t)}$ and
$g_{\Lambda(t)}$ are as
in Lemma \ref{lehin} with
$\lz$ replaced by $\Lambda(t)$, and
all the implicit positive constants
independent of both $f$ and $t$.
Taking the supremum over $t\in(0,\fz)$,
we obtain $f\in (H_X(\cx),L^\fz(\cx))_{\tz,\fz}$
and
$$
\|f\|_{(H_X(\cx),L^\fz(\cx))_{\tz,\fz}}
\ls \|f\| _{WH_{X^{1/(1-\tz)}}(\cx)}
$$
with
the implicit positive constant independent
of $f$. This finishes the proof of \eqref{tie0}.

Next, we show that
\begin{equation}\label{tie1}
(H_X(\cx), L^{\fz}(\cx))_{\tz,\fz}
\subset WH_{X^{1/(1-\tz)}}(\cx).
\end{equation}
Indeed, for any $f\in
(H_X(\cx), L^{\fz}(\cx))_{\tz,\fz}$, let
$$
\mathcal{F}_f:=
\lf\{(f_1,f_2):\ f_1\in H_X(\cx),\
f_2\in L^{\fz}(\cx),\ f=f_1+f_2\r\}.
$$
From the definition of
$K(t,f^*;X(\cx),L^{\fz}(\cx))$,
the estimate that
$f^*\le (f_1)^*+(f_2)^*$ for any
$f\in (H_X(\cx), L^{\fz}(\cx))_{\tz,\fz}$
and any $(f_1,f_2)\in\mathcal{F}_f$,
Definition \ref{debb}(ii),
and the estimate that $\|f^*\|_{L^{\fz}(\cx)}
\ls \|f\|_{L^{\fz}(\cx)}$ for any $f\in L^{\fz}(\cx)$,
it follows that, for any $t\in (0,\fz)$,
$f\in (H_X(\cx), L^{\fz}(\cx))_{\tz,\fz}$,
and $(f_1,f_2)\in\mathcal{F}_f$,
\begin{align*}
&K(t,f^*;X(\cx),L^{\fz}(\cx))\\
&\quad\le \lf\|f^*\mathbf{1}_{\{x\in\cx:
\ f^*(x)\le 2(f_1)^*(x)\}}\r\|_{X(\cx)}
+t\lf\|f^*\mathbf{1}_{\{x\in\cx:\
f^*(x)> 2(f_1)^*(x)\}}\r\|_{L^{\fz}(\cx)}\\
&\quad\le \lf\|f^*\mathbf{1}_{\{x\in\cx:
\ f^*(x)\le 2(f_1)^*(x)\}}\r\|_{X(\cx)}
+t\lf\|f^*\mathbf{1}_{\{x\in\cx:\
f^*(x)\le 2(f_2)^*(x)\}}\r\|_{L^{\fz}(\cx)}\\
&\quad\ls\lf\|(f_1)^*\r\|_{X(\cx)}+
t\lf\|(f_2)^*\r\|_{L^{\fz}(\cx)}
\ls \lf\|f_1\r\|_{H_X(\cx)}
+t\lf\|f_2\r\|_{L^{\fz}(\cx)}.
\end{align*}
Taking the infimum over $(f_1,f_2)\in\mathcal{F}_f$,
we obtain,
for any $t\in (0,\fz)$
and $f\in (H_X(\cx), L^{\fz}(\cx))_{\tz,\fz}$,
\begin{equation*}
K(t,f^*;X(\cx),L^{\fz}(\cx))
\ls K(t,f;H _X(\cx),L^{\fz}(\cx)),
\end{equation*}
which further implies that, for any
$f\in (H_X(\cx), L^{\fz}(\cx))_{\tz,\fz}$,
$$
\|f^*\|_{(X(\cx), L^\fz(\cx))_{\tz,\fz}}
\ls\|f\|_{(H_X(\cx), L^\fz(\cx))_{\tz,\fz}}.
$$
From this and Theorem \ref{thin},
we infer that, for any
$f\in (H_X(\cx), L^{\fz}(\cx))_{\tz,\fz}$,
 $f^*\in WX^{1/(1-\tz)}(\cx)$
and
$$\|f\|_{WH_{X^{1/(1-\tz)}}(\cx)}
\ls \|f\|_{(H_X(\cx),L^\fz(\cx))_{\tz,\fz}}.$$
This finishes the proof of \eqref{tie1}, and hence
of Theorem \ref{thhin}.
\end{proof}

\begin{remark}
\begin{itemize}
\item[(i)]
We refer the reader to
\cite[Theorem 4.1]{wyyz21} for
the corresponding result of
Theorem \ref{thin} in the Euclidean space case. The result obtained
in \cite[Theorem 4.1]{wyyz21}
is a special case of Theorem \ref{thin} when $\cx=\rn$.
Furthermore, in the proof
of Theorem \ref{thin}, we use the
Aoki--Rolewicz theorem and hence show that
the assumption that ``$X^{1/r}$ is
a BBF space'' of
\cite[Theorem 4.1]{wyyz21} is superfluous.
Moreover, Theorem \ref{thin} holds true
if $X(\cx)$ is replaced by any BQBF space given
in Section \ref{sap}.
\item[(ii)]
We refer the reader to
\cite[Theorem 4.5]{wyyz21} for the
corresponding result of
Theorem \ref{thhin} in the Euclidean space case. We assume that
$p_-\in(\oz/(\oz+\eta),\fz)$
in Theorem \ref{thhin}, which might be the known best possible.
\end{itemize}
\end{remark}

\section{Calder\'on--Zygmund Operators\label{sc}}

In this section, we discuss the boundedness
of Calder\'on--Zygmund operators from
the Hardy space $H_X(\cx)$ to the weak BQBF space $WX(\cx)$,
or to the weak Hardy space $WH_X(\cx)$.
We first recall the concept of
Calder\'on--Zygmund operators on
$\cx$; see, for instance,
\cite[Definition 12.1]{ah13}.

\begin{definition}\label{deke}
Let $\ez\in (0,1]$. A function
$K:\cx\times\cx\setminus\{(x,x):
\ x\in\cx\}\to\cc$ is called an
\emph{$\ez$-Calder\'on--Zygmund
kernel} if there exists a
positive constant $C_K$ such that
\begin{enumerate}
\item[(i)]
for any $x,\ y\in\cx$ with $x\neq y$,
\begin{equation*}
|K(x,y)|\le C_K\frac{1}{V(x,y)};
\end{equation*}
\item[(ii)] for any $x,\ x',\ y\in\cx$
with $x\neq y$ and
$d(x,x')\leq (2A_0)^{-1}d(x,y)$,
\begin{equation*}
\lf|K(x,y)-K(x',y)\r|+\lf|K(y,x)-K(y,x')\r|
\leq C_K\left[
\frac{d(x,x')}{d(x,y)}\right
]^{\ez}\frac{1}{V(x,y)}.
\end{equation*}
\end{enumerate}
A linear operator $T$,
which is bounded on $L^2(\cx)$,
is called
an \emph{$\ez$-Calder\'on--Zygmund operator} if
\begin{itemize}
\item[(iii)]
there exists an
$\ez$-Calder\'on--Zygmund kernel $K$
such that, for any $f\in L^2(\cx)$ with
bounded support, and $x\notin \supp(f)$,
\begin{equation*}
T(f)(x)=\int_{\cx}K(x,y)f(y)\,d\mu(y).
\end{equation*}
\end{itemize}
\end{definition}
For any
$\ez$-Calder\'on--Zygmund operator $T$
with $\ez\in(0,1]$,
we sometimes need to assume $T^*1=0$
which means that, for any $a\in L^2(\cx)$ with
bounded support, and $\int_{\cx}a(z)\,d\mu(z)=0$, one has
$$\int_{\cx}T(a)(z)\,d\mu(z)=0.$$
We also need the following assumption to obtain the main
theorems of this section.
\begin{assumption}\label{asw}
Let $\tau\in(0,\fz)$, $\oz$ be as in \eqref{eqoz},
 and $X(\cx)$ a BQBF space. Assume that
there exists a positive constant $C$ such that, for any sequence
$\{f_j\}_{j\in\nn}\subset \mathscr{M}(\cx)$,
\begin{equation*}
\lf\|\lf\{\sum_{j\in\nn}\lf[\cm\lf(f_j\r)\r]
^{\frac{\oz+\tau}{\oz}}\r\}
^\frac{\oz}{\oz+\tau}\r\|
_{WX^{\frac{\oz+\tau}{\oz}}(\cx)}
\le C\lf\|\lf(\sum_{j\in\nn}\lf|f_j\r|
^{\frac{\oz+\tau}{\oz}}\r)
^{\frac{\oz}{\oz+\tau}}\r\|
_{X^{\frac{\oz+\tau}{\oz}}(\cx)}.
\end{equation*}
\end{assumption}
\begin{remark}\label{rew}
By \cite[Theorem 1.2]{gly09}, we find that,
if $X(\cx):=L^{\oz/(\oz+\tau)}(\cx)$
in Assumption \ref{asw},
then Assumption \ref{asw} holds true.
Furthermore, see Section \ref{sap} for more
function spaces satisfying Assumption \ref{asw}.
\end{remark}
Now, we have the following boundedness of
Calder\'{o}n--Zygmund operators.
\begin{theorem}\label{thcz}
Let $\oz$ be as in \eqref{eqoz},
$\eta$ as in \eqref{eqwa}, $\ez\in(0,\eta)$,
$T$ an $\ez$-Calder\'on--Zygmund operator,
and $X(\cx)$ a \emph{BQBF} space.
Suppose that $X(\cx)$ has an
absolutely continuous quasi-norm
and satisfies Assumption \ref{asfs} with
$p_-=\oz/(\oz+\ez)$,
Assumption \ref{asas} with
$s_0\in(\oz/(\oz+\eta),p_-)$
and $p_0\in(s_0,\fz)$, and
Assumption \ref{asw} with $\tau=\ez$.
\begin{itemize}
\item[\textup{(i)}]
Then $T$ can be extended to a unique bounded
linear operator from $H_X(\cx)$ to $WX(\cx)$.
\item[\textup{(ii)}]
If $T^*1=0$,
then $T$ can be extended to a unique bounded
linear operator from $H_X(\cx)$
to $WH_X(\cx)$.
\end{itemize}
\end{theorem}

To prove this theorem, we need to recall the concept of
the \emph{finite atomic Hardy space}.
For any $q\in [1,\fz]$, $s\in (0,1]$,
and any BQBF space
 $X(\cx)$,
the \emph{finite atomic Hardy space}
$H^{X,q,s}_{\mathrm{fin}}(\cx)$ associated
with $X(\cx)$ is defined to be
the set of all the functions $f\in L^q(\cx)$
satisfying that there exist an $n\in\nn$, a sequence
$\{\lz_j\}^n_{j=1}\subset(0,\fz)$, and a sequence
$\{a_j\}^n_{j=1}$ of $(X(\cx),q)$-atoms
supported, respectively, in balls
$\{B_j\}^n_{j=1}$ such that $f=\sum^n_{j=1}\lz_ja_j$.
Moreover, let
\begin{align*}
\|f\|_{H^{X,q,s}_{\mathrm{fin}}(\cx)}
&:=\inf\lf\|\lf\{\sum^n_{j=1}\lf[\frac{\lz_j}
{\|\mathbf{1}_{B_j}\|_{X(\cx)}}
\r]^{s}\mathbf{1}_{B_j}\r\}^{1/s}
\r\|_{X(\cx)},
\end{align*}
where the infimum is taken over all
the decompositions
of $f$ as above. The following lemma is a
part of \cite[Theorem 4.23 and Lemma 6.4]{yhyy21a}.
\begin{lemma}\label{lefin}
Let $\oz$ be as in \eqref{eqoz}, $\eta$ as in \eqref{eqwa},
and $X(\cx)$ a \emph{BQBF} space.
Suppose that $X(\cx)$
satisfies Assumption \ref{asfs} with $p_-
\in (\oz/(\oz+\eta),\fz)$, and Assumption \ref{asas}
with $s_0\in(\oz/(\oz+\eta),\min\{p_-,1\})$
and $p_0\in(s_0,\fz)$.
Let $q\in(\max\{1,p_0\},\fz)$. Then
$\|\cdot\|_{H_X(\cx)}$ and
$\|\cdot\|_{H^{X,q,s_0}_{\mathrm{fin}}(\cx)}$
are equivalent on ${H^{X,q,s_0}_{\mathrm{fin}}(\cx)}$.
Furthermore, if $X(\cx)$ has an absolutely continuous quasi-norm,
then $H^{X,q,s_0}_{\mathrm{fin}}(\cx)$
is dense in $H_X(\cx)$.
\end{lemma}

\begin{proof}[Proof of Theorem \ref{thcz}]
Let $\ez$, $T$, $X(\cx)$, $p_-$, $s_0$, and $p_0$
be as in the present theorem, and
$q_0\in (p_0,\fz)\cap[2,\fz)$. Obviously,
for any $f\in H^{X,q_0,s_0}_{\mathrm{fin}}(\cx)$,
we find that
$f\in L^2(\cx)$ and hence $T(f)$ is well defined
and $T(f)\in L^2(\cx)\subset (\gs)'$ with
$\bz$, $\gz\in(0,\eta)$.
From Lemma \ref{lefin}, it follows that
$H^{X,q_0,s_0}_{\mathrm{fin}}(\cx)$
is dense in $H_X(\cx)$.
By this density, to prove the present theorem,
we only need to show that,
for any $f\in H^{X,q_0,s_0}_{\mathrm{fin}}(\cx)$,
 $\|T(f)\|_{WX(\cx)}\ls \|f\|_{H_X(\cx)}$ [resp.,
$\|T(f)\|_{WH_X(\cx)}\ls \|f\|_{H_X(\cx)}$].

Now, fix $f\in H^{X,q_0,s_0}_{\mathrm{fin}}(\cx)$.
Using Lemma \ref{lefin}, we find that
there exist an $n\in\nn$, a sequence
$\{\lz_j\}^n_{j=1}\subset (0,\fz)$, and
a sequence $\{a_j\}^n_{j=1}$ of
$(X(\cx),q_0)$-atoms
supported, respectively, in balls
$\{B_j\}^n_{j=1}$ such that $f=\sum^n_{j=1}\lz_ja_j$ and
\begin{equation}\label{tceq6}
\|f\|_{H_X(\cx)}
\sim \lf\|\lf\{\sum^n_{j=1}\lf[\frac{\lz_j}
{\|\mathbf{1}_{B_j}\|_{X(\cx)}}
\r]^{s_0}\mathbf{1}_{B_j}\r\}^{1/s_0}
\r\|_{X(\cx)}
\end{equation}
with the positive equivalence constants
independent of $f$.
For any $j\in\{1,\ldots,n\}$, let $x_j$ and $r_j$
denote, respectively, the center and the radius of $B_j$.

To prove (i), by
\eqref{tceq6}, it suffices to show that
\begin{equation}\label{tceq8}
\|T(f)\|_{WX(\cx)}\ls
\lf\|\lf\{\sum^n_{j=1}\lf[\frac{\lz_j}
{\|\mathbf{1}_{B_j}\|_{X(\cx)}}
\r]^{s_0}\mathbf{1}_{B_j}\r\}^{1/s_0}
\r\|_{X(\cx)},
\end{equation}
where the implicit positive constant is
independent of $f$.
From  the definition of $WX(\cx)$,
the linearity of $T$, \eqref{eqsigm} with $p=1$,
and Definition \ref{debb}(ii), we deduce that
\begin{align}\label{tceq7}
\|T(f)\|_{WX(\cx)}
&=\sup_{\lz\in(0,\fz)}\lf\{
\lz\lf\|\mathbf{1}_{\{x\in\cx:\
|T(f)(x)|>\lz\}}\r\|_{X(\cx)}\r\}\\
&\le\sigma_1\lf[\sup_{\lz\in(0,\fz)}\lf\{
\lz\lf\|\mathbf{1}_{\{x\in\cx:\
\sum^n_{j=1}\lz_j|T(a_j)(x)|
\mathbf{1}_{4A_0^2B_j}(x)
>\lz/2\}}\r\|_{X(\cx)}\r\}\r.\notag\\
&\quad+\lf.\sup_{\lz\in(0,\fz)}\lf\{
\lz\lf\|\mathbf{1}_{\{x\in\cx:\
\sum^n_{j=1}\lz_j|T(a_j)(x)|
\mathbf{1}_{(4A_0^2B_j)^{\complement}}(x)
>\lz/2\}}\r\|_{X(\cx)}\r\}\r]\notag\\
&\le 2\sigma_1\lf[\lf\|\sum^n_{j=1}\lz_j\lf|T\lf(a_j\r)\r|
\mathbf{1}_{4A_0^2B_j}\r\|_{X(\cx)}
+\lf\|\sum^n_{j=1}\lz_j\lf|T(a_j)\r|
\mathbf{1}_{(4A_0^2B_j)^{\complement}}
\r\|_{WX(\cx)}\r]\notag\\
&=:2\sigma_1(\mathrm{I}_1+\mathrm{I}_2),\notag
\end{align}
where $\sigma_1\in[1,\fz)$ is as in
\eqref{eqsigm} with $p=1$.

We first estimate $\mathrm{I}_1$.
By the H\"{o}lder inequality,
the boundedness of $T$ on $L^{q_0}(\cx)$
(see, for instance, \cite[Chapter III, Theorem 4.2]{cw71}),
and Definition \ref{deat}(ii) with $a=a_j$ and $B=B_j$,
we find that, for any
$j\in \{1,\ldots,n\}$,
\begin{align*}
&\lf\|T\lf(a_j\r)\mathbf{1}_{4A_0^2B_j}
\r\|_{L^{p_0}(\cx)}\\
&\quad\le\lf[\mu\lf(4A_0^2B_j\r)\r]^{1/p_0-1/q_0}
\lf\|T\lf(a_j\r)\mathbf{1}_{4A_0^2B_j}
\r\|_{L^{q_0}(\cx)}\\
&\quad\ls \lf[\mu\lf(4A_0^2B_j\r)\r]^{1/p_0-1/q_0}
\lf\|a_j\r\|_{L^{q_0}(\cx)}
\ls\lf[\mu\lf(4A_0^2B_j\r)\r]^{1/p_0}
\lf\|\mathbf{1}_{B_j}\r\|^{-1}_{X(\cx)}.
\end{align*}
From this, Lemma \ref{lees}(iv) with $s=s_0$,
Definition \ref{debb}(ii),
Proposition \ref{pras} with $\lz_j$, $a_j$,
and $B_j$ replaced, respectively, by
$\lz_j/\|\mathbf{1}_{B_j}\|_{X(\cx)}$,
a constant multiple of
$\|\mathbf{1}_{B_j}\|_{X(\cx)}T(a_j)
\mathbf{1}_{4A_0^2B_j}$, and $4A_0^2B_j$, and
Proposition \ref{prfs}(i) with $t$,
$\tau$, and $\lz_j$
replaced, respectively, by $s_0$,
$4A_0^2$, and
$[\lz_j/\|\mathbf{1}_{B_j}\|_{X(\cx)}]^{s_0}$, it follows that
\begin{align}\label{tceq4}
\mathrm{I}_1
&\le \lf\|\sum^n_{j=1}\lf[
\lz_j\lf|T\lf(a_j\r)\r|\mathbf{1}_{4A_0^2
B_j}\r]^{s_0}\r\|^{1/s_0}_{X^{1/s_0}(\cx)}\\
&\ls \lf\|\sum^n_{j=1}\lf[\frac{\lz_j}
{\|\mathbf{1}_{B_j}\|_{X(\cx)}}\r]^{s_0}
\mathbf{1}_{4A_0^2B_j}\r\|^{1/s_0}_{X^{1/s_0}(\cx)}
\ls \lf\|\lf\{\sum^n_{j=1}\lf[\frac{\lz_j}
{\|\mathbf{1}_{B_j}\|_{X(\cx)}}\r]^{s_0}
\mathbf{1}_{B_j}\r\}^{1/s_0}\r\|_{X(\cx)},\notag
\end{align}
where all the implicit positive constants
are independent of $f$.

Now, we estimate $\mathrm{I}_2$.
By  $\int_{\cx}a_j(z)\,d\mu(z)=0$ and
$\supp(a_j)\subset B_j$ for any $j\in\{1,\ldots,n\}$,
Definition \ref{deke}(ii),
the estimate that $\int_{\cx}|a_j(z)|\,d\mu(z)
\le \mu(B_j)\|\mathbf{1}_{B_j}\|_{X(\cx)}^{-1}$
for any $j\in\{1,\ldots,n\}$,
and Lemma \ref{lees}(iii) with $B$
replaced by $B_j$,
we conclude that, for any $j\in\{1,\ldots,n\}$ and
$x\in (4A_0^2B_j)^{\complement}$,
\begin{align*}
\lf|T\lf(a_j\r)(x)\r|
&\le \int_{\cx}\lf|K(x,z)-K(x,x_j)\r|\lf|a_j(z)\r|
\,d\mu(z)\\
&\ls \int_{B_j}\lf[\frac{d(x_j,z)}
{d(x_j,x)}\r]^{\ez}\frac{1}{V(x_j,x)}
\lf|a_j(z)\r|\,d\mu(z)\\
&\ls \lf[\frac{r_j}{d(x_j,x)}\r]^{\ez}
\frac{\mu(B_j)}{V(x_j,x)}\lf\|\mathbf{1}
_{B_j}\r\|^{-1}_{X(\cx)}
\ls \lf[\cm\lf(\mathbf{1}_{B_j}\r)\r]^{
\frac{\oz+\ez}{\oz}}\lf\|\mathbf{1}
_{B_j}\r\|^{-1}_{X(\cx)}.
\end{align*}
From this, Definition \ref{debb}(ii)
with $X(\cx)$ replaced by $WX(\cx)$,
Assumption \ref{asw} with $\tau=\ez$, and
Lemma \ref{lees}(iv) with $s=s_0$,
we deduce that
\begin{align*}
\mathrm{I}_2
&\ls \lf\|\sum^n_{j=1}
\frac{\lz_j}{\|\mathbf{1}_{B_j}\|_{X(\cx)}}
\lf[\cm\lf(\mathbf{1}_{B_j}\r)\r]^{\frac
{\oz+\ez}{\oz}}\r\|_{WX(\cx)}\\
&\ls \lf\|\sum^n_{j=1}
\frac{\lz_j}{\|\mathbf{1}_{B_j}\|
_{X(\cx)}}\mathbf{1}_{B_j}\r\|_{X(\cx)}
\ls \lf\|\lf\{\sum^n_{j=1}\lf[\frac{\lz_j}
{\|\mathbf{1}_{B_j}\|_{X(\cx)}}\r]^{s_0}
\mathbf{1}_{B_j}\r\}^{1/s_0}\r\|_{X(\cx)},
\end{align*}
where all the implicit positive constants
are independent of $f$.
By this, \eqref{tceq4}, and \eqref{tceq7},
we obtain \eqref{tceq8} and hence
complete the proof of (i).

Next, we prove (ii).
Recall that $T(f)\in (\gs)'$ with
$\bz$, $\gz\in(\ez,\eta)$.
 By this and \eqref{tceq6}, to prove (ii),
 it suffices to show that
\begin{equation}\label{tceq1}
\lf\|\lf[T(f)\r]^*\r\|_{WX(\cx)}
\ls \lf\|\lf\{\sum^n_{j=1}\lf[\frac{\lz_j}
{\|\mathbf{1}_{B_j}\|_{X(\cx)}}
\r]^{s_0}\mathbf{1}_{B_j}\r\}^{1/s_0}
\r\|_{X(\cx)},
\end{equation}
where the implicit positive constant is independent
of $f$.

Indeed, from the definition of $WX(\cx)$,
the linearity of $T$,
\eqref{eqsigm} with $p=1$, and
Definition \ref{debb}(ii),
it follows that, for any $\lz\in(0,\fz)$,
\begin{align}\label{tceq3}
&\lf\|\lf[T(f)\r]^*\r\|_{WX(\cx)}\\
&\quad=\sup_{\lz\in(0,\fz)}\lf\{
\lz\lf\|\mathbf{1}_{\{x\in\cx:\
[T(f)]^*(x)>\lz\}}\r\|_{X(\cx)}\r\}\notag\\
&\quad\le\sigma_1\lf[\sup_{\lz\in(0,\fz)}\lf\{
\lz\lf\|\mathbf{1}_{\{x\in\cx:\
\sum^n_{j=1}\lz_j[T(a_j)]^*(x)
\mathbf{1}_{4A_0^2B_j}(x)
>\lz/2\}}\r\|_{X(\cx)}\r\}\r.\notag\\
&\quad\quad\lf.+\sup_{\lz\in(0,\fz)}\lf\{
\lz\lf\|\mathbf{1}_{\{x\in\cx:\
\sum^n_{j=1}\lz_j[T(a_j)]^*(x)
\mathbf{1}_{(4A_0^2B_j)
^{\complement}}(x)
>\lz/2\}}\r\|_{X(\cx)}\r\}\r]\notag\\
&\quad\le 2\sigma_1\lf\{\lf\|
\sum^n_{j=1}\lz_j\lf[T\lf(a_j\r)\r]^*
\mathbf{1}_{4A_0^2B_j}
\r\|_{X(\cx)}\notag
+\lf\|\sum^n_{j=1}\lz_j\lf[T(a_j)\r]^*
\mathbf{1}_{(4A_0^2B_j)^\complement}
\r\|_{WX(\cx)}\r\}\notag\\
&\quad=:2\sigma_1(\mathrm{J}_1+\mathrm{J}_2),\notag
\end{align}
where $\sigma_1\in[1,\fz)$ is as in
\eqref{eqsigm} with $p=1$.

We first estimate $\mathrm{J}_1$.
From the H\"{o}lder inequality,
the estimate that
$[T(a_j)]^*\ls\cm(T(a_j))$
for any $j\in \{1,\ldots,n\}$
(see, for instance,
\cite[Proposition 3.9]{gly08}),
the boundedness of both $\cm$ and $T$
on $L^{q_0}(\cx)$, and Definition \ref{deat}(ii)
with $a=a_j$ and $B=B_j$,
we deduce that, for any
$j\in \{1,\ldots,n\}$,
\begin{align*}
&\lf\|\lf[T\lf(a_j\r)\r]^*
\mathbf{1}_{4A_0^2B_j}
\r\|_{L^{p_0}(\cx)}\\
&\quad\le \lf[\mu\lf(4A_0^2B_j\r)\r]^{1/p_0-1/q_0}
\lf\|\lf[T\lf(a_j\r)\r]^*
\mathbf{1}_{4A_0^2B_j}\r\|_{L^{q_0}(\cx)}\\
&\quad\ls  \lf[\mu\lf(4A_0^2B_j\r)\r]^{1/p_0-1/q_0}
\lf\|\cm\lf(T\lf(a_j\r)\r)\r\|_{L^{q_0}(\cx)}\\
&\quad\ls \lf[\mu\lf(4A_0^2B_j\r)\r]^{1/p_0-1/q_0}
\lf\|a_j\r\|_{L^{q_0}(\cx)}
\ls \lf[\mu\lf(4A_0^2B_j\r)\r]^{1/p_0}
\lf\|\mathbf{1}_{B_j}\r\|_{X(\cx)}^{-1}.
\end{align*}
By this, Lemma \ref{lees}(iv) with $s=s_0$,
Definition \ref{debb}(ii),
Proposition \ref{pras} with $\lz_j$, $a_j$, and
$B_j$ replaced, respectively, by
$\lz_j/\|\mathbf{1}_{B_j}\|_{X(\cx)}$,
a constant multiple of
$\|\mathbf{1}_{B_j}\|_{X(\cx)}[T(a_j)]^*
\mathbf{1}_{4A_0^2B_j}$,
and $4A_0^2B_j$, and Proposition \ref{prfs}(i)
with $t$, $\tau$, and $\lz_j$ replaced,
respectively, by $s_0$, $4A_0^2$, and $[\lz_j/
\|\mathbf{1}_{B_j}\|_{X(\cx)}]^{s_0}$, we have
\begin{align}\label{tceq2}
\mathrm{J}_1
&\le \lf\|\sum^n_{j=1}\lf\{
\lz_j\lf[T\lf(a_j\r)\r]^*\mathbf{1}_{4A_0^2
B_j}\r\}^{s_0}\r\|^{1/s_0}_{X^{1/s_0}(\cx)}\\
&\ls \lf\|\sum^n_{j=1}\lf[\frac{\lz_j}
{\|\mathbf{1}_{B_j}\|_{X(\cx)}}\r]^{s_0}
\mathbf{1}_{4A_0^2B_j}\r\|^{1/s_0}_{X^{1/s_0}(\cx)}
\ls \lf\|\lf\{\sum^n_{j=1}\lf[\frac{\lz_j}
{\|\mathbf{1}_{B_j}\|_{X(\cx)}}\r]^{s_0}
\mathbf{1}_{B_j}\r\}^{1/s_0}\r\|_{X(\cx)},\notag
\end{align}
where all the implicit positive constants
are independent of $f$.

Now, we estimate $\mathrm{J}_2$. By
$T^*1=0$ and some arguments similar to those
used in \cite[pp.\,74--75]{zhy20},
we conclude that,
for any $j\in\{1,\ldots,n\}$ and
$x\in (4A_0^2B_j)^{\complement}$,
\begin{equation*}
\lf[T\lf(a_j\r)\r]^*(x)\ls \lf\|\mathbf{1}_{B_j}
\r\|_{X(\cx)}^{-1}\frac{\mu(B_j)}{V(x_j,x)}
\lf[\frac{r_j}{d(x_j,x)}\r]^\ez,
\end{equation*}
which, combined with Lemma \ref{lees}(iii)
with $B$ replaced by $B_j$,
implies that, for any $j\in\{1,\ldots,n\}$ and
$x\in (4A_0^2B_j)^{\complement}$,
\begin{equation*}
\lf[T\lf(a_j\r)\r]^*(x)\ls \lf\|\mathbf{1}_{B_j}
\r\|_{X(\cx)}^{-1}\lf[\cm\lf(\mathbf{1}
_{B_j}\r)(x)\r]^{\frac{\oz+\ez}{\oz}}.
\end{equation*}
Using this, Definition \ref{debb}(ii) with
$X(\cx)$ replaced by $WX(\cx)$,
Assumption \ref{asw} with $\tau=\ez$,
and Lemma \ref{lees}(iv) with $s=s_0$,
we obtain
\begin{align*}
\mathrm{J}_2
&\ls \lf\|\sum^n_{j=1}
\frac{\lz_j}{\|\mathbf{1}_{B_j}\|
_{X(\cx)}}\lf[\cm\lf(\mathbf{1}_{B_j}\r)
\r]^{\frac{\oz+\ez}{\oz}}\r\|_{WX(\cx)}\\
&\ls \lf\|\sum^n_{j=1}
\frac{\lz_j}{\|\mathbf{1}_{B_j}\|
_{X(\cx)}}\mathbf
{1}_{B_j}\r\|_{X(\cx)}
\ls \lf\|\lf\{\sum^n_{j=1}
\lf[\frac{\lz_j}{\|\mathbf{1}_{B_j}\|
_{X(\cx)}}\r]^{s_0}\mathbf
{1}_{B_j}\r\}^{1/s_0}\r\|_{X(\cx)},
\end{align*}
where all the implicit positive constants
are independent of $f$.
By this, \eqref{tceq3}, and \eqref{tceq2},
we then complete the proof of \eqref{tceq1},
and hence of Theorem \ref{thcz}.
\end{proof}

The concept of $A_1(\cx)$-weights is given
in Subsection \ref{sswe} below.
If a BQBF space $X(\cx)$
does not have an absolutely continuous quasi-norm
but embeds into a weighted Lebesgue space
with an $A_1(\cx)$-weight,
we obtain the following theorem on the boundedness of
Calder\'{o}n--Zygmund operators.

\begin{theorem}\label{thc}
Let $\oz$ be as in \eqref{eqoz},
$\eta$ as in \eqref{eqwa}, $\ez\in(0,\eta)$,
$T$ an $\ez$-Calder\'on--Zygmund operator, and
$X(\cx)$ a \emph{BQBF} space.
Suppose that $X(\cx)$
satisfies Assumption \ref{asfs} with
$p_-=\oz/(\oz+\ez)$,
Assumption \ref{asw} with $\tau=\ez$,
and the following assumption that
there exist an $s_-\in(0,p_-)$ and a $p_1\in(p_-,\fz)$ such that,
for any $s\in(s_-,p_-)$,
$X^{1/s}(\cx)$ is a \emph{BBF} space,
the Hardy--Littlewood maximal operator $\cm$
is bounded on the $\frac{1}{(p_1/s)'}$-convexification of
the associate space $(X^{1/s})'(\cx)$ with
its \emph{operator norm} denoted by $\|\cm\|_{(s)}$
for simplicity, and
\begin{equation}\label{eqm1}
\limsup_{s\in(s_-,p_-),\ s\to p_-}
\|\cm\|_{(s)}
+\limsup_{s\in(s_-,p_-),\ s\to p_-}
\lf\|\mathbf{1}_{B(x_0,1)}\r\|_{(X^{1/s})'(\cx)}<\fz,
\end{equation}
where $x_0\in\cx$ is the fixed basepoint.
\begin{itemize}
\item[\textup{(i)}]
Then $T$ can be extended to a bounded
linear operator from $H_X(\cx)$
to $WX(\cx)$.
\item[\textup{(ii)}]
If $T^*1=0$,
then $T$ can be extended to a bounded
linear operator from $H_X(\cx)$
to $WH_X(\cx)$.
\end{itemize}
\end{theorem}

\begin{proof}
Let $\ez$, $T$, $X(\cx)$, $p_-$, $s_-$,
$p_1$, and $\|\cm\|_{(s)}$
 be as in the present theorem,
$s_1\in(\max\{s_-,
\oz/(\oz+\eta)\},p_-)$, and
$q_1\in (p_1,\fz)\cap[2,\fz)$.
For any $f\in H_X(\cx)$,
$f$ might not be in $L^2(\cx)$
and hence $T(f)$ might not be well defined.
To extend $T$ to $H_X(\cx)$,
we embed $H_X(\cx)$ into a
weighted Hardy space.

Let $w:=[\cm(\mathbf{1}_{B(x_0,1)})]^{1/(p_1/s_-)'}$.
By  \cite[Theorem 2.17]{yhyy21a},
we find that $w$ is an $A_1(\cx)$-weight.
Then we claim that $X(\cx)$ continuously embeds into
$L^{p_-}_w(\cx)$, where $L^{p_-}_w(\cx)$ is defined
as in \eqref{eqweight} with $p$ replaced by $p_-$.
By the estimate that
$\cm(\mathbf{1}_{B(x_0,1)})\le 1$,
we find that, for any $s\in(s_-,p_-)$,
\begin{align*}
\lf[\cm\lf(\mathbf{1}_{B(x_0,1)}\r)\r]^{\frac{1}{(p_1/s_-)'}}
\le \lf[\cm\lf(\mathbf{1}_{B(x_0,1)}\r)\r]^{\frac{1}{(p_1/s)'}}.
\end{align*}
From this, the definition of $L^{p_-}_w(\cx)$,
the H\"{o}lder inequality on $X^{1/s}(\cx)$
(see, for instance, \cite[Lemma 2.18]{yhyy21a}), and
the boundedness of $\cm$ on
the $\frac{1}{(p_1/s)'}$-convexification of
$(X^{1/s})'(\cx)$,
we deduce that, for any $s\in (s_-,p_-)$
and $g\in X(\cx)$,
\begin{align*}
\|g\|^s_{L^s_w(\cx)}
&=\int_{\cx}|g(z)|^s\lf[\cm\lf(\mathbf{1}_{
B(x_0,1)}\r)(z)\r]^{\frac{1}{(p_1/s_-)'}}\,d\mu(z)\\
&\le \lf\||g|^s\r\|_{X^{1/s}(\cx)}
\lf\|\lf[\cm\lf(\mathbf{1}_{B(x_0,1)}\r)\r]^{
\frac{1}{(p_1/s_-)'}}\r\|_{
(X^{1/s})'(\cx)}\\
&\le \|g\|^s_{X(\cx)}
\lf\|\lf[\cm\lf(\mathbf{1}_{B(x_0,1)}\r)\r]^{
\frac{1}{(p_1/s)'}}\r\|_{(X^{1/s})'(\cx)}\\
&\le \lf[\|\cm\|_{(s)}\r]^{\frac{1}{(p_1/s)'}}
\lf\|\mathbf{1}_{B(x_0,1)}\r\|_{(X^{1/s})'(\cx)}
\|g\|^s_{X(\cx)}.
\end{align*}
Letting $s\to p_-$ and using the Fatou lemma
combined with \eqref{eqm1}, we obtain,
for any $g\in X(\cx)$,
$$
\|g\|_{L^{p_-}_w(\cx)} \ls \|g\|_{X(\cx)},
$$
which completes the proof of the above claim.
Let $WL^{p_-}_w(\cx)$, $H^{p_-}_w(\cx)$, and $WH^{p_-}_w(\cx)$
be, respectively, $WX(\cx)$, $H_X(\cx)$, and $WH_X(\cx)$
with $X(\cx)$ replaced by $L^{p_-}_w(\cx)$.

Next, we fix an $f\in H_X(\cx)$.
By the embedding $X(\cx)\subset L^{p_-}_w(\cx)$,
we conclude that
\begin{align}\label{t4eq3}
\|f\|_{H^{p_-}_w(\cx)}
=\lf\|f^*\r\|_{L^{p_-}_w(\cx)}
\ls\lf\|f^*\r\|_{X(\cx)}
\sim \|f\|_{H_X(\cx)}<\fz
\end{align}
and $f\in H^{p_-}_w(\cx)$.
On the other hand, by Lemma \ref{leha},
we find that there exist a sequence
$\{\lz_j\}_{j\in\nn}\subset (0,\fz)$ and
a sequence $\{a_j\}_{j\in\nn}$ of
$(X(\cx),q_1)$-atoms
supported, respectively, in balls $\{B_j\}_{j\in\nn}$
such that $f=\sum_{j\in\nn}\lz_ja_j$ in
$(\gs)'$ with $\bz$, $\gz\in(\oz[1/s_1-1],\eta)$, and
\begin{equation}\label{t4eq1}
\lf\|\lf\{\sum_{j\in\nn}\lf[\frac{\lz_j}
{\|\mathbf{1}_{B_j}\|_{X(\cx)}}\r]^{s_1}
\mathbf{1}_{B_j}\r\}^{1/s_1}\r\|_{X(\cx)}
\sim \|f\|_{H_X(\cx)}
\end{equation}
with the positive equivalence constants
independent of $f$.
Then we show that
\begin{equation}\label{t4eq2}
f=\sum_{j\in\nn}\lz_ja_j
\end{equation}
in $H^{p_-}_w(\cx)$.
Observe that
$\|\mathbf{1}_{B_j}\|^{-1}_{L^{p_-}_w(\cx)}
\|\mathbf{1}_{B_j}\|_{X(\cx)}a_j$
is an $(L^{p_-}_w(\cx),q_1)$-atom
for any $j\in\nn$.
Using this, Lemma \ref{leha} with
$H_X(\cx)$ replaced by $H^{p_-}_w(\cx)$,
and the embedding $X(\cx)\subset L^{p_-}_w(\cx)$
combined with \eqref{t4eq1},
we find that, for any $n\in\nn$,
\begin{align*}
&\lf\|f-\sum^n_{j=1}\lz_ja_j
\r\|_{H^{p_-}_w(\cx)}\\
&\quad=\lf\|\sum^\fz_{j=n+1}\frac{\lz_j
\|\mathbf{1}_{B_j}\|_{L^{p_-}_w(\cx)}}
{\|\mathbf{1}_{B_j}\|_{X(\cx)}}
\lf[\lf\|\mathbf{1}_{B_j}\r\|^{-1}_{L^{p_-}_w(\cx)}
\lf\|\mathbf{1}_{B_j}\r\|_{X(\cx)}a_j\r]
\r\|_{H^{p_-}_w(\cx)}\\
&\quad\ls\lf\|\lf\{\sum^\fz_{j=n+1}\lf[\frac{\lz_j}
{\|\mathbf{1}_{B_j}\|_{X(\cx)}}\r]^{s_1}
\mathbf{1}_{B_j}\r\}^{1/s_1}\r\|_{L^{p_-}_w(\cx)}\\
&\quad\ls\lf\|\lf\{\sum^\fz_{j=n+1}\lf[\frac{\lz_j}
{\|\mathbf{1}_{B_j}\|_{X(\cx)}}\r]^{s_1}
\mathbf{1}_{B_j}\r\}^{1/s_1}\r\|_{X(\cx)}<\fz.
\end{align*}
Letting $n\to\fz$ and using the dominated convergence
theorem, we then complete the proof of
the above claim \eqref{t4eq2}.

Now, we prove (i). By Theorem \ref{thwe}(iii) below,
we find that $T$ can be extended to a bounded linear
operator from $H^{p_-}_w(\cx)$ to $WL^{p_-}_w(\cx)$.
Using this, \eqref{t4eq3}, and \eqref{t4eq2},
we conclude that $T(f)\in WL^{p_-}_w(\cx)$ and
$T(f)=\sum_{j\in\nn}\lz_jT(a_j)$
in $WL^{p_-}_w(\cx)$. From this and
\cite[Proposition 1.1.9(ii) and Theorem 1.1.11]{g14},
we deduce that there exists a subsequence
$\{j_l\}_{l\in\nn}\subset \nn$ such that
$j_l\to\fz$ as $l\to\fz$, and, for
$\mu$-almost every
point $x\in\cx$,
$$
T(f)(x)=\lim_{l\to\fz}\sum^{j_l}_{l=1}
\lz_lT(a_l)(x),
$$
 which further implies that
$|T(f)|\le \sum_{j\in\nn}\lz_j|T(a_j)|$
$\mu$-almost everywhere.
By this and some arguments similar to those
used in the estimation
of \eqref{tceq8}, we obtain
\begin{equation*}
\|T(f)\|_{WX(\cx)}
\le \lf\|\sum_{j\in\nn}\lz_j\lf|T(a_j)\r|\r\|_{WX(\cx)}
\ls\lf\|\lf\{\sum_{j\in\nn}\lf[\frac{\lz_j}
{\|\mathbf{1}_{B_j}\|_{X(\cx)}}
\r]^{s_1}\mathbf{1}_{B_j}\r\}^{1/s_1}
\r\|_{X(\cx)}
\end{equation*}
with the implicit positive constant independent
of $f$. Using this and \eqref{t4eq1},
we then complete the proof of (i).

To prove (ii), by Theorem
\ref{thwe}(iii) below, we find that
$T$ can be extended to a bounded linear
operator from $H^{p_-}_w(\cx)$ to $WH^{p_-}_w(\cx)$.
Using this,
\eqref{t4eq3}, and \eqref{t4eq2},
we find that $T(f)\in WH^{p_-}_w(\cx)$ and
$T(f)=\sum_{j\in\nn}\lz_jT(a_j)$
in $WH^{p_-}_w(\cx)$. From this and Remark \ref{rede}(i)
with $WH_X(\cx)$ replaced by $WH^{p_-}_w(\cx)$,
it follows that $T(f)=\sum_{j\in\nn}\lz_jT(a_j)$ in $(\gs)'$
with $\bz$, $\gz\in(\ez,\eta)$,
which further implies that
$$[T(f)]^*\le \sum_{j\in\nn}\lz_j\lf[T(a_j)\r]^*.$$
By this and some arguments similar to those
used in the estimation of
\eqref{tceq1}, we conclude that
\begin{align*}
\|T(f)\|_{WH_X(\cx)}
&=\lf\|[T(f)]^*\r\|_{WX(\cx)}
\le \lf\|\sum_{j\in\nn}\lz_j\lf[T(a_j)\r]^*\r\|_{WX(\cx)}\\
&\ls\lf\|\lf\{\sum_{j\in\nn}\lf[\frac{\lz_j}
{\|\mathbf{1}_{B_j}\|_{X(\cx)}}
\r]^{s_1}\mathbf{1}_{B_j}\r\}^{1/s_1}
\r\|_{X(\cx)},
\end{align*}
where all the implicit positive constants are
independent of $f$. Using this and \eqref{t4eq1},
we then complete the proof of (ii),
and hence of Theorem \ref{thc}.
\end{proof}

\begin{remark}
\begin{itemize}\label{recz}
\item[(i)]
Let $\ez\in (0,\eta)$.
By the fact that $L^{\oz/(\oz+\ez)}(\cx)$ has
an absolutely continuous quasi-norm and Remarks \ref{reas}
and \ref{rew}, we find that  $L^{\oz/(\oz+\ez)}(\cx)$
satisfies all the assumptions needed in Theorem \ref{thcz}.
Moreover, see Section \ref{sap} for more
function spaces satisfying all the assumptions
needed in Theorem \ref{thcz}.
\item[(ii)]
In the above proof of Theorem \ref{thc},
we borrow some ideas from the proof
of \cite[Theorem 6.3]{zyyw20}.
Observe that Theorem \ref{thc} can be applied to the
BQBF space
without an absolutely continuous quasi-norm, and now
we give such an example.
Recall that, for any $q\in(0,\fz)$ and $p\in[q,\fz)$,
the \emph{Morrey space $M^p_q(\rn)$}
is defined to be the set of all the measurable
functions $f$ on $\rn$ such that
$$
\|f\|_{M^p_q(\rn)}:=
\sup_B\lf\{|B|^{1/p-1/q}\|f\|_{L^q(B)}\r\}<\fz,
$$
where the supremum is taken over all balls $B$ of $\rn$,
and $|B|$ denotes the Lebesgue measure of $B$.
Observe that the Morrey space may not have an
absolutely continuous quasi-norm
(see, for instance, \cite[Remark 3.4]{wyy20}).
Checking \cite[Remark 2.7(e)]{wyy20}
and its reference,
we find that, in Theorem \ref{thc},
if $X(\cx):=M^p_q(\rn)$
with $q=n/(n+\ez)$ and $p\in[q,\fz)$,
then all the assumptions needed
in Theorem \ref{thc} and hence Theorem \ref{thc} itself hold true.

\item[(iii)]
Using the results from \cite[Section 8]{yhyy21b},
we find that, in Theorems \ref{thcz} and \ref{thc},
if $p_-\in(\oz/(\oz+\ez),\fz)$, then
$T$ can be extended to a bounded linear operator from
$H_X(\cx)$ to $X(\cx)$, or to $H_X(\cx)$.
By this and Proposition \ref{prwb},
we find that Theorems \ref{thcz} and \ref{thc}
still hold true if $p_-\in (\oz/(\oz+\ez),\fz)$.
However, if $p_-=\oz/(\oz+\ez)$ as in
Theorems \ref{thcz} and \ref{thc},
we only obtain the conclusions of Theorems \ref{thcz}
and \ref{thc}, namely, $T$ may not be bounded
from $H_X(\cx)$ to $X(\cx)$, or
to $H_X(\cx)$ in this case.
In this sense, $p_-=\oz/(\oz+\ez)$ is
called the \emph{critical case}
or the \emph{endpoint case}.
\end{itemize}
\end{remark}

\section{Applications\label{sap}}

In this section, we apply the main results
obtained in the above sections to
some specific BQBF
spaces, such as Lebesgue spaces, weighted Lebesgue spaces,
Orlicz spaces, and variable Lebesgue spaces, which
shows the wide generality of the results
obtained in the above sections.
We discuss these spaces in the
following subsections in turn.

\subsection{Lebesgue Spaces}

Let $p\in(0,\fz)$.
Observe that the Lebesgue space $L^p(\cx)$
is a quasi-Banach function space and
hence a BQBF space; see, for instance,
\cite[Remark 2.7(i)]{yhyy21a}.
Then we can apply the theorems obtained in
the previous sections to Lebesgue spaces.

\begin{theorem}\label{thle}
\begin{enumerate}
\item[\textup{(i)}]
Let $p\in(1,\fz)$. Then the conclusion of Theorem \ref{thmw} holds true
with $X(\cx)$ replaced by $L^p(\cx)$.
\item[\textup{(ii)}]
Let $p\in(\oz/(\oz+\eta),\fz)$ with $\oz$ as in
\eqref{eqoz} and $\eta$ as in \eqref{eqwa}.
Then the conclusions of Theorems \ref{thm}, \ref{tham}, \ref{thma},
and \ref{thhin} hold true with $X(\cx)$ replaced by $L^p(\cx)$.
\item[\textup{(iii)}]
Let $\ez\in (0,\eta)$ and $p=\oz/(\oz+\ez)$.
Then the conclusion of Theorem \ref{thcz}
holds true with $X(\cx)$ replaced
by $L^p(\cx)$.
\end{enumerate}
\end{theorem}

\begin{proof}
By \cite[Lemma 3.5]{zhy20}, we find that
all the assumptions needed in Theorem \ref{thmw} with $X(\cx)$ replaced
by $L^p(\cx)$ with $p\in(1,\fz)$ hold true, which further
implies (i).

From Remark \ref{reas} and \cite[Lemma 3.5]{zhy20}, we infer that,
for any $p\in(\oz/(\oz+\eta),\fz)$,
$L^p(\cx)$ satisfies all the assumptions
needed in Theorems \ref{thm}, \ref{tham}, \ref{thma},
and \ref{thhin}, which completes the proof of (ii).

Using Remark \ref{recz}(i),
we conclude that the Lebesgue space
$L^{\oz/(\oz+\ez)}(\cx)$ with $\ez$ as in (iii)
satisfies all the assumptions needed in Theorem \ref{thcz}.
This finishes the proof of (iii), and hence of Theorem \ref{thle}.
\end{proof}

\begin{remark}
Wu et al. \cite{ww12} and Ding et al. \cite{dw10}
studied weak Hardy spaces on RD-spaces.
Recently, Zhou et al. \cite{zhy20}
investigated Hardy--Lorentz spaces
on $\cx$. The results of Theorem \ref{thle},
except the atomic reconstruction,
are included in \cite[Theorems 3.6 and 4.6,
Lemma 8.3, Corollary 8.5, and Theorems 8.12 and 8.13]{zhy20}.
Comparing the atomic reconstruction of weak Hardy spaces
obtained in \cite[Theorem 4.4]{zhy20} with
that in Theorem \ref{thle}(ii),
we find that, in \cite[Theorem 4.4]{zhy20},
$p\in(\oz/(\oz+\eta),1]$ and
$q\in(1,\fz]$, while,
in Theorem \ref{thle}(ii),
$p\in(\oz/(\oz+\eta),\fz)$ and
$q\in(p,\fz]\cap[1,\fz]$.
\end{remark}
\subsection{Weighted Lebesgue Spaces}\label{sswe}
To start this subsection,
we recall the concept of weights on $\cx$;
see, for instance, \cite[Chapter I]{st89}.
Let the symbol $\mathbb{B}$ denote the
set of all balls of $\cx$.
A locally integrable function $w:\cx\to[0,\fz)$
is called an \emph{$A_p(\cx)$-weight}
with $p\in[1,\fz)$ if
\begin{equation}\label{eqawei}
[w]_{A_p(\cx)}:=
\sup_{B\in\mathbb{B}}
\lf\{\lf[\mu(B)\r]^{-p}\|w\|_{L^1(B)}
\lf\|w^{-1}\r\|_{L^{1/(p-1)}(B)}\r\}<\fz,
\end{equation}
where $1/(p-1):=\fz$ when $p=1$.
Let $A_\fz(\cx):=\cup_{
p\in[1,\fz)}A_p(\cx)$. For any
$w\in A_\fz(\cx)$, the \emph{critical index $q_w$}
of $w$ is defined by setting
\begin{equation}\label{eqqw}
q_w:=\inf\lf\{p\in[1,\fz):\
w\in A_p(\cx)\r\}.
\end{equation}

Let $p\in (0,\fz)$ and $w\in A_\fz(\cx)$.
The \emph{weighted Lebesgue space} $L^p_w(\cx)$
is defined as in \eqref{eqweight}.
Observe that $L^p_w(\cx)$
is a BQBF space (see, for instance, \cite[Remark 2.7(iii)]{yhyy21a}),
but not necessarily a quasi-Banach function space
(see, for instance, \cite[Remark 5.22(ii)]{wyyz21} on the Euclidean space case).
If $X(\cx)=L^p_w(\cx)$,
then $WX(\cx)$ as in Definition \ref{dewb}
becomes the \emph{weighted weak
Lebesgue space $WL^p_w(\cx)$}.
The following theorem is an application of
the results obtained in the
above sections to weighted
Lebesgue spaces.

\begin{theorem}\label{thwe}
\begin{enumerate}
\item[\textup{(i)}]
Let $w\in A_{\fz}(\cx)$ and $p\in(q_w,\fz)$
with $q_w$ as in \eqref{eqqw}.
Then the conclusion of Theorem \ref{thmw} holds true
with $X(\cx)$ replaced by $L^p_w(\cx)$.
\item[\textup{(ii)}]
Let $w\in A_\fz(\cx)$ and $p\in(q_w\oz/(\oz+\eta),\fz)$
with $q_w$ as in \eqref{eqqw},
$\oz$ as in \eqref{eqoz}, and $\eta$ as in \eqref{eqwa}.
Then the conclusions of Theorems
\ref{thm}, \ref{tham}, \ref{thma},
and \ref{thhin} hold true with
$X(\cx)$ replaced by $L^p_w(\cx)$.
\item[\textup{(iii)}]
Let $\ez\in (0,\eta)$, $w\in A_1(\cx)$,
and $p=\oz/(\oz+\ez)$.
Then the conclusion of Theorem
\ref{thcz} holds true with $X(\cx)$
replaced by $L^p_w(\cx)$.
\end{enumerate}
\end{theorem}

\begin{proof}
To prove the present theorem,  we only need to show that
$L^p_w(\cx)$ satisfies all the assumptions needed,
respectively, in
Theorems \ref{thmw}, \ref{thm}, \ref{tham},
\ref{thma}, \ref{thhin}, and \ref{thcz}, which then further
implies the desired conclusions of the present theorem.

Let $w\in A_\fz(\cx)$ and $p\in(q_w\oz/(\oz+\eta),\fz)$.
By \cite[Theorem 6.5]{fmy20} (see
\cite[Theorem 3.1]{aj80} for a detailed proof
of the Euclidean space case),
we find that Assumption \ref{asfs}
holds true when $X(\cx)=L^p_w(\cx)$
and $p_-=p/q_w$.

Let $w\in A_\fz(\cx)$ and $p\in(q_w,\fz)$.
From Remark \ref{rem}
with $X(\cx)$ replaced by $L^p_w(\cx)$,
we deduce that there exists a $t\in(1,\fz)$
such that $\cm$ is bounded on $WL^{p/t}_w(\cx)$.

Let $w\in A_\fz(\cx)$, $p\in(q_w\oz/(\oz+\eta),\fz)$,
and $s_0\in (\oz/(\oz+\eta),\min\{1,p/q_w\})$.
It is easy to show that $L^{p/s_0}_w(\cx)$ is
a BBF space;
see, for instance, \cite[Remark 4.21(iii)]{yhyy21a}.
Observe that
the associate space of $L^{p/s_0}_w(\cx)$
is $L^{(p/s_0)'}_{w^{1-(p/s_0)'}}(\cx)$.
Using $w\in A_{p/s_0}(\cx)$ and \eqref{eqawei},
we conclude that $w^{1-(p/s_0)'}\in A_{(p/s_0)'}(\cx)$.
By this and \cite[p.\,5, Lemma 8]{st89},
we find that there exists some
$p_0\in (s_0,\fz)$ such that
$w^{1-(p/s_0)'}\in A_{(p/s_0)'/(p_0/s_0)'}(\cx)$,
which further implies that $\cm$ is bounded on the
$\frac{1}{(p_0/s_0)'}$-convexification
of $L^{(p/s_0)'}_{w^{1-(p/s_0)'}}(\cx)$.
Thus, Assumption \ref{asas} holds true
with $X(\cx)$ replaced by $L^p_w(\cx)$.

Let $\ez\in(0,\eta)$, $w\in A_1(\cx)$, and
$p=\oz/(\oz+\ez)$. From the dominated
convergence theorem, it follows that $L^p_w(\cx)$ has an
absolutely continuous quasi-norm. Using \cite[Theorem 6.5]{fmy20},
we find that Assumption \ref{asw} holds true when
$X(\cx)=L^p_w(\cx)$  and $\tau=\ez$.

This finishes
the proof of Theorem \ref{thwe}.
\end{proof}

\begin{remark}
The atomic characterization of weighted weak Hardy spaces
$WH^p_w(\cx)$ was obtained in
\cite[Theorem 1.6]{wj12} if $d$ satisfies the following
additional regular assumption  that there exist constants
$\widetilde{\eta}\in(0,1]$ and $C\in(0,\fz)$
such that, for any $x$, $y$, $z\in\cx$,
\begin{equation}\label{eqd}
\lf|d(x,y)-d\lf(z,y\r)\r|
\le C[d(x,z)]^{\widetilde{\eta}}
[d(x,y)+d(z,y)]^{1-\widetilde{\eta}}.
\end{equation}
However, we obtain the atomic decomposition
without \eqref{eqd} in Theorem \ref{thwe}(ii).
To the best of our knowledge, we find that
Theorem \ref{thwe} is new even when
$\cx$ is an RD-space.
\end{remark}

\subsection{Orlicz Spaces}

We first recall the concepts of both
Orlicz functions and Orlicz spaces;
see, for instance, \cite{rr91}.
A function $\Phi:[0,\fz)\to[0,\fz)$
is called an \emph{Orlicz function}
if $\Phi$ is non-decreasing,
$\Phi(0)=0$,
$\Phi(t)>0$ with $t\in (0,\fz)$,
and $\lim_{t\to\fz}\Phi(t)=\fz$.
For any $p\in (-\fz,\fz)$, an
Orlicz function $\Phi$ is said to
be of \emph{lower} (resp., \emph{upper})
\emph{type} $p$ if there exists a
positive constant $C_{(p)}$, depending on
$p$, such that, for any $t\in (0,\fz)$
and $s\in (0,1)$ (resp., $s\in [1,\fz)$),
\begin{equation*}
\Phi(st)\le C_{(p)}s^p\Phi(t).
\end{equation*}
In what follows, for any given $s\in (0,\fz)$,
let $\Phi_s(t):=\Phi(t^s)$ for any $t\in(0,\fz)$.
Now, we introduce the concept
of the Orlicz space on $\cx$.
\begin{definition}\label{deos}
Let $\Phi$ be an Orlicz function of both lower
type $p^-_\Phi\in (0,\fz)$ and upper type
$p^+_\Phi\in[p^-_\Phi,\fz)$. The \emph{Orlicz space}
$L^\Phi(\cx)$ is defined to be the set of
all the functions
$f\in\mathscr{M}(\cx)$ such that
\begin{equation*}
\|f\|_{L^\Phi(\cx)}:=
\inf\lf\{\lz\in(0,\fz):\
\int_{\cx}\Phi\lf(\frac{|f(x)|}{\lz}\r)
\,d\mu(x)\le 1\r\}<\fz.
\end{equation*}
\end{definition}

\begin{remark}\label{reor}
Let all the symbols be as in Definition \ref{deos}.
\begin{enumerate}
\item[\textup{(i)}]
By \cite[Lemma 2.5]{zyyw19}, without loss of generality,
we may always assume
that $\Phi$ is continuous and strictly increasing.
\item[\textup{(ii)}]
Let $s\in (0,\fz)$. It is easy to show that $\Phi_s$ is
of lower type $sp^-_\Phi$ and upper type $sp^+_\Phi$.
 Furthermore, observe that the $s$-convexification
$(L^\Phi)^s(\cx)$ of $L^\Phi(\cx)$ is $L^{\Phi_s}(\cx)$.
\item[\textup{(iii)}]
Let $\lz$, $A\in(0,\fz)$. If
$\int_{\cx}\Phi(|f(z)|/\lz)\,d\mu(z)\le A$,
then there exists a positive constant $C$,
depending only on $A$ and $p^-_\Phi$, such that
$\|f\|_{L^\Phi(\cx)}\le C\lz$ (see \cite[Lemma 2.7]{zyyw19}
for the Euclidean space case); here we omit its proof.
\end{enumerate}
\end{remark}

Let $\Phi$ be as in Definition \ref{deos}.
Observe that, for any $\mu$-measurable set $E\subset \cx$
with $\mu(E)<\fz$,
$$
\int_{\cx}\Phi\lf(\mathbf{1}_E(z)\r)\,d\mu(z)
=\mu(E)\Phi(1)<\fz,
$$
which implies that $\mathbf{1}_E\in L^\Phi(\cx)$.
It is easy to show that $L^\Phi(\cx)$ also
satisfies (i) through (iii) of Definition \ref{debb}.
From the above results, we
infer that $L^\Phi(\cx)$ is a quasi-Banach function
space and hence a BQBF space. If $X(\cx)=L^\Phi(\cx)$,
then $WX(\cx)$ as in Definition \ref{dewb}
becomes the \emph{weak Orlicz space $WL^\Phi(\cx)$}.
Borrowing some ideas from the proof of
\cite[Theorem 1.3.1]{kk91},
we obtain the following Fefferman--Stein vector-valued
maximal inequality.
\begin{proposition}\label{proz}
Let $s\in(1,\fz)$ and $\Phi$ be an Orlicz function
of both lower type
$p^-_\Phi\in[1,\fz)$ and upper type
$p^+_\Phi\in[p^-_\Phi,\fz)$. Then there
exists a positive constant $C$ such that, for
any sequence $\{f_j\}_{j\in\nn}\subset \mathscr{M}(\cx)$,
\begin{equation*}
\lf\|\lf\{\sum_{j\in\nn}\lf[\cm\lf(f_j\r)\r]^s\r\}
^{1/s}\r\|_{WL^\Phi(\cx)}
\le C\lf\|\lf(\sum_{j\in\nn}\lf|f_j\r|^s\r)
^{1/s}\r\|_{L^\Phi(\cx)}.
\end{equation*}
\end{proposition}

\begin{proof}
Let $s$, $\Phi$, $p^-_\Phi$, and $p^+_\Phi$
be as in the present proposition.
For any $f:=\{f_j\}_{j\in\nn}\subset\mathscr{M}(\cx)$
and $x\in\cx$, we then write
$\cm(f)(x):=\{\cm(f_j)(x)\}_{j\in\nn}$
and $\|f(x)\|_{l^s}:=[\sum_{j\in\nn}|f_j(x)|^s]^{1/s}$
for the simplicity of the representation.

Next, we claim that, for any $\az\in(0,\fz)$ and
$f:=\{f_j\}_{j\in\nn}\subset\mathscr{M}(\cx)$,
\begin{equation}\label{eqo0}
\Phi(\az)\,\mu\lf(\lf\{x\in\cx:\ \lf\|
\cm(f)(x)\r\|_{l^s}>\az\r\}\r)
\ls\int_{\cx}\Phi\lf(\lf\|f(x)\r\|_{l^s}\r)\,d\mu(x).
\end{equation}
To this end, let
$$
_\az f:=\lf\{f_j\mathbf{1}_{\{
x\in\cx:\ \|f(x)\|_{l^s}>\az\}}\r\}_{j\in\nn}
\quad\mathrm{and}\quad
^\az f:=\lf\{f_j\mathbf{1}_{\{
x\in\cx:\ \|f(x)\|_{l^s}\le \az\}}
\r\}_{j\in\nn}.
$$
By the Minkowski inequality and the
sublinearity of $\cm$, we find that
\begin{align}\label{eqo1}
&\Phi(\az)\,\mu\lf(\lf\{x\in\cx:\ \lf\|
\cm(f)(x)\r\|_{l^s}>\az\r\}\r)\\
&\quad\le \Phi(\az)\,\mu\lf(\lf\{x\in\cx:\ \lf\|\cm(
_\az f)(x)\r\|_{l^s}>\az/2\r\}\r)\notag\\
&\quad\quad+\Phi(\az)\,
\mu\lf(\lf\{x\in\cx:\ \lf\|
\cm(^\az f)(x)\r\|_{l^s}>\az/2\r\}\r)\notag\\
&\quad=:\mathrm{I}_1+\mathrm{I}_2.\notag
\end{align}
Using the Fefferman--Stein vector-valued
maximal inequality from $L^{p^-_\Phi}(\cx)$
to $WL^{p^-_\Phi}(\cx)$ (see, for instance,
\cite[Theorem 1.2]{gly09}), the definition of $_\az f$, and
 the lower type $p^-_\Phi$ property of $\Phi$, we obtain
\begin{align*}
\mathrm{I}_1
&\ls \frac{\Phi(\az)}{\az^{p^-_\Phi}}
\int_{\cx}\lf[\lf\| _\az f(x)\r\|_{l^s}
\r]^{p^-_\Phi}\,d\mu(x)\\
&\sim\frac{\Phi(\az)}{\az^{p^-_\Phi}}\int_{\cx}
\lf[\lf\|f(x)\r\|_{l^s}\r]^{p^-_\Phi}\mathbf{1}_{\{
x\in\cx:\ \|f(x)\|_{l^s}>\az\}}(x)\,d\mu(x)\\
&\ls \int_{\cx}\Phi\lf(
\lf\|f(x)\r\|_{l^s}\r)\,d\mu(x)
\end{align*}
with the implicit positive constants independent
of both $\az$ and $f$.
By the Fefferman--Stein vector-valued
maximal inequality
from $L^{p^+_\Phi}(\cx)$
to $WL^{p^+_\Phi}(\cx)$ (see, for instance,
\cite[Theorem 1.2]{gly09}), the definition of $^\az f$,
and the upper type $p^+_\Phi$ property of $\Phi$,
we conclude that
\begin{align*}
\mathrm{I}_2
&\ls \frac{\Phi(\az)}{\az^{p^+_\Phi}}
\int_{\cx}\lf[\lf\| ^\az f(x)\r\|_{l^s}
\r]^{p^+_\Phi}\,d\mu(x)\\
&\sim\frac{\Phi(\az)}{\az^{p^+_\Phi}}\int_{\cx}
\lf[\lf\|f(x)\r\|_{l^s}\r]^{p^+_\Phi}\mathbf{1}_{\{
x\in\cx:\ \|f(x)\|_{l^s}\le\az\}}(x)\,d\mu(x)\\
&\ls \int_{\cx}\Phi\lf(
\lf\|f(x)\r\|_{l^s}\r)\,d\mu(x)
\end{align*}
with all the implicit positive constants
independent of both $\az$ and $f$.
Combining $\eqref{eqo1}$ and
estimates of $\mathrm{I}_1$ and $\mathrm{I}_2$,
we obtain \eqref{eqo0}.

By \eqref{eqo0} with $f$ and $\az$
replaced, respectively,
by $f/\lz$ and $\az/\lz$, we
conclude that there exists a positive constant $A$
such that, for any $\az\in(0,\fz)$,
$f:=\{f_j\}_{j\in\nn}\subset\mathscr{M}(\cx)$, and
$\lz\in(\|\,\|f\|_{l^s}\|_{L^\Phi(\cx)},\fz)$,
\begin{align*}
&\int_\cx\Phi\lf(\frac{\az\mathbf{1}_{\{x\in\cx: \
\|\cm(f)(x)\|_{l^s}>\az\}}(x)}{\lz}\r)\,d\mu(x)\\
&\quad=\Phi(\az/\lz)\, \,\mu\lf(\lf\{x\in\cx:\ \lf\|
\cm(f/\lz)(x)\r\|_{l^s}>\az/\lz\r\}\r)\\
&\quad\le A\int_{\cx}\Phi\lf(\frac{\|f(x)
\|_{l^s}}{\lz}\r)\,d\mu(x)\le A,
\end{align*}
where $f/\lz:=\{f_j/\lz\}_{j\in\nn}$.
Letting $\lz\to \|\,\|f\|_{l^s}\|_{L^\Phi(\cx)}$,
using this and Remark \ref{reor}(iii) with $f$ replaced by
$\az\mathbf{1}_{\{x\in\cx: \
\|\cm(f)(x)\|_{l^s}>\az\}}$,
we find that, for any $\az\in(0,\fz)$ and
$f:=\{f_j\}_{j\in\nn}\subset\mathscr{M}(\cx)$,
\begin{equation*}
\lf\|\az\mathbf{1}_{\{x\in\cx: \
\|\cm(f)(x)\|_{l^s}>\az\}}\r\|_{L^\Phi(\cx)}
\ls \lf\|\lf\|f\r\|_{l^s}\r\|_{L^\Phi(\cx)},
\end{equation*}
which, together with the definition of
$WL^\Phi(\cx)$, implies that, for any
$f:=\{f_j\}_{j\in\nn}\subset\mathscr{M}(\cx)$,
$$
\lf\|\lf\|\cm(f)\r\|_{l^s}\r\|_{WL^\Phi(\cx)}
\ls \lf\|\lf\|f\r\|_{l^s}\r\|_{L^\Phi(\cx)}.
$$
This finishes the proof of Proposition \ref{proz}.
\end{proof}

Now, we apply all the main theorems obtained
in the previous sections to Orlicz spaces.

\begin{theorem}\label{thor}
Let $\Phi$ be as in Definition \ref{deos}
with $p^-_\Phi\in(0,\fz)$
and $p^+_\Phi\in[p^-_\Phi,\fz)$.
\begin{enumerate}
\item[\textup{(i)}]
If $p^-_\Phi\in(1,\fz)$, then the conclusion
of Theorem \ref{thmw} holds true
with $X(\cx)$ replaced by $L^\Phi(\cx)$.
\item[\textup{(ii)}]
If $p^-_\Phi\in(\oz/(\oz+\eta),\fz)$ with $\oz$ as in
\eqref{eqoz} and $\eta$ as in \eqref{eqwa},
then the conclusions of Theorems
\ref{thm}, \ref{tham}, \ref{thma},
and \ref{thhin} hold true with $X(\cx)$
replaced by $L^\Phi(\cx)$.
\item[\textup{(iii)}]
If $\ez\in (0,\eta)$ and $p^-_\Phi=\oz/(\oz+\ez)$,
then the conclusion of Theorem \ref{thcz}
holds true with $X(\cx)$ replaced
by $L^\Phi(\cx)$.
\end{enumerate}
\end{theorem}

\begin{proof}
Let $\Phi$, $p^-_\Phi$, and $p^+_\Phi$ be
as in the present theorem.
To prove the present theorem, it suffices to show that
the Orlicz space satisfies all the assumptions needed in
Theorems \ref{thmw}, \ref{thm}, \ref{tham},
\ref{thma}, \ref{thhin}, and \ref{thcz}, which then further
implies the desired conclusions.

Let $p^-_\Phi\in(\oz/(\oz+\eta),\fz)$,
$t\in(0,p^-_\Phi)$, and $s\in(1,\fz)$.
By Remark \ref{reor}(ii) and
\cite[Theorem 6.6]{fmy20} (see also
\cite[Theorem 2.10]{lhy12} for the Euclidean space case),
we find that, for any
$\{f_j\}_{j\in\nn}\subset\mathscr{M}(\cx)$,
$$
\int_\cx\Phi_{1/t}\lf(\lf\{\sum_{j\in\nn}
\lf[\cm\lf(f_j\r)(x)\r]^s\r\}^{1/s}\r)\,d\mu(x)
\ls\int_\cx\Phi_{1/t}\lf(\lf[\sum_{j\in\nn}
\lf|f_j(x)\r|^s\r]^{1/s}\r)\,d\mu(x).
$$
Using this with $\{f_j\}_{j\in\nn}$ replaced by
$\{f_j/\lz\}_{j\in\nn}$,
we conclude that there exists a positive constant
$A$ such that, for any
$\{f_j\}_{j\in\nn}\subset\mathscr{M}(\cx)$
and $\lz\in(\|(\sum_{j\in\nn}
|f_j|^s)^{1/s}\|_{L^{\Phi_{1/t}}(\cx)},\fz)$,
$$
\int_\cx\Phi_{1/t}\lf(\frac{\{\sum_{j\in\nn}
[\cm(f_j)(x)]^s\}^{1/s}}{\lz}\r)\,d\mu(x)
\le A\int_\cx\Phi_{1/t}\lf(\frac{[\sum_{j\in\nn}
|f_j(x)|^s]^{1/s}}{\lz}\r)\,d\mu(x)\le A.
$$
Letting $\lz\to \|(\sum_{j\in\nn}
|f_j|^s)^{1/s}\|_{L^{\Phi_{1/t}}(\cx)}$,
from Remark \ref{reor}(iii),
we deduce that, for any
$\{f_j\}_{j\in\nn}\subset\mathscr{M}(\cx)$,
$$
\lf\|\lf\{\sum_{j\in\nn}\lf[\cm\lf(f_j\r)
\r]^s\r\}^{1/s}\r\|_{L^{\Phi_{1/t}}(\cx)}
\ls \lf\|\lf(\sum_{j\in\nn}
\lf|f_j\r|^s\r)^{1/s}\r\|_{L^{\Phi_{1/t}}(\cx)},
$$
which implies that Assumption \ref{asfs}
holds true with $X(\cx)=L^\Phi(\cx)$
and $p_-=p^-_\Phi$.

Let $p^-_\Phi\in(1,\fz)$. By Remark \ref{rem}
with $X(\cx)$ replaced by $L^\Phi(\cx)$,
we find that there exists a $t\in(1,\fz)$ such
that $\cm$ is bounded on $WL^{\Phi_{1/t}}(\cx)$.

Let $p^-_\Phi\in(\oz/(\oz+\eta),\fz)$,
$s_0\in (\oz/(\oz+\eta),\min\{1,p^-_\Phi\})$,
and $p_0\in (p^+_\Phi,\fz)$. Using Remark \ref{reor}(ii)
and \cite[Lemma 2.16]{zyyw19}, without loss of
generality, we may regard
 $\Phi_{1/s_0}$ as a continuous convex Orlicz
function of both lower type
$p^-_\Phi/s_0$ and upper type $p^+_\Phi/s_0$.
Observe that $L^{\Phi_{1/s_0}}(\cx)$ is a
BBF space; see, for instance, \cite[Remark 4.21(iv)]{yhyy21a}.
By \cite[p.\,61, Proposition 4 and p.\,100,
Proposition 1]{rr91}, we conclude that the associate
space of $L^{\Phi_{1/s_0}}(\cx)$
is $L^\Psi(\cx)$,
where, for any $t\in (0,\fz)$,
$$\Psi(t):=\sup_{u\in(0,\fz)}
\lf\{tu-\Phi_{1/s_0}(u)\r\}.$$
From \cite[Proposition 7.8]{shyy17}, we deduce
that $\Psi$ is of lower type $(p^+_\Phi/s_0)'$.
By this and $p_0\in(p^+_\Phi,\fz)$, we find that
$\cm$ is bounded on the $\frac{1}{(p_0/s_0)'}$-convexification
of $L^\Psi(\cx)$.
Thus, Assumption \ref{asas} holds true with
$X(\cx)$ replaced by $L^\Phi(\cx)$.

Let $\ez\in(0,\eta)$ and $p^-_\Phi=\oz/(\oz+\ez)$.
By the dominated convergence theorem and Remark \ref{thor}(i),
we find that $L^\Phi(\cx)$ has an
absolutely continuous quasi-norm.
From Proposition \ref{proz}, we deduce that
Assumption \ref{asw} holds true when
$X(\cx)=L^\Phi(\cx)$ and $\tau=\ez$.

This finishes the proof
of Theorem \ref{thor}.
\end{proof}

\begin{remark}
To the best of our knowledge, Theorem \ref{thor} is new
even when $\cx$ is an RD-space.
\end{remark}

\subsection{Variable Lebesgue Spaces}

We first recall the concept of variable Lebesgue spaces
and refer the reader to \cite{cf13, dhhr11} for more details.
A $\mu$-measurable function $p(\cdot):\cx\to(0,\fz)$
is called a \emph{variable exponent}. For any
variable exponent $p(\cdot)$, let
$p^-$ denote its essential infimum
and $p^+$ its essential supremum.

\begin{definition}\label{deva}
Let $p(\cdot)$ be a variable exponent
with $0<p^-\le p^+<\fz$.
The \emph{variable Lebesgue space $L^{p(\cdot)}(\cx)$}
is defined to be the set of all the
functions $f\in\mathscr{M}(\cx)$ such that
\begin{equation*}
\|f\|_{L^{p(\cdot)}(\cx)}:=\inf
\lf\{\lz\in (0,\fz):\ \int_{\cx}
\lf[\frac{|f(x)|}{\lz}\r]^{p(x)}\,d\mu(x)
\le 1\r\}<\fz.
\end{equation*}
\end{definition}

Let $p(\cdot)$ be as in Definition \ref{deva}.
Observe that $L^{p(\cdot)}(\cx)$ is a
quasi-Banach function space and
hence a BQBF space; see, for instance,
\cite[Remark 2.7(iv)]{yhyy21a}.
If $X(\cx)=L^{p(\cdot)}(\cx)$, then
$WX(\cx)$ as in Definition \ref{dewb}
becomes the \emph{variable weak Lebesgue space
$WL^{p(\cdot)}(\cx)$}.

Let $x_0\in\cx$ be the fixed basepoint.
Recall that a variable exponent $p(\cdot)$
is said to be \emph{globally log-H\"{o}lder
continuous} if there exist constants
$C\in (0,\fz)$ and $p_\fz\in \rr$
such that, for any $x$, $y\in\cx$,
$$
|p(x)-p(y)|\le C\frac{1}{\log(e+1/d(x,y))}
\quad \mathrm{and}\quad
|p(x)-p_\fz|\le C\frac{1}{\log(e+d(x,x_0))}.
$$

Borrowing some ideas from the
proofs of both \cite[Theorem 5.24]{cf13}
(the extrapolation method)
and \cite[Proposition 7.8]{yyyz16}, we obtain
the following proposition.

\begin{proposition}\label{prv}
Let $s\in(1,\fz)$ and $p(\cdot)$ be a globally log-H\"{o}lder
continuous variable exponent with $1\le p^-\le p^+<\fz$.
Then there
exists a positive constant $C$ such that, for
any sequence $\{f_j\}_{j\in\nn}\subset \mathscr{M}(\cx)$,
\begin{equation*}
\lf\|\lf\{\sum_{j\in\nn}\lf[\cm\lf(f_j\r)\r]^s\r\}
^{1/s}\r\|_{WL^{p(\cdot)}(\cx)}
\le C\lf\|\lf(\sum_{j\in\nn}\lf|f_j\r|^s\r)
^{1/s}\r\|_{L^{p(\cdot)}(\cx)}.
\end{equation*}
\end{proposition}

\begin{proof}
Let all the symbols be as in the present proposition.
For any $x\in\cx$, let $p'(x)$ denote the conjugate index
of $p(x)$; namely, $1/p(x)+1/p'(x)=1$.
For any $f:=\{f_j\}_{j\in\nn}\subset\mathscr{M}(\cx)$
and $x\in\cx$, we write
$\cm(f)(x):=\{\cm(f_j)(x)\}_{j\in\nn}$
and $\|f(x)\|_{l^s}:=[\sum_{j\in\nn}|f_j(x)|^s]^{1/s}$
for the simplicity of the representation.

By \cite[Lemma 2.9]{zsy16} (see also \cite[Theorem 9.2]{ins14}
for the Euclidean space case), we find that,
for any $\az\in(0,\fz)$ and
$f:=\{f_j\}_{j\in\nn}\subset\mathscr{M}(\cx)$,
\begin{equation}\label{eqv0}
\lf\|\az\mathbf{1}_{\{x\in\cx:\
\|\cm(f)(x)\|_{l^s}>\az\}}
\r\|_{L^{p(\cdot)}(\cx)}
\sim \sup_{h\in\Delta}\lf\{
\int_\cx \az\mathbf{1}_{\{x\in\cx:\
\|\cm(f)(x)\|_{l^s}>\az\}}(x)\lf|h(x)\r|\,d\mu(x)\r\}
\end{equation}
with the positive equivalence constants independent
of $\az$ and $f$, where
$$
\Delta:=\lf\{h\in L^{p'(\cdot)}(\cx):\
\|h\|_{L^{p'(\cdot)}(\cx)}=1\r\}.
$$

Next, we use the extrapolation method.
By \cite[Theorem 1.1]{cs18}, we find that $\cm$
is bounded on $L^{p'(\cdot)}(\cx)$ with its operator
norm denoted by
$\|\cm\|_{L^{p'(\cdot)}(\cx)\to L^{p'(\cdot)}(\cx)}$.
For any $h\in L^{p'(\cdot)}(\cx)$ and $x\in\cx$, let
$$
\mathcal{L}(h)(x):=|h(x)|+\sum^\fz_{k=1}\frac{\cm^k(h)(x)}
{2^k\|\cm\|^k_{L^{p'(\cdot)}(\cx)\to L^{p'(\cdot)}(\cx)}},
$$
where, for any $k\in\nn$,
$\cm^k$ is the $k$ iterations of $\cm$. Observe that
the following facts hold true:
\begin{itemize}
\item[\textup{(i)}]
for any $x\in\cx$, $|h(x)|\le \mathcal{L}(h)(x)$;
\item[\textup{(ii)}]
$\mathcal{L}$ is bounded on $L^{p'(\cdot)}(\cx)$
and $\|\mathcal{L}(h)\|_{L^{p'(\cdot)}(\cx)}
\le 2\|h\|_{L^{p'(\cdot)}(\cx)}$;
\item[\textup{(iii)}]
$\mathcal{L}(h)$ is an $A_1(\cx)$-weight.
\end{itemize}

By \eqref{eqv0}, (i), (iii),
the weighted Fefferman--Stein vector-valued
maximal inequality (see, for instance, \cite[Theorem 6.5]{fmy20}),
the H\"{o}lder inequality for variable Lebesgue
spaces (see, for instance,
\cite[Remark 2.3(iii)]{zsy16}), and (ii), we conclude that,
for any $\az\in(0,\fz)$ and
$f:=\{f_j\}_{j\in\nn}\subset\mathscr{M}(\cx)$,
\begin{align*}
&\lf\|\az\mathbf{1}_{\{x\in\cx:\
\|\cm(f)(x)\|_{l^s}>\az\}}
\r\|_{L^{p(\cdot)}(\cx)}\\
&\quad\ls \sup_{h\in \Delta}\lf[
\int_\cx \az\mathbf{1}_{\{x\in\cx:\
\|\cm(f)(x)\|_{l^s}>\az\}}(x)
\mathcal{L}(h)(x)\,d\mu(x)\r]\\
&\quad\ls \sup_{h\in \Delta}\lf\{
\int_\cx \|f(x)\|_{l^s}
\mathcal{L}(h)(x)\,d\mu(x)\r\}\\
&\quad\ls \lf\|\lf\|f\r\|_{l^s}\r\|_{L^{p(\cdot)}(\cx)}
\sup_{h\in \Delta}\lf\{
\lf\|\mathcal{L}(h)\r\|_{L^{p'(\cdot)}(\cx)}\r\}
\ls \lf\|\lf\|f\r\|_{l^s}\r\|_{L^{p(\cdot)}(\cx)},
\end{align*}
where all the implicit positive constants are
independent of both $\az$ and $f$.
Combining this and the definition of $WL^{p(\cdot)}(\cx)$,
we find that, for any
$f:=\{f_j\}_{j\in\nn}\subset\mathscr{M}(\cx)$,
$$
\lf\|\lf\|\cm(f)\r\|_{l^s}\r\|_{WL^{p(\cdot)}(\cx)}
\ls \lf\|\lf\|f\r\|_{l^s}\r\|_{L^{p(\cdot)}(\cx)}.
$$
This finishes the proof of Proposition \ref{prv}.
\end{proof}

Now, we apply the main results obtained in the
previous sections to variable Lebesgue spaces as follows.

\begin{theorem}\label{thva}
Let $p(\cdot)$ be a globally log-H\"{o}lder
continuous variable exponent with $0< p^-\le p^+<\fz$.
\begin{enumerate}
\item[\textup{(i)}]
If $p^-\in(1,\fz)$, then the conclusion
of Theorem \ref{thmw} holds true
with $X(\cx)$ replaced by $L^{p(\cdot)}(\cx)$.
\item[\textup{(ii)}]
If $p^-\in(\oz/(\oz+\eta),\fz)$ with $\oz$ as in
\eqref{eqoz} and $\eta$ as in \eqref{eqwa},
then the conclusions of Theorems \ref{thm},
\ref{tham}, \ref{thma},
and \ref{thhin} hold true with $X(\cx)$
replaced by $L^{p(\cdot)}(\cx)$.
\item[\textup{(iii)}]
If $\ez\in (0,\eta)$ and $p^-=\oz/(\oz+\ez)$,
then the conclusion of Theorem \ref{thcz}
holds true with $X(\cx)$ replaced
by $L^{p(\cdot)}(\cx)$.
\end{enumerate}
\end{theorem}

\begin{proof}
Let $p(\cdot)$, $p^-$, and $p^+$ be as in the present theorem.
To prove this theorem, we only need to
show that $L^{p(\cdot)}(\cx)$
satisfies all the assumptions needed in
Theorems \ref{thmw}, \ref{thm}, \ref{tham},
\ref{thma}, \ref{thhin}, and \ref{thcz}, which then
further implies the desired conclusion.

Let $p^-\in(\oz/(\oz+\eta),\fz)$.
From \cite[Theorem 2.7]{zsy16}, it follows
that Assumption \ref{asfs} holds true with
$X(\cx)$ and $p_-$ replaced, respectively, by
$L^{p(\cdot)}(\cx)$ and $p^-$.

Let $p^-\in(1,\fz)$.
By Remark \ref{rem} with $X(\cx)$
replaced by $L^{p(\cdot)}(\cx)$, we
conclude that there exists a $t\in(1,\fz)$
such that $\cm$
is bounded on $WL^{p(\cdot)/t}(\cx)$.

Let $p^-\in(\oz/(\oz+\eta),\fz)$,
$s_0\in (\oz/(\oz+\eta),\min\{1,p^-\})$,
and $p_0\in (p^+,\fz)$. It is easy to show that
$L^{p(\cdot)/s_0}(\cx)$ is a BBF space;
see, for instance, \cite[Remark 4.21(v)]{yhyy21a}.
By \cite[Lemma 2.9]{zsy16}
(see also \cite[Theorem 9.2]{ins14} for the Euclidean space case), we
find that the associate space of
$L^{p(\cdot)/s_0}(\cx)$ is
$L^{(p(\cdot)/s_0)'}(\cx)$, where
$\frac{1}{(p(x)/s_0)'}+\frac{1}{p(x)/s_0}=1$
for any $x\in\cx$. Observe that the essential infimum
of $(p(\cdot)/s_0)'$ is $(p^+/s_0)'$.
Using this and \cite[Theorem 1.1]{cs18},
we find that $\cm$ is bounded on the
$\frac{1}{(p_0/s_0)'}$-convexification of
$L^{(p(\cdot)/s_0)'}(\cx)$. Thus,
Assumption \ref{asas} holds true with
$X(\cx)$ replaced by $L^{p(\cdot)}(\cx)$.

Let $\ez\in (0,\eta)$ and $p^-=\oz/(\oz+\ez)$.
By the dominated convergence theorem, we find that
$L^{p(\cdot)}(\cx)$ has an absolutely continuous quasi-norm.
From Proposition \ref{prv}, it follows that
Assumption \ref{asw} holds true
with $X(\cx)=L^{p(\cdot)}(\cx)$
and $\tau=\ez$.

This finishes the proof of Theorem \ref{thva}.
\end{proof}
\begin{remark}
To the best of our knowledge, Theorem \ref{thva}
is new even when $\cx$ is an RD-space.
\end{remark}

\bigskip

\noindent Jingsong Sun, Dachun Yang (Corresponding author) and Wen Yuan

\medskip

\noindent Laboratory of Mathematics and Complex Systems
(Ministry of Education of China),
School of Mathematical Sciences, Beijing Normal University,
Beijing 100875, People's Republic of China

\smallskip

\noindent {\it E-mails}: \texttt{jssun@mail.bnu.edu.cn} (J. Sun)

\noindent\phantom{{\it E-mails:}} \texttt{dcyang@bnu.edu.cn} (D. Yang)

\noindent\phantom{{\it E-mails:}} \texttt{wenyuan@bnu.edu.cn} (W. Yuan)

\begin{thebibliography}{99}
\bibitem{syy}
J. Sun, D. Yang and W. Yuan,
Molecular characterization of weak Hardy spaces associated
with ball quasi-Banach function spaces on spaces
of homogeneous type with its applications to Littelwood--Paley
function characterizations, Submitted.

\vspace{-0.3cm}
\end{thebibliography}
\end{document}